\documentclass[12pt]{amsart}
\usepackage{custom}

\newtheorem{construction}[theorem]{Construction}

\usepackage{pgfplots}
\usepackage{xxcolor}

\usetikzlibrary{decorations.pathreplacing}

\tikzset{
  jumpdot/.style={mark=*,solid},
  exclr/.append style={jumpdot,color=red,fill=white},
  exclb/.append style={jumpdot,color=blue,fill=white},
  incl/.append style={jumpdot},
}

\pgfdeclareradialshading[tikz@ball]{ball}{\pgfqpoint{0bp}{0bp}}{%
 color(0bp)=(tikz@ball!0!white);
 color(5bp)=(tikz@ball!97!black);
 color(10bp)=(tikz@ball!90!black);
 color(15bp)=(tikz@ball!70!black);
 color(20bp)=(black!70);
 color(30bp)=(black!70)}
\makeatother

\tikzset{viewport/.style 2 args={
    x={({cos(-#1)*1cm},{sin(-#1)*sin(#2)*1cm})},
    y={({-sin(-#1)*1cm},{cos(-#1)*sin(#2)*1cm})},
    z={(0,{cos(#2)*1cm})}
}}

\pgfplotsset{only foreground/.style={
    restrict expr to domain={rawx*\CameraX + rawy*\CameraY + rawz*\CameraZ}{-0.05:100},
}}
\pgfplotsset{only background/.style={
    restrict expr to domain={rawx*\CameraX + rawy*\CameraY + rawz*\CameraZ}{-100:0.05}
}}

\def\addFGBGplot[#1]#2;{
    \addplot3[#1,only background, opacity=0.25] #2;
    \addplot3[#1,only foreground] #2;
}

\newcommand{\ViewAzimuth}{-30}
\newcommand{\ViewElevation}{10}
\newcommand{\cvspace}{.6em} 
\newcommand{\bigslant}[2]{{\raisebox{.2em}{$#1$}\left/\raisebox{-.2em}{$#2$}\right.}}
\newcommand{\rhot}{{\tilde{\rho}}}
\newcommand{\pta}{{\tilde{p}_1}}
\newcommand{\ptb}{{\tilde{p}_2}}
\newcommand{\qta}{{\tilde{q}_1}}
\newcommand{\qtb}{{\tilde{q}_2}}
\newcommand{\tj}{\tilde{J}}
\newcommand{\ba}{\bar{a}}
\newcommand{\vt}{\boldsymbol{t}}

\usepackage[perpage]{footmisc}

\usepackage{geometry}
\usepackage{marginnote}
\DeclareMathOperator{\area}{Area}
\DeclareMathOperator{\cindex}{index}
\DeclareMathOperator{\img}{img}

\newcommand{\ho}{\textup{H}}

\newcommand{\aditcont}{1} 

\newtheorem{exmp}{Example}[section]

\usepackage{bm}
\renewcommand{\boldsymbol}[1]{\bm{#1}}
\newcommand{\spL}{\boldsymbol{l}}
\newcommand{\spR}{\boldsymbol{r}}
\newcommand{\retr}{\mathfrak{R}}
\newcommand{\rot}{\boldsymbol{R}}
\newcommand{\mI}{{\boldsymbol{I}}}
\newcommand{\mQ}{{\boldsymbol{Q}}}
\newcommand{\mP}{{\boldsymbol{P}}}
\newcommand{\cLspacea}{\bar{\mathcal{L}}_\rho(\mI,\mQ)}
\newcommand{\Lspace}{\mathcal{L}_{\rho_0}(\mQ)}
\newcommand{\cLspace}{\bar{\mathcal{L}}_{\rho_0}(\mQ)}
\newcommand{\imspace}{\mathcal{I}(\mI,\mQ)}
\newcommand{\mat}[1]{\boldsymbol{#1}}
\newcommand{\vet}[1]{\boldsymbol{#1}}

\begin{document}

\begin{abstract}
While the topology of the space of all smooth immersed curves on the $2$-sphere $\mathbb{S}^2$ that start and end at given points in given directions is well known, it is an open problem to understand the homotopy type of its subspaces consisting of the curves whose geodesic curvatures are constrained to a prescribed proper open interval. In this article we prove that, under certain circumstances for endpoints and end directions, these subspaces are not homotopically equivalent to the whole space. Moreover, we give an explicit construction of exotic generators for some homotopy and cohomology groups. It turns out that the dimensions of these generators depend on endpoints and end directions.
A version of the h-principle is used to prove these results.
\end{abstract}

\maketitle

\tableofcontents

\newpage

\begin{table}[!htb]\caption{Recurrent Notations and Terms}
\begin{center}
\begin{tabular}{r c p{10cm} }
\toprule
$\mathbb{N}$ & $\quad$ & Non-negative integers \\[5pt]
$\mathbb{S}^2$ & $\quad$ & The unit sphere of center $0$ in the Euclidean space $\mathbb{R}^3$ \\[5pt]
$e_1$, $e_2$, $e_3$ & $\quad$ & Canonical basis of $\mathbb{R}^3$ \\[5pt]
$r$, $\rho$ & $\quad$ & Positive numbers, used to denote radius\\[5pt]
$a$, $b$, $p_1$, $p_2$, $q_1$, $q_2$ & $\quad$ & Points in $\mathbb{S}^2$ \\[5pt]
$\alpha$, $\beta$, $\gamma$ & $\quad$ & Curves in $\mathbb{S}^2$ \\[5pt]
$s$, $t$ & $\quad$ & Parameters of a curve \\[5pt]
$\vt_\gamma$, $\vet{n}_\gamma$ & $\quad$ & The tangent and the normal vectors of curve $\gamma$ \\[5pt]
$\mathfrak{F}_\gamma$ & $\quad$ & Frenet frame \\[5pt]
$x$, $y$, $z$ & $\quad$ & Real numbers \\[5pt]
$i$, $j$, $k$, $l$, $n$, $m$ & $\quad$ & Usually represent an integer or a natural number \\[5pt]
$u$, $v$, $w$ & $\quad$ & Vectors of $\mathbb{R}^3$ or $\textup{T}\mathbb{S}^2$ \\[5pt]
$\mI$ & $\quad$ & Identity matrix \\[5pt]
$\mP$, $\mQ$ & $\quad$ & Matrices in $\SO$\\[5pt]
$\kappa_0$ & $\quad$ & $\kappa_0\in (0,+\infty]$ represents the curvature constraint that appears on the definition of $\Lspace$\\[5pt]
$\rho_0$ & $\quad$ & $\rho_0\in \left[0,\frac{\pi}{2}\right)$ represents the radius that appears on the definition of $\Lspace$\\[5pt]
$\Lspace$, $\mathcal{L}_{-\kappa_0}^{+\kappa_0}(\mQ)$ & $\quad$ & Space of curves with geodesic curvature in $(-\kappa_0,+\kappa_0)$ with start frame $\mI$ and end frame $\mQ$\\[5pt]
$\cLspace$, $\bar{\mathcal{L}}_{-\kappa_0}^{+\kappa_0}(\mQ)$ & $\quad$ & Space of curves with geodesic curvature in $[-\kappa_0,+\kappa_0]$ with start frame $\mI$ and end frame $\mQ$\\[5pt]
$\mathcal{I}(\mI,\mQ)$, $\mathcal{I}(\mQ)$ & $\quad$ & Space of $C^1$ immersed curves in $\mathbb{S}^2$ with start frame $\mI$ and end frame $\mQ$\\[5pt]
$\mathcal{C}$, $\mathcal{C}_0$ & $\quad$ & Subsets of $\cLspace$ \\[5pt]
$\rot_\theta(v)$ & $\quad$ & Rotation matrix of angle $\theta$ with axis $v$ \\[5pt]
$\zeta_{p,r}$ & $\quad$ & Circle on sphere with intrinsic radius $r$ and center at $p$ \\[5pt]
$\circlearrowleft$, $\circlearrowright$ & $\quad$ & Used to represent the orientation of a circle \\[5pt]
$J$, $K$ & $\quad$ & Subsets of $\mathbb{R}$, $J$ is often used to denote the domain of the map $\gamma$ \\[5pt]
$f$, $g$, $F$, $G$ & $\quad$ & Applications between spaces \\[5pt]
$B_r(v)$, $\bar{B}_r(v)$ & $\quad$ & Open ball and closed ball in $\mathbb{S}^2$ with intrinsic radius $r$ centered at $v$ \\[5pt]
$\mathcal{H}_v$ & $\quad$ & Hemisphere of $\mathbb{S}^2$ given by $\mathcal{H}_v=\{u\in\mathbb{S}^2;\langle u,v\rangle >0\}$ \\[5pt]
\bottomrule
\end{tabular}
\end{center}
\label{tab:notations}
\end{table}

\begin{table}[!htb]\caption{Recurrent Notations and Terms}
\begin{center}
\begin{tabular}{r c p{10cm} }
\toprule
$(\theta,\varphi)$ & $\quad$ & Parallel and meridian coordinates with axis in direction of a vector $v\in\mathbb{S}^2$ \\[5pt]
$d(p,q)$ & $\quad$ & When $p,q\in\mathbb{S}^2$, it denotes the distance from $p$ to $q$ measured on $\mathbb{S}^2$ \\[5pt]
$\exp$ & $\quad$ & The matrix exponential on anti-symmetric matrices $\text{so}_3(\mathbb{R})=\textup{T}_\mI\SO$ or the Exponential map on the sphere $\mathbb{S}^2$\\[5pt]
$\img$ & $\quad$ & Image of an application\\[5pt]
$ ^*$, $ ^\dagger$, $ ^\ddagger$ & $\quad$ & Footnote marker symbols\\[5pt]
curvature of a curve in $\mathbb{S}^2$ & $\quad$ & Means the \emph{geodesic} curvature of the curve in $\mathbb{S}^2$ \\[5pt]
\bottomrule
\end{tabular}
\end{center}
\label{tab:notations2}
\end{table}

\newpage
\section{Introduction}

This section is an overview of the background and the history of the problem which we study in this article. We also present some related topics. 

\subsection{Topology of the space of curves in $\mathbb{S}^2$}


\noindent The topology of the space of curves on differential manifolds is a very interesting topic for research. 
Let $\mathbb{S}^n$ denote the standard $n$-dimensional unit sphere in the $(n+1)$-dimensional Euclidean space $\mathbb{R}^{n+1}$. Here we introduce previous works for the case of immersed regular curves on the two dimensional unit sphere $\mathbb{S}^2$ in the $3$-dimensional Euclidean space $\mathbb{R}^3$. In 1956, S. Smale proved that the space of $C^r$ ($r\geq 1$) immersions $\mathbb{S}^1\to\mathbb{S}^2$, i.e., $C^r$ regular closed curves on $\mathbb{S}^2$, has only two connected components. Each of them are homotopically equivalent to $\SO\times\Omega\mathbb{S}^3$, where $\Omega\mathbb{S}^3$ denotes the space of all continuous closed curves in $\mathbb{S}^3$ with the $C^0$ topology. This result is a consequence of a much more general theorem (\cite{smale}, thm. A) by him. The properties of $\Omega\mathbb{S}^3$ are well understood (see \cite{botttu}, \textsection 16 and \cite{milnor} \textsection 17). Later in 1970, J. A. Little proved the following theorem.

\begin{theorem}[J. A. Little \cite{little}]\label{little}
There are exactly $6$ second order non-degenerate\footnote{We call a closed curve in $\mathbb{S}^2$ second order non-degenerate when its geodesic curvature is continuous and different from $0$. A regular homotopy of curves on $\mathbb{S}^2$, $h:\mathbb{S}^1\times [0,1]\to\mathbb{S}^2$, is called nondegenerate if each curve $h_t:\mathbb{S}^1\to\mathbb{S}^2$, $t\in [0,1]$, is nondegenerate and if the geodesic curvature is continuous on $\mathbb{S}^1\times [0,1]$.} regular homotopy classes of closed curves on $\mathbb{S}^2$. Moreover, the following $6$ curves on the sphere, denoted by $\gamma_j:[0,1]\to\mathbb{S}^2$, for $j\in\{-3, -2, -1, 1, 2, 3\}$ are in different non-degenerate homotopy classes (see the Figure \ref{fig:little}):
$$
\gamma_j(t)\coloneqq \frac{\sqrt{2}}{2}(1,0,0)+\frac{\sqrt{2}}{2}\big[\sin(2j\pi t)\cdot(0,1,0)-\cos(2j\pi t)\cdot(0,0,1)\big].
$$
\end{theorem} 

 In other words, there are a total of $6$ connected components in the space of non-degenerate curves in $\mathbb{S}^2$. Each one contains exactly one of the curves $\gamma_j$ given above. The components that contain $\gamma_{\pm 1}$ are known to be contractible.

\begin{figure}[htb] \label{fig:little}
\centering
\begin{tikzpicture}
\begin{scope}[xshift = -4cm, decoration={markings, mark=at position 0.13 with {\arrow {Latex}}, mark=at position 0.63 with {\arrow {Latex}}}]
\draw [postaction={decorate}] plot [smooth cycle, tension = 1.2] coordinates {(0.7,0.2) (0.65,1.35) (-0.65,1.35) (-0.7,0.2)};
\draw[thick,dashed] (0,0.8) circle (1.7cm);
\end{scope}
\begin{scope}[decoration={markings, mark=at position 0.1 with {\arrow {Latex}}, mark=at position 0.6 with {\arrow {Latex}}}]
\draw [postaction={decorate}] plot [smooth cycle, tension = 1.2] coordinates {(0.5,0) (0.75,1.3) (-0.75,1.3) (-0.5,0) (0.55,0.55) (0,1.3) (-0.55,0.55)};
\draw[thick,dashed] (0,0.8) circle (1.7cm);
\end{scope}
\begin{scope}[xshift = 4cm, decoration={markings, mark=at position 0.08 with {\arrow {Latex}}, mark=at position 0.43 with {\arrow {Latex}}, mark=at position 0.58 with {\arrow {Latex}}, mark=at position 0.74 with {\arrow {Latex}}}]
\draw [postaction={decorate}] plot [smooth cycle, tension = 1.2] coordinates {(0.8,0) (0.75,1.3) (-0.75,1.3) (-0.8,0) (-0.10,0.35) (-0.4,1.0) (-0.70,0.35) (0,-0.2) (0.70,0.35) (0.4,1.0) (0.10,0.35)};
\draw[thick,dashed] (0,0.8) circle (1.7cm);
\end{scope}
\end{tikzpicture}
\caption{Figure above illustrates three different curves on a hemisphere of $\mathbb{S}^2$ with positive geodesic curvature. These three curves, from left to right, lie in different connected components containing $\gamma_1$, $\gamma_2$ and $\gamma_3$, respectively.}
\end{figure}
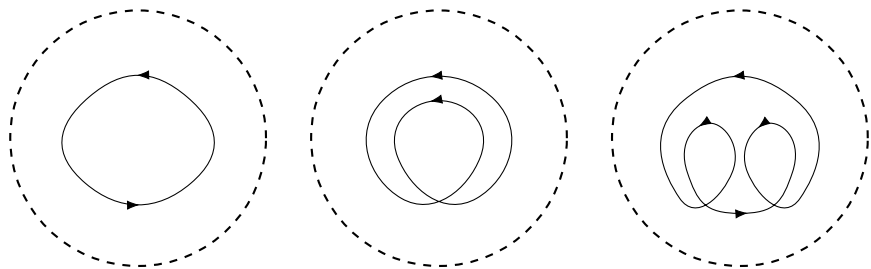

By reflecting each curve in $\mathbb{S}^2$ across a plane passing through the origin, we see that it defines a homeomorphism from each component of the set of curves with positive geodesic curvature into each component of the set of curves with negative geodesic curvature. Thus the topologies of the connected component that contains $\gamma_j$ and the connected component that contains $\gamma_{-j}$ are exactly the same for $j = 1,2,3$. So, to fully understand the topology of the set of non-degenerate curves, it is enough to understand the topology of the set of curves with positive geodesic curvature. We also call a curve \emph{locally convex} if its geodesic curvature is always positive. 

In 1999, B. Z. Shapiro and B. A. Khesin \cite{shakhe} studied the topology of the space of all smooth immersed locally convex curves (not necessarily closed) in $\mathbb{S}^2$ which start and end at given points and given directions. They found the number of connected components of this space. 
More precisely, 
\begin{theorem}[B. Z. Shapiro, B. A. Khesin]\label{shakhe}
The space of locally convex curves on $\mathbb{S}^2$ with given initial and final frames consists of $3$ connected components if there exists a disconjugate curve connecting them. Otherwise the space consists of $2$ connected components. 

Here a curve $\gamma:[0,1]\to\mathbb{S}^2$ is called \emph{conjugate} if there exists a great circle on $\mathbb{S}^2$ having at least $3$ transversal intersections with $\gamma$. Otherwise it is called \emph{disconjugate}.
\end{theorem}
Their work extends Theorem \ref{little} by Little because the space of closed curves is a particular case in which initial and final points and directions coincide respectively.

During 2009-2012, in \cite{sald1}, \cite{sald2} and \cite{sald}, N. C. Saldanha did several further works on the higher homotopy properties of the space of locally convex curves on $\mathbb{S}^2$. 
More precisely, he proved the following result:

\begin{theorem}[N. C. Saldanha]\label{sald} Under the same notations of Theorem \ref{little}, 
the component that contains the curve $\gamma_2$ is homotopically equivalent to $(\Omega\mathbb{S}^3)\vee\mathbb{S}^2\vee\mathbb{S}^6\vee\mathbb{S}^{10}\cdots$. The component that contains the curve $\gamma_3$ is homotopically equivalent to $(\Omega\mathbb{S}^3)\vee\mathbb{S}^4\vee\mathbb{S}^8\vee\mathbb{S}^{12}\cdots$.
\end{theorem}

\noindent Moreover N. C. Saldanha gave an explicit homotopy for space of curves with prescribed initial and final Frenet frames: 

\begin{theorem}[N. C. Saldanha]\label{sald2} The space of locally convex curves on $\mathbb{S}^2$ with prescribed initial and final frames consists of connected components of the following types, which depend on its lifted Frenet frame $\bsm{z} \in \mathbb{S}^3$ with basepoint $\mathbf{1}$\footnote{Here we are viewing $\mathbb{S}^3$ as the subset of Quaternions, $\mathbf{1}$ denotes the identity of multiplication of Quaternions.}:
\begin{itemize}
\item $(\Omega\mathbb{S}^3)\vee\mathbb{S}^0\vee\mathbb{S}^4\vee\mathbb{S}^8\vee\mathbb{S}^{12}\vee\cdots$ if $\bsm{z}$ is convex;
\item $(\Omega\mathbb{S}^3)\vee\mathbb{S}^2\vee\mathbb{S}^6\vee\mathbb{S}^{10}\vee\mathbb{S}^{14}\vee\cdots$ if $-\bsm{z}$ is convex;
\item $\Omega\mathbb{S}^3$ if neither $\bsm{z}$ nor $-\bsm{z}$ is convex.
\end{itemize}
\end{theorem}

\noindent For the precise definition of \emph{convexity} of $\bsm{z}\in\mathbb{S}^3$, refer \cite{sald} p.3-4. Despite the omission of an apparent complexity in the hypothesis, Theorem \ref{sald2} is a more general version of Theorem \ref{sald}. 

Recently, in 2013, N. C. Saldanha and P. Zühlke \cite{salzuh} extended Little's result to the space of closed curves with geodesic curvature constrained in an open interval:
\begin{theorem}\label{saldzuhl}
Let $\kappa_1,\kappa_2$ be extended real numbers: $-\infty \leq \kappa_1 < \kappa_2 \leq +\infty$, and let $\rho_i = \arccot \kappa_i $ for $i=1,2$\footnote{We use the conventional function $\arccot:\mathbb{R}\to\left(0,\pi\right)$, we put $\arccot (+\infty) = 0$ and $\arccot (-\infty)= \pi$, extending it to $[-\infty,+\infty]$.}. Let
$$ n=\left\lfloor \frac{\pi}{\rho_1-\rho_2}\right\rfloor +1 .$$

\noindent Then the space $\mathcal{L}_{\kappa_1}^{\kappa_2}$ of closed curves on $\mathbb{S}^2$ with geodesic curvature in the interval $(\kappa_1,\kappa_2)$ has exactly $n$ connected components $\mathcal{L}_1,\ldots\mathcal{L}_n$. Denote by $\gamma_j$ the circle traversed $j$ times described by the formula below:
$$\gamma_j = \frac{\sqrt{2}}{2}(1,0,0)+\frac{\sqrt{2}}{2}\big[\sin(2j\pi t)(0,1,0)-\cos(2j\pi t)(0,0,1)\big].$$
For each $j\in\{1,2,\ldots,n\}$, the component $\mathcal{L}_j$ contains the curve $\gamma_j:[0,1]\to\mathbb{S}^2$. 

The component $\mathcal{L}_{n-1}$ also contains $\gamma_{(n-1)+2k}$ for all $k\in\mathbb{N}$, and $\mathcal{L}_n$ also contains $\gamma_{n+2k}$ for all $k\in\mathbb{N}$. Moreover, for $n \geq 3$, each of $\mathcal{L}_1,\ldots ,\mathcal{L}_{n-2}$ is homeomorphic to the space $\SO \times \mathbb{E}$, where $\mathbb{E}$ is the separable Hilbert space.
\end{theorem}

However, at the moment, not much is known about the higher homotopy structures of the spaces $\mathcal{L}_{n-1}$ and $\mathcal{L}_n$ which appear in Theorem \ref{saldzuhl}, except for the cases in which $\rho_1-\rho_2=\frac{\pi}{2}$ (see Theorem \ref{sald}) and $\rho_1-\rho_2=\pi$, respectively. In the former case, the space $\mathcal{L}_{\kappa_1}^{\kappa_2}$ is homeomorphic to the space of locally convex curves, so its components have homotopy types as described in Theorem \ref{sald}. In the latter case, $\mathcal{L}_{\kappa_1}^{\kappa_2}=\mathcal{L}_{-\infty}^{+\infty}$ is the space of $C^1$ immersed curves without restrictions on the curvature, which is mentioned in the beginning of Introduction. 
Based on the above fact, they conjectured the connected components $\mathcal{L}_{n-1}$ and $\mathcal{L}_{n}$ to be homotopically equivalent to $(\Omega\mathbb{S}^3) \vee \mathbb{S}^{n_1}\vee \mathbb{S}^{n_2}\vee\mathbb{S}^{n_3}\vee \cdots$. On the other hand, it is unknown whether the result of Theorem \ref{saldzuhl} holds for the space of curves that are not closed, except for the case $\rho_1-\rho_2 = \frac{\pi}{2}$ (see Theorem \ref{sald2}).

It is worth to mention that in 2014, N. C. Saldanha and P. Zühlke \cite{salzuh2} solved the related problem for the space of curves in the plane $\mathbb{R}^2$ with prescribed initial and final Frenet frames. 

In this article we consider the space $\mathcal{L}_{\kappa_1}^{\kappa_2}(P,Q)$ of curves in $\mathbb{S}^2$ with prescribed initial and final Frenet frames $P,Q\in\SO$ respectively, and obtain a result consistent with the conjecture by proving the existence of a non-trivial map $F:\mathbb{S}^{n_1}\to\mathcal{L}_{\kappa_1}^{\kappa_2}(P,Q)$. It turned out that the existence of such map and the dimension $n_1$ are linked to the maximum number of arcs of angle $\pi$ for each of four types of ``maximal'' critical curves. It is not clear how to adapt the method of proof in \cite{salzuh2} which was used for the plane case to the sphere case, so we use entirely different method in this article.

Now we give an intuitive and brief statement of the main theorem (Theorem \ref{mainTheorem}, which we will prove in this article). Let $\mathcal{L}^{+\kappa_0}_{-\kappa_0}(\mI,\mQ)$ be the space of $C^1$ immersed curves $\gamma:[0,1]\to\mathbb{S}^2$ which satisfies the condition $-\kappa_0<\kappa_\gamma^-(t)\leq\kappa_\gamma^+(t)<\kappa_0$ for all $t \in [0,1]$, starting at Frenet frame $\mI$ and ending at Frenet frame $\mQ$ (see the definitions of $\kappa_\gamma^-(t)$ and $\kappa_\gamma^+(t)$ in Section \ref{sec2}). This space contains the space ${\mathcal{C}^r}_{-\kappa_0}^{+\kappa_0}(I,Q)$ of $C^r$ immersed curves on $\mathbb{S}^2$ for $r\in\{2,3,\ldots,\infty\}$ with geodesic curvature constrained in the interval $(-\kappa_0,+\kappa_0) $. N. C. Saldanha and P. Zühlke (\cite{salzuh}, Lemma 1.11) proved that the inclusion $i:{\mathcal{C}^r}_{-\kappa_0}^{+\kappa_0}(I,Q)\to\mathcal{L}_{-\kappa_0}^{+\kappa_0}(I,Q)$ is a homotopic equivalence. \emph{In this article we consider the case that $\kappa_0 > 1$. For convenience, denote $\rho_0= \arccot \kappa_0$, then $\rho_0\in (0,\frac{\pi}{4} )$}. We need the following concept. 

\begin{definition} We call a curve $\gamma\in\mathcal{L}^{+\kappa_0}_{-\kappa_0}(\mI,\mQ)$ \emph{critical} if it is a concatenation of a finite number of arcs of circles and satisfies the following properties. Let $r_0$, $r_1$,$\ldots$ , $r_k$ be the radii and $\gamma_0$, $\gamma_1$,$\ldots$ , $\gamma_k$ be these arcs of circles, respectively.
\begin{enumerate}
\item The centers of all circles lie in a unique great circle.
\item Each circle has radius in $\left(\rho_0,\frac{\pi}{2}-\rho_0\right)\cup \left(\frac{\pi}{2}+\rho_0,\pi-\rho_0\right)$.
\item For each $i\in\{1,2,\ldots,k-1\}$, $\gamma_i$ has length equal to $\pi\sin r_i$.
\item $\gamma_0$ and $\gamma_k$ have length $<\pi\sin r_0$ and $<\pi\sin r_k$, respectively.
\item The signs of the geodesic curvatures of all segments of arcs of $\gamma$ are alternating. In other words, for each $i\in\{0,1,\ldots,k-1\}$, if the curvature of $\gamma_i$ is positive then the curvature of $\gamma_{i+1}$ is negative and vice-versa.
\item $\gamma$ does not have self-intersections. 
\end{enumerate}
Given a critical curve, we associate to it a string of alternating signs of type ``$+-+-\cdots $'' or ``$-+-+\cdots $'' by the rule: We ``walk'' along the curve and measure the curvature of $\gamma$ from start to end. If the curvature jumps from positive to negative we put a ``$+$'' sign, for each jump from negative to positive we put a ``$-$'' sign.
\end{definition}

Refer Figure \ref{fig:criticalcurve2} for a more geometric view of critical curves.

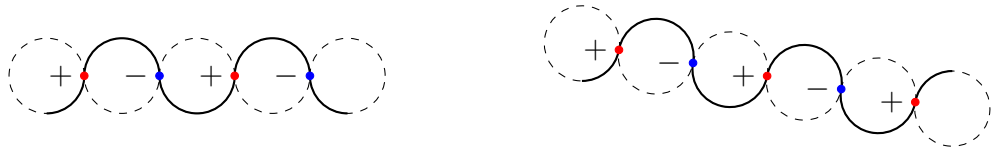
\begin{figure}[htb]
\centering
\begin{tikzpicture}[every text node part/.style={align=center}]
\begin{scope}[scale=.5,yshift=-25]
\draw[dashed] (0,1) circle (1) ;
\draw[dashed] (2,1) circle (1) ;
\draw[dashed] (4,1) circle (1) ;
\draw[dashed] (6,1) circle (1) ;
\draw[dashed] (8,1) circle (1) ;
\draw[thick] (0,1) ++(-90:1) arc (-90:0:1) node[left]{$+$};
\draw[thick] (2,1) ++(180:1) arc (180:0:1) node[left]{$-$};
\draw[thick] (4,1) ++(-180:1) arc (-180:0:1) node[left]{$+$};
\draw[thick] (6,1) ++(180:1) arc (180:0:1)node[left]{$-$};
\draw[thick] (8,1) ++(-180:1) arc (-180:-90:1);
\draw[fill,red] (1,1) circle (0.1);
\draw[fill,blue] (3,1) circle (0.1);
\draw[fill,red] (5,1) circle (0.1);
\draw[fill,blue] (7,1) circle (0.1);
\end{scope}
\begin{scope}[scale=.5,xshift=400,rotate=-10]
\draw[dashed] (0,1) circle (1) ;
\draw[dashed] (2,1) circle (1) ;
\draw[dashed] (4,1) circle (1) ;
\draw[dashed] (6,1) circle (1) ;
\draw[dashed] (8,1) circle (1) ;
\draw[dashed] (10,1) circle (1) ;
\draw[thick] (0,1) ++(-80:1) arc (-80:0:1) node[left]{$+$};
\draw[thick] (2,1) ++(180:1) arc (180:0:1) node[left]{$-$};
\draw[thick] (4,1) ++(-180:1) arc (-180:0:1) node[left]{$+$};
\draw[thick] (6,1) ++(180:1) arc (180:0:1) node[left]{$-$};
\draw[thick] (8,1) ++(-180:1) arc (-180:0:1) node[left]{$+$};
\draw[thick] (10,1) ++(180:1) arc (180:100:1);
\draw[fill,red] (1,1) circle (0.1);
\draw[fill,blue] (3,1) circle (0.1);
\draw[fill,red] (5,1) circle (0.1);
\draw[fill,blue] (7,1) circle (0.1);
\draw[fill,red] (9,1) circle (0.1);
\end{scope}
\end{tikzpicture}
\caption{The curve on the left is critical of type $+-+-$ and the curve on the right is critical of type $+-+-+$ (the signs represents the inclination of tangent vectors at dotted points). Meanwhile the dashed circles have radii greater than $\rho_0=\arccot \kappa_0$ and are aligned so that their centers are on the same geodesic.}
\label{fig:criticalcurve2}
\end{figure}

We now give an intuitive description of the main theorem:

\begin{theorem}[informal statement of the main theorem (Theorem \ref{mainTheorem})] Given a matrix $\mQ\in\SO$, the following information about the topology of $\mathcal{L}^{+\kappa_0}_{-\kappa_0}(\mI,\mQ)$ can be obtained by analyzing critical curves in $\mathcal{L}^{+\kappa_0}_{-\kappa_0}(\mI,\mQ)$. If there exist an integer $n\geq 1$ and critical curves of types $\underbrace{+-+-\ldots}_{n+1}$ and $\underbrace{-+-+\ldots}_{n+1}$, and there is neither a critical curve of types $\underbrace{+-+-\ldots}_{n+2}$ nor $\underbrace{-+-+\ldots}_{n+2}$, then there exists an exotic generator of $\ho_n\big(\mathcal{L}^{+\kappa_0}_{-\kappa_0}(\mI,\mQ)\big) $.
\end{theorem}

\noindent The generator mentioned in the theorem above will be constructed explicitly. A formal and detailed statement of the main theorem will be presented in Section 2.

On the other hand, we show that for some $\mQ$, the map $\boldsymbol{i}:\Lspace\hookrightarrow\mathcal{I}(\mI,\mQ)$ is a homotopical equivalence in Section \ref{trivialresult}. This result gives information about the topology of the space of curves based only on $\mQ$ and $\rho_0$. 

\subsection{The topology of curves in higher dimension spheres, plane and other spaces}

One may be curious whether there are similar properties for the space of curves on spheres $\mathbb{S}^n$ of higher dimensions. Indeed there are some studies: \cite{shasha}, \cite{shap}, \cite{salsha}, \cite{alves}, \cite{alves2} and \cite{goulart}.

The research on the topological aspects of spaces of curves has not been restricted exclusively to sphere $\mathbb{S}^n$. For curves on $2$-dimensional Euclidean plane, here we mention the articles: \cite{whit}, \cite{dubins1}, \cite{dubins2}, \cite{salzuh2}, \cite{salzuh3}, \cite{ayala4}, \cite{ayala3}, \cite{ayala2}, \cite{ayala}. We also mention \cite{anisov} for $\mathbb{RP}^2$, \cite{salzuh4} for two dimensional hyperbolic space with constant curvature $-1$, and \cite{smale} and \cite{mossad} for general Riemannian manifolds, respectively. 

\textbf{Acknowledgement}. I would like to thank my advisor Nicolau C. Saldanha for his guidance and great support throughout my PhD studies. I also would like to thank Capes and Faperj for financial support (scholarship) during my graduate studies at PUC-Rio. 

\newpage
\section{Statement of the main theorem}\label{sec2}

This section begins with the introduction of some definitions which we will use throughout the text. After that, we present our main result and the scheme of proof.

\subsection{Definition of $\mathcal{I}$}
For completeness, we present the definition for the space $\mathcal{I}$ of immersed curves in $\mathbb{S}^2$. We consider all $C^1$ applications of type $\gamma: I_\gamma\to\mathbb{S}^2$, $\gamma'(t)\neq 0$, for all $t\in I_\gamma$, where $I_\gamma\subset\mathbb{R}$ is a closed non-degenerate interval. We say that two applications:
$$ \alpha:I_\alpha\to\mathbb{S}^2 \quad \text{and} \quad \beta:I_\beta\to\mathbb{S}^2 . $$
\noindent are equivalent if there exists a $C^1$ strictly increasing bijection $\bar{t}:I_\alpha\to I_\beta$, $\bar{t}'(t)>0$, such that:
$$ \alpha(t) = \left(\beta\circ\bar{t}\right)(t). $$
\noindent One may have noted that $\beta$ is just a \emph{reparametrization} of $\alpha$. We use the notation $\alpha\sim\beta$. It can be easily verified that $\sim$ is an equivalence relation. 
The space of $C^1$ \emph{immersed curves} on $\mathbb{S}^2$ denoted by $\mathcal{I}$ is the following quotient space:
$$\mathcal{I} = \bigslant{\left\{\gamma:I_\gamma\to\mathbb{S}^2; \gamma \text{ is a $C^1$ application and $\gamma'(t)\neq 0$, for all $t\in I_\gamma$}\right\}}{\sim}. $$
\noindent \label{equivpg} By abuse of notation, we will use $\alpha$ to represent the equivalence class $[\alpha]=\{\beta;\alpha\sim\beta\}\in\mathcal{I}$, and call $\alpha$ a $C^1$ immersed curve on $\mathbb{S}^2$, or an immersed curve for short. 
Now we recall the concept of arc-length. Given an immersed curve $\gamma: [0,1]\to\mathbb{S}^2$ with parameter $t$, define the arc-length of $\gamma$ by $s:[0,1]\to [0,L_\gamma]$ as follows:
$$ s(t)\coloneqq \int_0^{t} |\gamma'(t)| dt  , $$
\noindent where $L_\gamma=\int_0^1|\gamma'(t)|dt$ is the \emph{length} of $\gamma$. Since $|\gamma'|> 0$, $s$ is a strictly increasing function. 
By re-parametrizing the curve by arc-length $s$ we obtain a curve $\gamma: [0,L_\gamma]\to\mathbb{S}^2$ 
with $|\gamma'(s)|\equiv 1$. We will use the notation $\vt_\gamma(t)$ to denote the unit tangent vector $\gamma'(s)\vert_{s=s(t)}$ at $\gamma(t)$. 

Given any two immersed curves $\alpha$ and $\beta$, let $L_\alpha$ and $L_\beta$ denote their lengths. Reparametrize both curves proportionally to arc-length with:
 $$|\alpha'(t)|=L_\alpha \quad \text{and} \quad |\beta'(t)|=L_\beta,$$
so that $\alpha,\beta:[0,1]\to\mathbb{S}^2$ have constant speeds respectively. Define:
$$\bar{d}(\alpha,\beta) = \max \left\{d\big(\alpha\left(t\right),\beta\left(t\right)\big)+d\big(\vt_{\alpha}\left(t\right),\vt_{\beta}\left(t\right)\big);t\in [0,1] \right\}.$$ 
\noindent In the equation above, $d$ is the distance measured on the surface between two points in $\mathbb{S}^2$. It is easy to check that $\bar{d}$ is well defined on $\mathcal{I}$, and a distance function. So the pair $(\mathcal{I},\bar{d})$ is a metric space.  We have the usual $C^1$ topology in $\mathcal{I}$, induced by the metric $\bar{d}:\mathcal{I}\times\mathcal{I}\to [0,\infty)$. We use this topology throughout the text. 

\subsection{Definition of spaces $\Lspace$ and $\bar{\mathcal{L}}_{\rho_0}(\mQ)$}

Given a $C^1$ immersed curve $\gamma:I\to \mathbb{S}^2$, we define the \emph{unit normal vector} $\boldsymbol{n}_\gamma$ to $\gamma$ by
$$\bsm{n}_\gamma(t)=\gamma(t)\times\vt_\gamma(t),$$
where $\times$ denotes the vector product in $\mathbb{R}^3$. If $\gamma$ also has the second derivative, the \emph{geodesic curvature} $\kappa_\gamma(s)$ at $\gamma(s)$ is defined by
\begin{equation}\label{curv}
\kappa_\gamma(s)=\left\langle\vt_\gamma'(s), \bsm{n}_\gamma(s)\right\rangle,
\end{equation} 
\noindent where $s$ is the arc-length of $\gamma$. Remember that we are working with $C^1$ curves, so the geodesic curvature may not be well defined for these curves. Here we establish a broader definition of the curvature $\kappa_\gamma(s)$ for $C^1$ regular curves below (see Figure \ref{fig:broader}, for an intuition of this concept). Given a $C^1$ curve $\gamma:I_1\to\mathbb{S}^2$ and a circle $\zeta:I_2\to\mathbb{S}^2$, we say that $\zeta$ is tangent to $\gamma$ at $\gamma(t_1)$, $t_1\in I_1$, from the \emph{left} if the next conditions are satisfied:
\begin{figure}
\begin{tikzpicture}
\begin{scope}[xshift = -0cm, decoration={markings, mark=at position 0.5 with {\arrowreversed {Latex}}}]
\draw[dashed,blue] (-1.0,0) circle (1.0);
\draw[dashed,red] (1.2,0) circle (1.2);
\draw [postaction={decorate},thick] (0,0) arc (180:270:1.4);
\end{scope}
\begin{scope}[xshift = -0cm, decoration={markings, mark=at position 0.5 with {\arrow {Latex}}}]
\draw [postaction={decorate},thick] (0,0) arc (0:90:1.7);
\end{scope}
\begin{scope}[xshift = -5cm, decoration={markings, mark=at position 0.3 with {\arrow {Latex}},mark=at position 0.7 with {\arrow {Latex}}}]
\draw[dashed,blue] (-1.1,0) circle (1.1);
\draw[dashed,red] (1.2,0) circle (1.2);
\draw [postaction={decorate},thick] plot [smooth, tension = 1.0] coordinates {(-.3,-1.5) (0,0) (-.3,1.5)};
\end{scope}
\begin{scope}[xshift = 5cm, decoration={markings, mark=at position 0.3 with {\arrow {Latex}}, mark=at position 0.7 with {\arrow {Latex}}}]
\draw[dashed,red] (1.4,0) circle (1.4);
\draw[postaction={decorate},scale=0.5,domain=-1:1,smooth,variable=\y,thick]  plot ({-\y*\y*\y*\y},{3*\y*\y*\y});
\end{scope}
\end{tikzpicture}
\caption{
In the left hand side image is a smooth curve. In the center image is a piece-wise $C^2$ curve. In the right hand side image is the curve given by the spherical projection of the plane curve $t\mapsto(-t^4,t^3)$. Note that there does not exist a circle tangent to this curve at $(0,0)$ from the left. The second and the third curves are $C^1$ regular curves, but not $C^2$. However, we will define the concept of the left and the right curvatures for these curves. For the rightmost curve, the left and right curvature at the projection of $(0,0)$ are both $+\infty$.}
\label{fig:broader}
\end{figure}
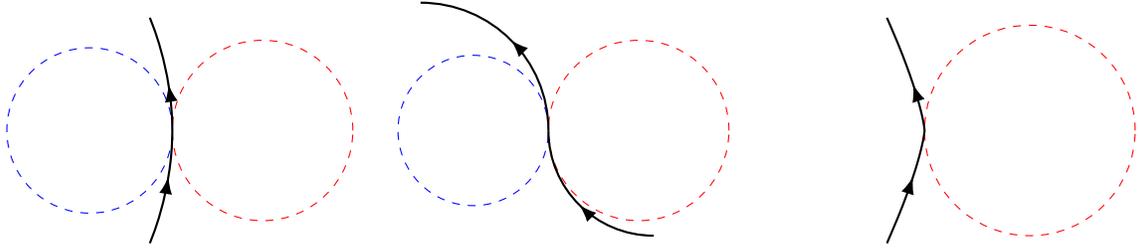

\begin{figure}
\begin{tikzpicture}
\begin{scope}[scale=1.7, xshift = 5cm, decoration={markings, mark=at position 0.3 with {\arrow {Latex}}, mark=at position 0.7 with {\arrow {Latex}}}]
\draw[dashed,red] (1.4,0) circle (1.4);
\draw[dashed,blue] (-1.4,0) circle (1.4);
\draw[postaction={decorate},scale=0.5,domain=-1:1,smooth,variable=\y,thick]  plot ({-2*\y*\y*\y*\y*\y},{3*\y*\y*\y});
\node[anchor=west] (0,0) {$\gamma(0)$};
\draw[fill=black] (0,0) circle (.025);
\end{scope}
\end{tikzpicture}
\caption{
The graph of the spherical projection of the plane curve $\gamma(t)$ given by $t\mapsto (-t^5,t^3)$. Note that there is neither a circle tangent to the curve at $(0,0)$ from the left nor from the right. The left and the right curvatures in the point of inflection $\gamma(0)$ are $+\infty$ and $-\infty$, respectively.}
\label{fig:ctex}
\end{figure}
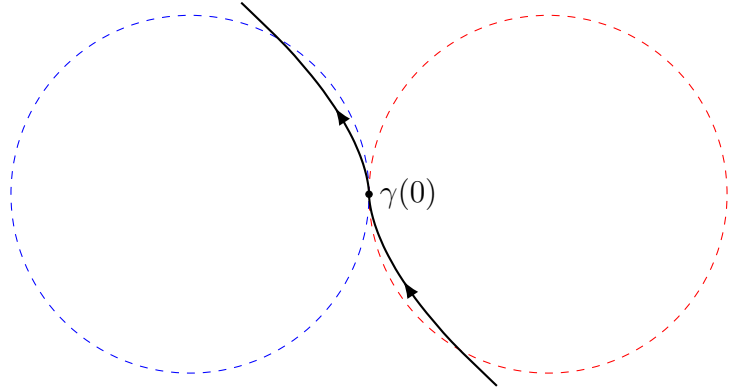

\begin{enumerate}
\item There exists a $t_2\in I_2$ such that $\gamma(t_1)=\zeta(t_2)$ and $\gamma'(t_1)=\zeta'(t_2)$.
\item Denote the center of $\zeta$ by $a$ so that $\zeta$ travels \emph{anti-clockwise} with respect to $a$ and denote by $r$ the radius (measured on sphere) of $\zeta$ in relation to $a$. There exists a $\delta>0$ such that:
$$ d(\gamma(t),a)\geq r, \quad \forall t\in (t_1-\delta,t_1+\delta) .$$
In the above inequality, $d$ is the distance measured on $\mathbb{S}^2$. 
\end{enumerate}
In the same manner we say that $\zeta$ is tangent to $\gamma$ at $\gamma(t_1)$ from the \emph{right} by replacing Condition (2) with:
\begin{enumerate}
\item[(2')] Denote the center of $\zeta$ by $a$ so that $\zeta$ travels \emph{anti-clockwise} with respect to it and denote by $r$ the radius (measured on sphere) of $\zeta$ in relation to $a$. There exists a $\delta>0$ such that:
$$ d(\gamma(t),a)\leq r, \quad \forall t\in (t_1-\delta,t_1+\delta) .$$ 
\end{enumerate}

We define the \emph{left} and the \emph{right curvatures}, denoted by $\kappa_\gamma^+$ and $\kappa_\gamma^-$, respectively:
\begin{align*}
\kappa_\gamma^+(t) &=\inf\left\{\cot(r);\text{where $r$ is the radius of a circle tangent to $\gamma$ at $\gamma(t)$ from the left.} \right\}\\
\kappa_\gamma^-(t) &= \sup\left\{\cot(r);\text{where $r$ is the radius of a circle tangent to $\gamma$ at $\gamma(t)$ from the right.} \right\}.
\end{align*}
\noindent We follow the convention that $\inf\emptyset = +\infty$ and $\sup\emptyset = -\infty$. Note that $\kappa_\gamma^+(t)\geq\kappa_\gamma^-(t)$ for all $t\in I_1$. When the equality occurs for some $t$, we define the \emph{curvature} of $\gamma$ as $\kappa_\gamma(t)=\kappa_\gamma^+(t)=\kappa_\gamma^-(t)$. For $C^2$ curves, the definition of curvature coincides with the usual definition of the \emph{geodesic curvature} (see Equation (\ref{curv})). 

We also define the \emph{Frenet frame} of $\gamma$ by:
$$\mathfrak{F}_{\gamma}(t)=\left( \begin{array}{ccc} | & | & | \\
\gamma(t) & \vt_\gamma (t) & \bsm{n}_\gamma(t) \\
| & | & | \end{array} \right) \in \textrm{SO}_3(\mathbb{R}) .$$

The space $\SO$ is homeomorphic to the unit tangent bundle of sphere $\textup{UT}\mathbb{S}^2$ by mapping the matrix $\bsm{M}\in\SO$ to the vector $\bsm{M}(0,1,0) \in \textup{T}_{\bsm{M}(1,0,0)}\mathbb{S}^2$. Now we define the spaces $\mathcal{I}(\mP,\mQ)$, $\mathcal{L}_{\kappa_1}^{\kappa_2}(\mP,\mQ)$ and $\bar{\mathcal{L}}_{\kappa_1}^{\kappa_2}(\mP,\mQ)$:

\begin{definition}\label{deflspace}
Given $\mP,\mQ\in \textrm{SO}_3(\mathbb{R})$, $\kappa_1,\kappa_2 \in [-\infty,+\infty]$, with $\kappa_1\leq\kappa_2$. 
\begin{itemize}
\item Let $\mathcal{I}(\mP,\mQ)$ be the space of all $C^1$ immersed curves in $\mathbb{S}^2$ with Frenet frames $\mathfrak{F}_\gamma(0)=\mP$ and $\mathfrak{F}_\gamma(1)=\mQ$. We will use the notation $\mathcal{I}(\mQ)$, when $\mP=\mI$.
\item Let $\mathcal{L}^{\kappa_2}_{\kappa_1}(\mP,\mQ)\subset\mathcal{I}(\mP,\mQ)$ be the subspace of curves that satisfies $\kappa_1<\kappa_\gamma^-(t)\leq\kappa_\gamma^+(t)<\kappa_2$ for all $t \in [0,1]$. 
\item Let $\bar{\mathcal{L}}^{\kappa_2}_{\kappa_1}(\mP,\mQ)\subset\mathcal{I}(\mP,\mQ)$ be the subspace of curves that satisfies $\kappa_1\leq\kappa_\gamma^-(t)\leq\kappa_\gamma^+(t)\leq\kappa_2$ for all $t \in [0,1]$. 
\end{itemize} 
We will also adopt shorter notations when these spaces are symmetric in the sense that $-\kappa_1=\kappa_2=\kappa_0$, with $\kappa_0\in (0,+\infty]$. Let $\rho_0 \coloneqq \arccot(\kappa_0)$, we will mostly use $\Lspace \coloneqq \mathcal{L}^{+\kappa_0}_{-\kappa_0}(\mI,\mQ)$ and  $\cLspace \coloneqq \bar{\mathcal{L}}^{+\kappa_0}_{-\kappa_0}(\mI,\mQ)$.
\end{definition}

Note that this definition of $\mathcal{L}_{\kappa_1}^{\kappa_2}(\mat{P},\mat{Q})$ differs from the definition of the space studied in \cite{salzuh}. In fact the spaces of curves originated from these definitions are indeed different sets, but there is a homotopy equivalence between them. For a detailed discussion and characterization of this space see Subsection
 \ref{appdeftop}. 
 
 There is no loss of generality in considering only the situation that $\mP=\mI$, because the space $\mathcal{L}_{\kappa_1}^{\kappa_2}(\mP,\mQ)$ is homeomorphic to $\mathcal{L}_{\kappa_1}^{\kappa_2}(\mI,\mP^{-1}\mQ)$ via the map $\gamma\mapsto \mP^{-1}\gamma$ (the same is valid for $\bar{\mathcal{L}}_{\kappa_1}^{\kappa_2}(\mP,\mQ)$). We study spaces $\mathcal{L}_{\kappa_1}^{\kappa_2}(\mI,\mQ)$ and $\bar{\mathcal{L}}_{\kappa_1}^{\kappa_2}(\mI,\mQ)$, there is no loss of generality in assuming the intervals $(\kappa_1,\kappa_2)$ and $[\kappa_1,\kappa_2]$ to be $(-\kappa_0,\kappa_0)$ and $[-\kappa_0,\kappa_0]$, respectively. This is due to the following result((1.22)Theorem A in \cite{salzuh}).

\begin{theorem}\label{thm8} Let $\mQ\in\SO$, $\kappa_1,\kappa_2, \bar{\kappa}_1, \bar{\kappa}_2\in [-\infty,+\infty]$ such that $\kappa_1<\kappa_2$ and $\bar{\kappa}_1<\bar{\kappa}_2$. Define $\rho_i = \arccot \kappa_i$ and $\bar{\rho}_i = \arccot \bar{\kappa}_i$, for $i=1,2$. Suppose that: 
$$\rho_1 - \rho_2 = \bar{\rho}_1 - \bar{\rho}_2 .$$
\noindent Then there exists a homeomorphism between the spaces $\mathcal{L}^{\kappa_2}_{\kappa_1}(\mQ)$ and $\mathcal{L}^{\bar{\kappa}_2}_{\bar{\kappa}_1}(\rot_{-\theta}\mQ\rot_{\theta})$, where $\theta = \rho_2 - \bar{\rho}_2$ and
$$\rot_\theta = \left(\begin{array}{ccc} \cos\theta & 0 & -\sin\theta \\
0 & 1 & 0 \\
\sin\theta & 0 & \cos\theta
\end{array} \right)$$
\noindent is the rotation matrix around the axis $(0,1,0)$ by the right-hand rule.
\end{theorem}
\noindent For the space $\bar{\mathcal{L}}_{\kappa_1}^{\kappa_2}(\mI,\mQ)$ the conclusion and the proof of Theorem \ref{thm8} are analogous. It turned out that the smoothness condition about the curve does not change the homology of the space $\Lspace$, due to the following theorem (the original proof is given in \cite{salzuh}): 

\begin{definition}
Let $\rho_0\in \left[0,\frac{\pi}{2}\right)$, $\kappa_0 = \arccot \rho_0$, $\mQ\in\SO$ and $r\in\{2,3,\ldots,\infty\}$. Define $\mathcal{C}_{\rho_0}(\mQ)$ to be the set of all $C^r$ regular curves $\gamma:[0,1]\to\mathbb{S}^2$ furnished with $C^r$ topology, with $\gamma$ such that:
\begin{enumerate}
\item $\mathfrak{F}_\gamma(0) = \mI$ and $\mathfrak{F}_\gamma(1) = \mQ$;
\item $-\kappa_0<\kappa_\gamma(t)<\kappa_0$ for each $t\in [0,1]$.
\end{enumerate}
\end{definition}

\begin{theorem}
Let $\rho_0\in \left[0,\frac{\pi}{2}\right)$, $\kappa_0 = \arccot \rho_0$, $\mQ\in\SO$ and $r\in\{2,3,\ldots,\infty\}$. Then the set inclusion $\boldsymbol{i}: \mathcal{C}_{\rho_0}(\mQ)\hookrightarrow\mathcal{L}_{\rho_0}(\mQ)$ is a homotopy equivalence. Therefore, the sets $\mathcal{C}_{\rho_0}(\mQ)$ and $\Lspace$ are homeomorphic.
\end{theorem}

\noindent This theorem is shown in the subsection \ref{appdeftop} (Theorem \ref{thmequiv}). However, in contrast to the previous remarks, the above property is only valid for $\Lspace$. In fact, in general, the spaces $\cLspace$ and $\bar{\mathcal{C}}_{\rho_0}(\mQ)$ are not homotopics.

\subsection{Statement of the main theorem}\label{secformulapq}
The space $\mathcal{I}(\mQ)$ is weakly homotopically equivalent to the space $\Omega\SO$ (the space of loops in $\SO$), refer \cite{elimis} and \cite{grom} for more details. Moreover, $\Omega\SO \simeq \Omega\mathbb{S}^3\sqcup\Omega\mathbb{S}^3$, namely one of these connected component consists of curves with even number of self-intersections, and the other one consists of curves with odd number of self-intersections. For description of the topology of $\Omega\mathbb{S}^3$ refer \cite{milnor}. In this book it is shown that the loop space $\Omega\mathbb{S}^3$ has the homotopy type of a CW-complex with exactly one cell in each of the dimensions $0$, $2$, $4$, $6$, \ldots, $2k$, \ldots, for $k\in \mathbb{N}$. 

Observe that $\Lspace\subset\mathcal{I}(Q)$. It is known from an analogous result of \cite{sald} that the inclusion map $\boldsymbol{i}:\Lspace\to\mathcal{I}(\mQ) $ induces a surjetive map on homology (refer Proposition \ref{prop64}):
\begin{equation}\label{eqhominc}
\ho_k\big(\boldsymbol{i}\big):\ho_k\big(\Lspace\big)\to\ho_k\big(\mathcal{I}(\mQ)\big).
\end{equation}
Our objective is to understand the topological structure of the space $\Lspace$. In this article we prove its homology differs from the homology of $\mathcal{I}(\mQ)$. Our strategy is to construct some specific non-trivial maps $F: \mathbb{S}^n \to \Lspace$ and $G: \Lspace \to \mathbb{S}^n$, for some $n=n_\mQ\in\mathbb{N}$ depending on $\mQ$, such that $F$ and $G$ satisfy the properties:
$$ (G\circ F) : \mathbb{S}^n \to \mathbb{S}^n \quad \text{has degree $1$ and $(\boldsymbol{i}\circ F):\mathbb{S}^n\to\mathcal{I}(\mQ)$ is a trivial map.}$$
\noindent The existence of such maps implies $\ho_n(\boldsymbol{i})\big([F]\big)=0$, but $[F]\neq 0$ in $\ho_n\big(\Lspace\big)$, where $[F]$ denotes the homotopy equivalence class of $F$. Hence the map $\ho_n\big(\boldsymbol{i}\big)$ is \emph{not} injective, from (\ref{eqhominc}) we deduce that the inclusion map $\boldsymbol{i}$ is not a homotopic equivalence. 

Denote by $\left\{\bar{\vet{e}}_1,\bar{\vet{e}}_2,\bar{\vet{e}}_3\right\}$ the basis in $\text{so}_3(\mathbb{R})=\tangent{\mI}{\SO}$ (the Lie algebra of $\textrm{SO}_3(\mathbb{R})$, which is the set of $3\times 3$ anti-symmetric matrices), given by:
$$  \bar{\vet{e}}_1 = \left(\begin{matrix}
0 & 0 & 0 \\
0 & 0 & -1 \\
0 & 1 & 0 \\
\end{matrix}\right),\quad 
\bar{\vet{e}}_2 = \left(\begin{matrix}
0 & 0 & 1 \\
0 & 0 & 0 \\
-1 & 0 & 0 \\
\end{matrix}\right)\quad \text{ and }\quad
\bar{\vet{e}}_3 = \left(\begin{matrix}
0 & -1 & 0 \\
1 & 0 & 0 \\
0 & 0 & 0 \\
\end{matrix}\right).$$
\noindent Note that the exponentials of matrices above are rotations around the axes $x$, $y$ and $z$ respectively. 

Let $v\in\mathbb{S}^2$. We define $\rot_{\rho} (v)$ as the anti-clockwise rotation of angle $\rho$ around the axis generated by the direction from the origin to $v$, following the right hand rule. This rotation is represented by a matrix in $\SO$, which is given by the following formula:
\begin{equation}\label{eqrotation}
\rot_\rho(v)=\exp\big(\rho\left(\left\langle e_1, \vet{v} \right\rangle\bar{\vet{e}}_1 + \left\langle e_2, v \right\rangle\bar{\vet{e}}_2 + \left\langle e_3, v \right\rangle\bar{\vet{e}}_3
\right)\big),
\end{equation}

\noindent where $e_1=(1,0,0)$, $e_2=(0,1,0)$ and $e_3=(0,0,1)$. For $\vet{M}\in\SO$, we use notations:
$$p_1(\mat{M})=\left[\rot_{\rho_0}(\mat{M}e_2)\right](\mat{M}e_1) \quad\text{ and }\quad p_2(\mat{M})=\left[\rot_{-\rho_0}(\mat{M}e_2)\right](\mat{M}e_1).$$ 

\noindent In the other words,
$$ p_1(\mat{M}) = \mat{M}(\cos\rho_0,0,\sin\rho_0) \quad\text{and}\quad p_2(\mat{M})=\mat{M}(\cos\rho_0,0,-\sin\rho_0) .$$
In the definition above, if one views the matrix $\mat{M}$ as the Frenet frame of a curve $\gamma\in\Lspace$ on $\gamma(t)$, the point $p_1$ is the center of the circle of radius $\rho_0$ tangent to $\gamma$ at $\gamma(t)$ from the left. Analogously, the point $p_2$ is the center of the circle of radius $\rho_0$ tangent to $\gamma$ at $\gamma(t)$ from the right. For example, $p_1(\mI)=(\cos\rho_0,0,\sin\rho_0)$ and $p_2(\mI)=(\cos\rho_0,0,-\sin\rho_0)$. 

Due to frequent appearance in the text, we will also use the shorter notations:
\begin{equation}\label{formulapq}
p_1 = p_1(\mI),\quad p_2 = p_2(\mI), \quad q_1 = p_1(\mQ) \quad \text{and} \quad q_2 = p_2(\mQ).
\end{equation}

\noindent Geometrically, $p_1$ and $p_2$ are the centers of the circles of radius $\rho_0$ tangent to the curves in $\Lspace$ at the time $t=0$ from the left and the right, respectively. On the other hand, $q_1$ and $q_2$ are the centers of the circles of radius $\rho_0$ tangent to the curves in $\Lspace$ at the end of the curves from the left and right, respectively. 

Consider the following lengths measured on $\mathbb{S}^2$ given by:
$$D_1 \coloneqq d(\p1,\q2), \quad D_2 \coloneqq d(\p2,\q1), \quad L_1 \coloneqq d(\p1,\q1), \quad L_2 \coloneqq d(\p2,\q2),$$ 
\noindent so that $D_1$, $D_2$ represent the lengths of two diagonals of quadrilateral $\square p_1q_1q_2p_2$ and $L_1$, $L_2$ represent the lengths of the sides $p_1q_1$ and $p_2q_2$, respectively (see Figure \ref{fig:criticalcurve}). 

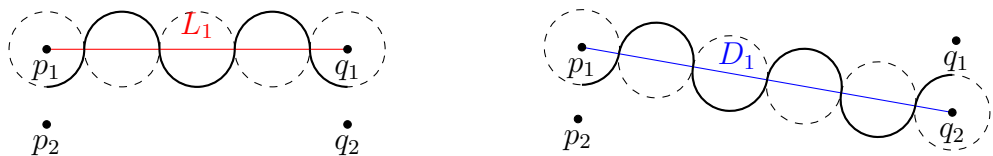
\begin{figure}[htb]
\centering
\begin{tikzpicture}[every text node part/.style={align=center}]
\begin{scope}[scale=.5,yshift=-25]
\draw[dashed] (0,1) circle (1) ;
\draw[dashed] (2,1) circle (1) ;
\draw[dashed] (4,1) circle (1) ;
\draw[dashed] (6,1) circle (1) ;
\draw[dashed] (8,1) circle (1) ;
\draw[red] (0,1) -- (8,1) node[midway,above,red] {$L_1$};
\draw[fill] (0,1) node[below] {$p_1$} circle (0.1) ;
\draw[fill] (0,-1) node[below] {$p_2$} circle (0.1) ;
\draw[fill] (8,1) node[below] {$q_1$} circle (0.1) ;
\draw[fill] (8,-1) node[below] {$q_2$} circle (0.1) ;
\draw[thick] (0,1) ++(-90:1) arc (-90:0:1);
\draw[thick] (2,1) ++(180:1) arc (180:0:1);
\draw[thick] (4,1) ++(-180:1) arc (-180:0:1);
\draw[thick] (6,1) ++(180:1) arc (180:0:1);
\draw[thick] (8,1) ++(-180:1) arc (-180:-90:1);
\end{scope}
\begin{scope}[scale=.5,xshift=400,yshift=-23,rotate=-10]
\draw[dashed] (0,1) circle (1) ;
\draw[dashed] (2,1) circle (1) ;
\draw[dashed] (4,1) circle (1) ;
\draw[dashed] (6,1) circle (1) ;
\draw[dashed] (8,1) circle (1) ;
\draw[dashed] (10,1) circle (1) ;
\draw[blue] (0,1) -- (10,1) node[midway,anchor=south east,blue] {$D_1$};
\draw[fill] (0,1) node[below] {$p_1$} circle (0.1) ;
\draw[fill] (0.23,-0.9) node[below] {$p_2$} circle (0.1) ;
\draw[fill] (10,1) node[below] {$q_2$} circle (0.1) ;
\draw[fill] (9.77,2.9) node[below] {$q_1$} circle (0.1) ;
\draw[thick] (0,1) ++(-80:1) arc (-80:0:1);
\draw[thick] (2,1) ++(180:1) arc (180:0:1);
\draw[thick] (4,1) ++(-180:1) arc (-180:0:1);
\draw[thick] (6,1) ++(180:1) arc (180:0:1);
\draw[thick] (8,1) ++(-180:1) arc (-180:0:1);
\draw[thick] (10,1) ++(180:1) arc (180:100:1);
\end{scope}
\end{tikzpicture}
\caption{These are critical curves of indices 3 and 4 respectively (from left to right) which are contained in $\cLspace$ (but \emph{not} in $\Lspace$). Note that the amount of hills and valleys that we are able to add on the critical curve is directly related to the distance between points, which are $L_1$ and $D_1$ respectively. To be able to construct a critical curve similar to the image in the left in $\Lspace$, we need $L_1> 8\rho_0$, and for the image in the right, we need $D_1>10\rho_0$. These examples motivate us to give the definition in Equation ($\ref{eqLD}$).}
\label{fig:criticalcurve}
\end{figure}

For $i=1,2$, define the truncated lengths which will be used to enunciate the main theorem:
\begin{equation} \label{eqLD}
\bar{L}_i \coloneqq 2\left\lceil\frac{L_i}{4\rho_0}\right\rceil - 3 \quad\text{and}\quad \bar{D}_i \coloneqq 2\left\lceil\frac{D_i}{4\rho_0}-\frac{1}{2}\right\rceil-2.
\end{equation}
\noindent In the equation above, $\lceil x \rceil$ represents the least integer that is greater than or equal to $x$. Note that $\bar{L}_i$ is always an odd integer and $\bar{D}_i$ is always an even integer. These two numbers describe, intuitively, the index of a ``maximal critical curve'' of even or odd type (see Figure \ref{fig:criticalcurve}). See also Figure \ref{fig:graph} for the graphs of $\bar{L}_i$ and $\bar{D}_i$ as functions of $L_i$ and $D_i$, respectively. The reason of this definition will be further clarified in the subsequent sections.

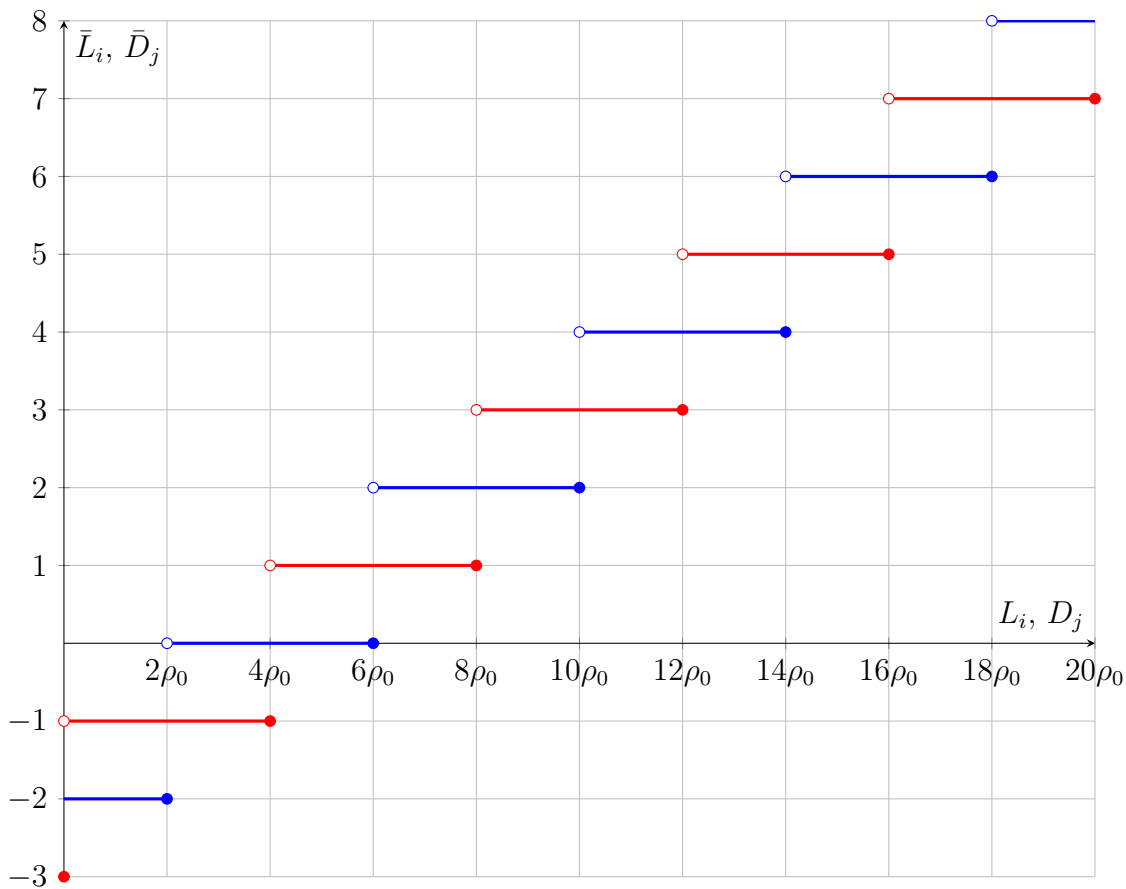
\begin{figure}[htb]
\begin{tikzpicture} 
\begin{axis}[ylabel={$\bar{L}_i$, $\bar{D}_j$},
        xlabel={$L_i$, $D_j$},
        xtick={2,4,6,8,10,12,14,16,18,20},
        ytick={-3,-2,-1,0,1,2,3,4,5,6,7,8},
        xticklabels={$2\rho_0$,$4\rho_0$,$6\rho_0$,$8\rho_0$,$10\rho_0$,$12\rho_0$,$14\rho_0$,$16\rho_0$,$18\rho_0$,$20\rho_0$},
        grid=both,axis lines=middle,scale=2]
\addplot [
    jump mark mid,
    domain=0:20,
    samples=100,
    very thick, red
] {2*ceil(x/4)-3};
\addplot [
    jump mark mid,
    domain=0:20,
    samples=100,
    very thick, blue
] {2*ceil(x/4-1/2)-2};
\addplot[incl,blue] coordinates {(2,-2)};
\addplot[incl,blue] coordinates {(6,0)};
\addplot[incl,blue] coordinates {(10,2)};
\addplot[incl,blue] coordinates {(14,4)};
\addplot[incl,blue] coordinates {(18,6)};
\addplot[incl,red] coordinates {(0,-3)};
\addplot[incl,red] coordinates {(4,-1)};
\addplot[incl,red] coordinates {(8,1)};
\addplot[incl,red] coordinates {(12,3)};
\addplot[incl,red] coordinates {(16,5)};
\addplot[incl,red] coordinates {(20,7)};
\addplot[exclb] coordinates {(2,0)};
\addplot[exclb] coordinates {(6,2)};
\addplot[exclb] coordinates {(10,4)};
\addplot[exclb] coordinates {(14,6)};
\addplot[exclb] coordinates {(18,8)};
\addplot[exclr] coordinates {(0,-1)};
\addplot[exclr] coordinates {(4,1)};
\addplot[exclr] coordinates {(8,3)};
\addplot[exclr] coordinates {(12,5)};
\addplot[exclr] coordinates {(16,7)};
\end{axis}
\end{tikzpicture}
\caption{Graph of $\bar{L}_i$ as function of $L_i$ in red and $\bar{D}_j$ as function of $D_j$ in blue.}
\label{fig:graph}
\end{figure}

\begin{lemma} If $\bar{L}_i>\bar{D}_j$ for all $i,j\in\{1,2\}$ then $\bar{L}_1=\bar{L}_2$. In the same manner, if $\bar{D}_i>\bar{L}_j$ for all $i,j\in\{1,2\}$ then $\bar{D}_1=\bar{D}_2$.
\end{lemma}

\begin{proof} For the first part of the lemma, suppose, by contradiction, that $\bar{L}_1\neq\bar{L}_2$. Without loss of generality, we assume that $\bar{L}_1>\bar{L}_2$. By the triangular inequality:
$$ |L_1-L_2|=|d(p_1,q_1)-d(p_2,q_2)|\leq d(p_1,p_2)+d(q_1,q_2)=4\rho_0 .$$
This implies $\bar{L}_2=\bar{L}_1-2$ (see Figure \ref{fig:graph}). On the other hand, again by the triangular inequality:
$$ |L_1-D_1|=|d(p_1,q_1)-d(p_1,q_2)|\leq d(q_1,q_2) =2\rho_0 .$$ 
This implies $\bar{D}_1\geq \bar{L}_1-1$ (see Figure \ref{fig:graph}). Thus $\bar{D}_1>\bar{L}_1-2=\bar{L}_2$ which contradicts the initial hypothesis that $\bar{L}_2>\bar{D}_1$. The proof for the second part of the lemma is analogous.
\end{proof}

With the lemma above, we introduce the definition of the index of $\mQ$ and the main theorem.

\begin{definition}[index of $\mQ$]\label{defnq} Define the \emph{index} of $\mQ$, denoted by $n_\mQ$, as follows:
\begin{itemize}
\item If $\bar{L}_i > \bar{D}_j$ for all $i,j\in\{1,2\}$ then we define $n_\mQ\coloneqq \bar{L}_1 = \bar{L}_2$.
\item If $\bar{D}_i > \bar{L}_j$ for all $i,j\in\{1,2\}$ then we define $n_\mQ\coloneqq \bar{D}_1 = \bar{D}_2$.
\end{itemize}
\end{definition}
\noindent Note that if neither of both cases in Definition \ref{defnq} occurs, $n_\mQ$ is not defined. Now we are ready to state our main theorem.

Denote by $\square p_1q_1q_2p_2 $ the geodesic quadrilateral\footnote{A quadrilateral in which all sides are geodesics.} on the sphere with its interior included, and define its $\rho$-neighborhood by:
$$ B_{\rho}(\square p_1q_1q_2p_2)=\{p\in\mathbb{S}^2 ; d(p,\square p_1q_1q_2p_2)<\rho \} .$$
\noindent Denote the closure of $B_{\rho}(\square p_1q_1q_2p_2)$ by $\bar{B}_{\rho}(\square p_1q_1q_2p_2)$
\begin{theorem}[main theorem]\label{mainTheorem}
\label{condq1q2} Let $\mQ\in\SO$ and $\rho_0\in (0,\frac{\pi}{4})$. Assume that the following conditions are satisfied. 
\begin{enumerate}
\item $\langle q_1,e_2\rangle >0$ and $\langle q_2,e_2\rangle >0$.
\item $\min\{D_1,D_2\}>2\rho_0 .$
\item $\square p_1q_1q_2p_2$ is a convex set.
\item There exists a $\delta>0$ satisfying that for all $\rhot\in [\rho_0,\rho_0+\delta)$, $\langle\qta, e_2\rangle >0$ and $\langle\qtb, e_2\rangle >0 $ and there exists a CSC curve (defined in Definition \ref{defcsc}) in $\bar{\mathcal{L}}_{\rhot}(\mI,\mQ)$ such that its image is contained in $B_{\rhot}(\square \pta\qta\qtb\ptb)$, where
\begin{align*}
\pta = (\cos\rhot,0,\sin\rhot),& \quad  \ptb = (\cos\rhot,0,-\sin\rhot), \\
 \qta = \mQ(\cos\rhot,0,\sin\rhot)& \quad \text{and} \quad \qtb = \mQ(\cos\rhot,0,-\sin\rhot).
\end{align*}
\end{enumerate}
Compare the values $\bar{L}_i$ and $\bar{D}_j$, for $i,j\in\{1,2\}$:
\begin{itemize}
\item If $\bar{L}_i > \bar{D}_j$ for $i,j\in\{1,2\}$ then there is an application $F:\mathbb{S}^{n_\mQ}\to\Lspace$ such that $[F]\in\ho_{n_\mQ}\big(\Lspace\big)$ is non-trivial, but $[\boldsymbol{i}\circ F]\in\ho_{n_\mQ}\big(\mathcal{I}(\mQ)\big)$ is trivial, where $n_\mQ = \bar{L}_1 = \bar{L}_2$. In particular, the inclusion $\boldsymbol{i}:\Lspace\hookrightarrow \mathcal{I}(\mQ)$ is not a homotopy equivalence.
\item If $\bar{D}_i > \bar{L}_j$ for $i,j\in\{1,2\}$ then there is an application $F:\mathbb{S}^{n_\mQ}\to\Lspace$ such that $[F]\in\ho_{n_\mQ}\big(\Lspace\big)$ is non-trivial, but $[\boldsymbol{i}\circ F]\in\ho_{n_\mQ}\big(\mathcal{I}(\mQ)\big)$ is trivial, where $n_\mQ = \bar{D}_1 = \bar{D}_2$. In particular the inclusion $\boldsymbol{i}:\Lspace\hookrightarrow \mathcal{I}(\mQ)$ is not a homotopy equivalence.
\end{itemize}
\end{theorem}

\begin{exmp} Given $\theta\in (0,\pi)$, let
$$\mQ = \left(\begin{array}{ccc} \cos\theta & -\sin\theta & 0 \\
\sin\theta & \cos\theta & 0 \\
0 & 0 & 1 
\end{array} \right) \in\SO .$$
For all $\rho_0\in \left(0,\frac{\pi}{4}\right)$, the set $\Lspace$ satisfies the hypothesis of the main theorem. In fact, the length-minimizing curve given by $\gamma_0(t)=(\cos (t),\sin (t), 0)$ for $t\in [0,\theta]$ obviously lies inside the quadrilateral $\square p_1q_1q_2p_2$. A direct computation shows that $q_1=(\cos\rho_0\cos\theta,\cos\rho_0\sin\theta,\sin\rho_0)$ and $q_2=(\cos\rho_0\cos\theta,\cos\rho_0\sin\theta,-\sin\rho_0)$. Thus:
$$ L_1=L_2=\arccos(\cos^2\rho_0\cos\theta+\sin^2\rho_0)\quad \text{and}\quad D_1=D_2=\arccos(\cos^2\rho_0\cos\theta-\sin^2\rho_0) .$$

\noindent Since these properties are invariant in a neighborhood of $\mQ$, the main theorem is valid for an open set in $\SO$ containing $\mQ$ given above.
\end{exmp}

A particular case of the main theorem follows from the example given above:

\begin{theorem}[a special case]\label{mainTheorem2}
Let $\rho_0\in (0,\frac{\pi}{4})$, $\theta\in(0,\pi)$ and $\mQ\in\SO$ given by
$$\mQ = \left(\begin{array}{ccc} \cos\theta & -\sin\theta & 0 \\
\sin\theta & \cos\theta & 0 \\
0 & 0 & 1 
\end{array} \right).$$
Compare the values $\bar{L}_i$ and $\bar{D}_j$, for $i,j\in\{1,2\}$:
\begin{itemize}
\item If $\bar{L}_i > \bar{D}_j$ for $i,j\in\{1,2\}$ then there is an application $F:\mathbb{S}^{n_\mQ}\to\Lspace$ such that $[F]\in\ho_{n_\mQ}\big(\Lspace\big)$ is non-trivial, but $[\boldsymbol{i}\circ F]\in\ho_{n_\mQ}\big(\mathcal{I}(\mQ)\big)$ is trivial, where $n_\mQ = \bar{L}_1 = \bar{L}_2$. In particular, the inclusion $\boldsymbol{i}:\Lspace\hookrightarrow \mathcal{I}(\mQ)$ is not a homotopy equivalence.
\item If $\bar{D}_i > \bar{L}_j$ for $i,j\in\{1,2\}$ then there is an application $F:\mathbb{S}^{n_\mQ}\to\Lspace$ such that $[F]\in\ho_{n_\mQ}\big(\Lspace\big)$ is non-trivial, but $[\boldsymbol{i}\circ F]\in\ho_{n_\mQ}\big(\mathcal{I}(\mQ)\big)$ is trivial, where $n_\mQ = \bar{D}_1 = \bar{D}_2$. In particular the inclusion $\boldsymbol{i}:\Lspace\hookrightarrow \mathcal{I}(\mQ)$ is not a homotopy equivalence.
\end{itemize}
\end{theorem}

The proof of Theorem \ref{mainTheorem} is divided into three parts, in Sections \ref{sec:mapF}, \ref{sec:mapG} and \ref{sec:gcircf}. These sections are, respectively, definition of $F$, definition of $G$ and the proof of the essential properties of $G\circ F$.

\newpage
\section{Defining the map $F:\mathbb{S}^{n_\mQ} \to \Lspace$} \label{sec:mapF}
In this section we will define the map $F:\mathbb{S}^{n_\mQ} \to \Lspace$. First we introduce some needed definitions and notations, then we construct a map $\bar{F}:\mathbb{R}^{n_\mQ}\to \bar{\mathcal{L}}_{\rhot}(\mQ)$, for some $\rhot>\rho_0$, and lastly we modify the map $\bar{F}$ into the desired $F$. From now on, the parameter $s$ will no longer be used to denote exclusively the arc-length of the curve.

\subsection{Preliminary definitions, notations and lemmas} 

Intuitively, the definition of $\bar{F}:\mathbb{R}^{n_Q}\to\Lspace$ may be thought as continuously deforming a curve (``elastic band'') that is constrained in the middle of $n_Q$ pairs of control circles (``reels''). By moving the position of these reels (refer Figures \ref{fig:auxcurves} and \ref{fig:auxcurves2}), the elastic band will follow the movement of reels. Here we will formalize this concept.

Recall that a basis $\{v_1,v_2,v_3\}$ of $\mathbb{R}^3$ is positive if:
$$\det\left(\begin{array}{ccc} | & | & | \\
v_1 & v_2 & v_3 \\
| & | & | 
\end{array}\right)>0 .$$
Given an $ a \in \mathbb{S}^2$, there are two (non-unique) vectors $u_1(a),u_2(a)\in\mathbb{S}^2$ such that $\{ a,u_1(a),u_2(a)\}$ forms a positive orthonormal basis of $\mathbb{R}^3$. Denote an \emph{oriented circle} with center $a\in\mathbb{S}^2$ and radius $\rho\in (0,\pi)$ as
\begin{equation}\label{circle}
\zeta_{a,\rho}(s) \coloneqq  (\cos \rho) \cdot a + (\sin \rho) \cdot\big[\cos{s}\cdot u_1(a)-\sin{s}\cdot u_2(a)\big],\quad s\in \mathbb{R}.
\end{equation}

We start defining two families of circles with the centers at $\tilde{p}_1$ or $\tilde{p}_2$, denoted by $\spL_i , \spR_i : \mathbb{R}\to \mathbb{S}^2$, with $i\in\{1,2,\ldots,n_\mQ\}$, which describe the positions of centers of ``control circles'' (reels). We fix an orientation on $\text{T}\mathbb{S}^2$ by setting $\{(0,1,0),(0,0,1)\}$ as a positive basis in $\text{T}_{(1,0,0)}\mathbb{S}^2$. Under the induced orientation on $\mathbb{S}^2\subset\mathbb{R}^3$, the normal vector points outwards. 

\begin{definition}\label{defcsc}
Let $\rho\in \left(0,\frac{\pi}{2}\right)$ and $\mP,\mQ\in\SO$. Consider the space $\bar{\mathcal{L}}_{\rho}(\mP,\mQ)$.
We say that a curve $\gamma$ in $\bar{\mathcal{L}}_{\rho}(\mP,\mQ)$ is of type CSC in $\bar{\mathcal{L}}_\rho(\mP,\mQ)$, if $\gamma$ is concatenation of the following three curves:
$$\gamma(t)=\left\{\begin{array}{ll} \gamma_{1}(t), \quad & t\in [0,t_1] \\
\gamma_{2}(t), \quad & t\in [t_1,t_2] \\
\gamma_{3}(t), \quad & t\in [t_3,1]
\end{array} \right. $$
where both $\gamma_{1}$ and $\gamma_{3}$ are arcs of circles of radius equal to either $\rho$ or $\pi-\rho$, and $\gamma_{2}$ is a segment of geodesic. 
We say that a curve $\gamma$ in $\bar{\mathcal{L}}_{\rho}(\mP,\mQ)$ is of type CCC if $\gamma$ is a concatenation of three arcs of circles of radius equal to either $\rho$ or $\pi-\rho$. 
\end{definition}

The term CSC stands for ``Curved-Straight-Curved'', meaning that the referred curve is composed by concatenation of 3 curves, the first one is an arc with constant geodesic curvature with modulus equal to $\kappa_0$, then comes the second which is a Geodesic segment, finally the last curve is again an arc with modulus equal to $\kappa_0$. Note that in Definition \ref{defcsc}, each of the three segments is allowed to have zero length (degenerate). If $\gamma_1$ is degenerate, then we also call the curve $\gamma$ of type SC. If both $\gamma_1$ and $\gamma_3$ are degenerate, we call $\gamma$ of type S, so on. This kind of nomenclature is commonly used in the studies of Dubins' curves (see, for example, \cite{ayala}).

For $n_\mQ$ an even number, consider $\varsigma = \min\{D_1,D_2\} = \min\{d(p_1,q_2),d(p_2,q_1)\}$. For $n_\mQ$ an odd number, consider $\varsigma = \min\{L_1,L_2\} = \min\{d(p_1,q_1),d(p_2,q_2)\}$. Then take 
$$\delta_0 = \frac{\varsigma - (2n_\mQ+2)\rho_0}{2n_\mQ+3}.$$ 
\noindent By the definition of $n_\mQ$, $\delta_0>0$. The purpose of the choice of $\delta_0$ is that, for $\rhot\in(\rho_0,\rho_0+\delta_0]$, it holds that $(2n_\mQ+2)\rhot < \varsigma$. This will allow us to construct critical curves of index $n_\mQ$ by using arcs of circles with radius $\geq\rhot$ in $\bar{\mathcal{L}}_{\rhot}(\mI,\mQ)\subset\Lspace$.

The following theorem is an adapted version of a part of a theorem proved by F. Monroy-Pérez (Theorem 6.1 in \cite{monroy}). This theorem was proven for the particular case in which the radius $\rho=\frac{\pi}{4}$. The original proof can be adapted to any $\rho\in \left(0,\frac{\pi}{4}\right)$.

\begin{theorem}\label{perez}  Let $\rho\in \left(0,\frac{\pi}{2}\right]$ and $\kappa=\cot\rho$. Every length-minimizing curve in $\bar{\mathcal{L}}_\rho(\mI,\mQ)$ is a concatenation of at most three pieces of arcs with constant curvature equal to $+\kappa$, $-\kappa$ and $0$. Moreover, if the length-minimizing curve contains a geodesic arc, then it is of the form CSC.
\end{theorem}

As a corollary of this theorem, we obtain:

\begin{corollary}\label{cor15} Let $p_1,p_2,q_1,q_2$ as defined previously. If $n_\mQ\geq 1$ then there exists a $\delta_1>0$ such that for every $\rhot\in (\rho_0,\rho_0+\delta_1]$, every length-minimizing curve in $\bar{\mathcal{L}}_\rhot(\mI,\mQ)$ is of type CSC.
\end{corollary}
\begin{proof}
Suppose by contradiction that for every $\delta_1>0$ there exists a $\rhot\in (\rho_0,\rho_0+\delta_1]$ such that there exists a CCC curve in $\bar{\mathcal{L}}_\rhot(\mI,\mQ)$. Assume, without loss of generality, that the first arc of this curve has positive curvature. Consider the points $\pta,\ptb,\qta,\qtb\in\mathbb{S}^2$:
\begin{align*}
\pta = (\cos\rhot,0,\sin\rhot),& \quad  \ptb = (\cos\rhot,0,-\sin\rhot), \\
 \qta = \mQ(\cos\rhot,0,\sin\rhot)& \quad \text{and} \quad \qtb = \mQ(\cos\rhot,0,-\sin\rhot).
\end{align*}

Let $c_2$ be the center of the second arc of the CCC curve, note that the centers of the first and third arcs of circles are $\pta$ and $\qta$, respectively. This and the triangular inequality implies:
$$ d(\pta,\qta)\leq d(\pta,c_2)+d(c_2,\qta) = 4\rhot.$$
In the equation above, take the limit $\rhot$ to $\rho_0$. Note that $\pta$ and $\qta$ converge to $p_1$ and $q_1$ respectively. Thus 
\begin{equation}\label{ineqdist}
L_1=d(p_1,q_1)\leq 4\rho_0.
\end{equation}

On the other hand, if $n_\mQ\geq 1$ is an odd number then $n_\mQ=\bar{L}_1\geq 1$. This implies $L_1> 4\rho_0$ (see graph of $L_1$ in Figure \ref{fig:graph}). For $n_\mQ\geq 2$ an even number, then $n_\mQ=\bar{D}_1\geq 2$, this and the triangular inequality imply 
$$L_1 = d(p_1,q_1) \geq |D_1-d(q_1,q_2)| > |6\rho_0-2\rho_0| =  4\rho_0 .$$
So in both cases we obtain $L_1 > 4\rho_0$, contradicting Inequality (\ref{ineqdist}). 
\end{proof}

\begin{corollary}\label{cor16} Let $\mQ\in \SO$ be such that $\langle q_1 , e_2\rangle > 0$, $\langle q_2 , e_2\rangle >0$ and $n_\mQ = 0$. Then there exists a $\delta_2>0$ such that for every $\rhot\in (\rho_0,\rho_0+\delta_2]$ the length-minimizing curve in $\bar{\mathcal{L}}_{\rhot}(\mI,\mQ)$ is of type CSC.
\end{corollary}

\begin{proof}
Note that $n_\mQ=0$ implies that: 
\begin{equation} \label{eqdist2}
d(p_1,q_2),d(p_2,q_1)>2\rho_0.
\end{equation}

Suppose, by contradiction, that the length-minimizing curve is a CCC curve. We will construct another curve whose length is less than the original CCC curve. It is easy to check that a length-minimizing $CCC$ curve satisfies the following properties:
\begin{enumerate}
\item The first and the third arcs have the same curvature, while the second arc has the opposite curvature.
\item If a $\text{C}_{\theta_1}\text{C}_{\theta_2}\text{C}_{\theta_3}$ curve is length-minimizing then $\theta_2>\pi$, where $\theta_1, \theta_2$ and $\theta_3$ denote the angles of the corresponding arcs of circles.
\end{enumerate}
Denote the three arcs of CCC curve by $\gamma_1$, $\gamma_2$ and $\gamma_3$, respectively. Also denote their correspondent circles by $C_1$, $C_2$ and $C_3$, and centers by $c_1$, $c_2$ and $c_3$, respectively. Suppose, without loss of generality, that $\gamma_1$ has positive curvature. By Item (2), the center $c_2$ lies in one of the hemispheres delimited by the geodesic passing through the centers $c_1$ and $c_3$. 

We consider another circle $\tilde{C}_2$ of the same radius also tangent to $C_1$ and $C_3$ whose the center lies in the other hemisphere. By Equation (\ref{eqdist2}), the CCC curve formed by concatenation of an arc of $C_1$, followed by an arc of $\tilde{C}_2$ and an arc of $C_3$ is well defined and strictly shorter than the original curve. It is a contradiction.
\end{proof}

For $n_\mQ=0$, define $F:\mathbb{S}^0\to\Lspace$ as $F(-1)=\gamma_0$ and $F(+1)=\gamma_0^{\left[0.5\#2\right]}$, where $\gamma_0:[0,1]\to\mathbb{S}^2$ is the length-minimizing CSC curve and $\gamma_0^{\left[0.5\#2\right]}$ is the curve $\gamma_0$ with two loops added at the instant $t=0.5$.

Fix $\delta_1$, $\delta_2$ given in Corollaries \ref{cor15} and \ref{cor16}. Also fix $\delta_3$ from the hypothesis of the main theorem. From now on, we fix a $\rhot\in \big(\rho_0,\rho_0+\min\{\delta_0,\delta_1,\delta_2,\delta_3\}\big]$ and assume $n_\mQ\geq 1$. As checked previously, for $n_\mQ\geq 1$, the hypothesis of Theorem \ref{mainTheorem} guarantees that the length-minimizing curve $\gamma_0\in\bar{\mathcal{L}}_{\tilde{\rho}}(\mI,\mQ)$ is of type CSC. We fix a length-minimizing CSC curve and denote it by: 
$$\gamma_0(t)=\left\{\begin{array}{ll} \gamma_{0,1}(t), \quad & t\in [0,t_1] \\
\gamma_{0,2}(t), \quad & t\in [t_1,t_2] \\
\gamma_{0,3}(t), \quad & t\in [t_3,1]
\end{array} \right. $$
where both $\gamma_{0,1}$ and $\gamma_{0,3}$ are arcs of circles of radius $\rhot$ and $\gamma_{0,2}$ is a segment of geodesic.
\\[.5em]
\textbf{A simpler construction choice for $n_\mQ$ odd case:}
 In this case, we construct $\tilde{F}:\mathbb{S}^{n_\mQ}\to\mathcal{L}_{\rho_0}(\tilde{\mP},\tilde{\mQ})$ into the space of curves that start at the frame $\tilde{\mP}=\mathfrak{F}_{\gamma_{0,2}}(t_1)$ and end at the frame $\tilde{\mQ}=\mathfrak{F}_{\gamma_{0,2}}(t_2)$. Then afterwards concatenate these curves with $\gamma_{0,1}$ and $\gamma_{0,3}$ at the beginning and the end, respectively. From this concatenation we obtain the desired map $F:\mathbb{S}^{n_\mQ}\to\Lspace$. So for $n_\mQ$ an odd number, we may suppose, without loss of generality that $\mQ$ is of form:
$$\mQ = \left(\begin{array}{ccc} \cos\theta & -\sin\theta & 0 \\
\sin\theta & \cos\theta & 0 \\
0 & 0 & 1 
\end{array} \right).$$
\noindent In this case the length-minimizing curve $\gamma_0$ is a geodesic segment. However if $n_\mQ$ is an even number, we follow the construction below.
\\[.5em]
\textbf{General construction for $n_\mQ\geq 1$ (for both even and odd cases):}
Let $\gamma_0$ the length-minimizing CSC curve in $B_{\rhot}(\square p_1q_1q_2p_2) $. We define the auxiliary curves:
$$\gamma_{0,l}(s) = \exp_{\gamma(s)} \big( \rhot \bsm{n}_{\gamma_0}(s) \big) \quad \text{and} \quad \gamma_{0,r}(s) = \exp_{\gamma(s)} \big( -\rhot \bsm{n}_{\gamma_0}(s) \big).$$
\noindent In the above equations, $\exp$ denotes the exponential map $\exp:\textup{T}\mathbb{S}^2\to\mathbb{S}^2$, $(p,v)\mapsto \exp_p(v)$.
Consider the points $\pta,\ptb,\qta,\qtb\in\mathbb{S}^2$:
\begin{align*}
\pta = (\cos\rhot,0,\sin\rhot),& \quad  \ptb = (\cos\rhot,0,-\sin\rhot), \\
 \qta = \mQ(\cos\rhot,0,\sin\rhot)& \quad \text{and} \quad \qtb = \mQ(\cos\rhot,0,-\sin\rhot).
\end{align*}

\noindent We will show two useful lemmas below:

\begin{lemma}\label{timming}
Let $\rho_1 = \min\{d(\pta,\qta),d(\pta,\qtb),d(\ptb,\qta),d(\ptb,\qtb)\}$. If $\rho_1 > 4\rhot$, then for any $\rho\in (2\rhot,\rho_1-2\rhot)$, $i\in\{1,2\}$ and $j\in\{l,r\}$\footnote{$l$ and $r$ in $\{l,r\}$ denote letters.} there is a unique number $s_{\rho,i,j}\in [0,2\pi)$ such that:
\begin{enumerate}
\item $\zeta_{\tilde{p}_i,\rho}(s_{\rho,i,j})\in\img (\gamma_{0,j})$. We denote the point $\zeta_{\tilde{p}_i,\rho}(s_{\rho,i,j})$ as $a_{\rho,i,j}$.
\item $\{\zeta_{\tilde{p}_i,\rho}'(s_{\rho,i,j}),\gamma_{0,j}'(s)\}$ forms a positive basis of $\text{T}_{a_{\rho,i,j}}\mathbb{S}^2$.
\end{enumerate}
\end{lemma} 

\begin{proof} Suppose, without loss of generality, that the parametrization domains for the curves is $[0,1]$. We denote by $\gamma_0$ the length-minimizing curve in $\bar{\mathcal{L}}_\rhot(\mI,\mQ)$, which is of the type CSC. 

For the existence, note that $\gamma_{0,l}(0)=\pta$, $\gamma_{0,l}(1)=\qta$, $\gamma_{0,r}(0)=\ptb$ and $\gamma_{0,r}(1)=\qtb$, so by continuity, the functions $d_{1,l}(s)\coloneqq d(\gamma_{0,l}(s),\pta)$, $d_{2,l}(s)\coloneqq d(\gamma_{0,l}(s),\ptb)$, $d_{1,r}(s)\coloneqq d(\gamma_{0,r}(s),\qta)$ and $d_{2,r}(s)\coloneqq d(\gamma_{0,r}(s),\qtb)$ always have the interval $[2\rhot,\rho_1]$ in its image. Also, these functions are strictly increasing for values of $s$ satisfying $2\rhot<d_{i,j}(s)<\rho_1 $ for all $i=1,2$ and $j=l,r$. This implies the uniqueness.
\end{proof}

\begin{lemma}\label{reparametrization} 
For every $\rho\in (2\rhot,\pi-2\rhot)$, let $\zeta_{\pta,\rho},\zeta_{\ptb,\rho}:\left[-\frac{\pi}{2},\frac{\pi}{2}\right]\to\mathbb{S}^2$ be the circles defined in Equation (\ref{circle}) by taking 
$$u_1(\pta) = u_1(\ptb) = (0,1,0), $$
$$u_2(\pta) = (-\sin\rhot,0,\cos\rhot)\quad \text{and} \quad u_2(\ptb) = (\sin\rhot,0,\cos\rhot) .$$ 
Then there exist exactly 2 distinct reparametrizations of $\zeta_{\ptb,\rho}$, which we denote $\zeta_{\ptb,\rho}^1$ and $\zeta_{\ptb,\rho}^2 : \left[-\frac{\pi}{2},\frac{\pi}{2}\right]\to\mathbb{S}^2$ (see Figure \ref{fig:twoparam}), such that:
$$d\left(\zeta_{\pta,\rho}(s),\zeta_{\ptb,\rho}^i(s)\right)=2\rhot \quad\text{ for all $s\in \left[-\frac{\pi}{2},\frac{\pi}{2}\right]$ and $i=1,2$.}$$
Moreover,
\begin{enumerate}
\item For $\zeta^1_{\ptb,\rho}$ there are $s^1_1,s^2_2\in \left[-\frac{\pi}{2},\frac{\pi}{2}\right]$, such that $s^1_1<s^1_2$,  $\zeta_{\pta,\rho}(s^1_1)\in \img(\zeta_{\ptb,\rho})$ and $\zeta_{\ptb,\rho}^1(s^1_2)\in \img(\zeta_{\pta,\rho})$. 
\item In the same way, for $\zeta^2_{\ptb,\rho}$ there are $s^2_1,s^2_2\in \left[-\frac{\pi}{2},\frac{\pi}{2}\right]$, such that $s^2_2<s^2_1$, $\zeta_{\ptb,\rho}^2(s^2_2)\in \img(\zeta_{\pta,\rho})$ and $\zeta_{\pta,\rho}(s^2_1)\in \img(\zeta_{\ptb,\rho})$. 
\end{enumerate}
\end{lemma}

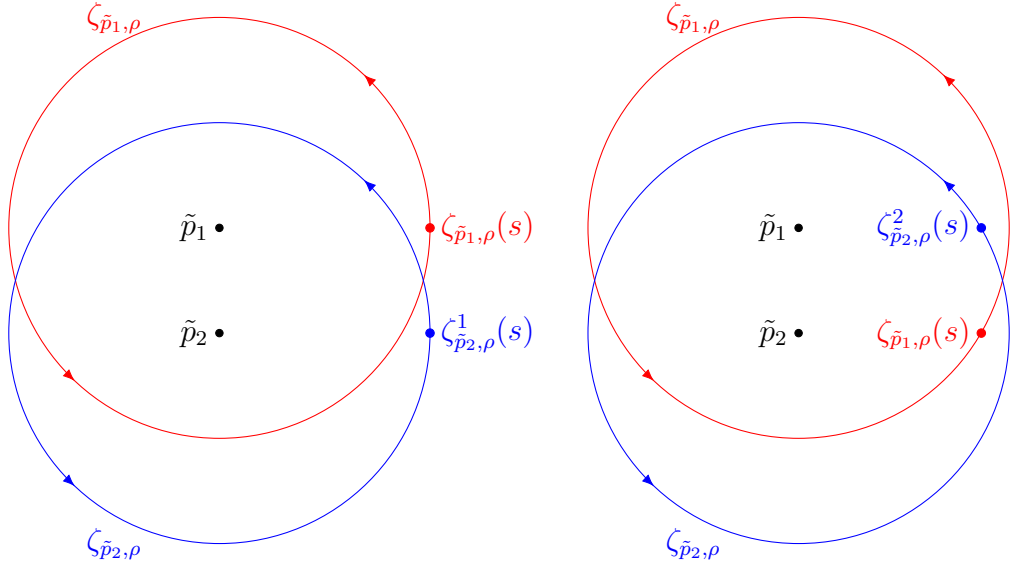
\begin{figure}
\begin{tikzpicture}
\begin{scope}[scale=0.7,xshift=5cm,decoration={markings, mark=at position 0.13 with {\arrow {Latex}}, mark=at position 0.63 with {\arrow {Latex}}}]
\draw[fill=black] (0,1) circle (.07) node[left] {$\tilde{p}_1$} ;
\draw[fill=black] (0,-1) circle (.07) node[left] {$\tilde{p}_2$};
\draw[red,postaction={decorate}] (0,1) circle (4);
\draw[blue,postaction={decorate}] (0,-1) circle (4);
\draw[red,fill] (3.47,-1) circle (.08)  node[left] {$\zeta_{\tilde{p}_1,\rho}(s)$};
\draw[blue,fill] (3.47,1) circle (.08)  node[left] {$\zeta_{\tilde{p}_2,\rho}^2(s)$};
\node[red] at (-2,5) {$\zeta_{\tilde{p}_1,\rho}$} ;
\node[blue] at (-2,-5) {$\zeta_{\tilde{p}_2,\rho}$} ;
\end{scope}
\begin{scope}[scale=0.7,xshift=-6cm,decoration={markings, mark=at position 0.13 with {\arrow {Latex}}, mark=at position 0.63 with {\arrow {Latex}}}]
\draw[fill=black] (0,1) circle (.07) node[left] {$\tilde{p}_1$} ;
\draw[fill=black] (0,-1) circle (.07) node[left] {$\tilde{p}_2$};
\draw[red,postaction={decorate}] (0,1) circle (4);
\draw[blue,postaction={decorate}] (0,-1) circle (4);
\draw[red,fill] (4,1) circle (.08) node[right] {$\zeta_{\tilde{p}_1,\rho}(s)$} ;
\draw[blue,fill] (4,-1) circle (.08)  node[right] {$\zeta_{\tilde{p}_2,\rho}^1(s)$} ;
\node[red] at (-2,5) {$\zeta_{\tilde{p}_1,\rho}$} ;
\node[blue] at (-2,-5) {$\zeta_{\tilde{p}_2,\rho}$} ;
\end{scope}
\end{tikzpicture}
\caption{These are two reparametrizations mentioned in Lemma \ref{reparametrization}.  The left-hand side image represents the reparametrization $\zeta_{\tilde{p}_2,\rho}^1$ and the right-hand side represents $\zeta_{\tilde{p}_2,\rho}^2$.}
\label{fig:twoparam}
\end{figure}

\begin{proof} For each point of type $a=\zeta_{\pta,\rho}(t)$ with $t\in\left(-\frac{\pi}{2},\frac{\pi}{2}\right)$, we draw a circle of radius $2\rhot$ centered at $a$ (measured in $\mathbb{S}^2$). We denote this circle by $\zeta_{a,2\rhot}$. Since the points $\pta,\ptb,a$ do not lie in the same geodesic, by triangular inequality, we have $d(a,\ptb)<d(a,\pta)+d(\pta,\ptb)=\rho+2\rhot$. So the circle $\zeta_{a,2\rhot}$ intercepts $\zeta_{\ptb,\rho}$ at 2 distinct points, namely: 
$$ a_1 = \zeta_{\ptb,\rho}(s_1(t)) \quad \text{and} \quad a_2 = \zeta_{\ptb,\rho}(s_2(t)), \text{ with $s_1(t)<s_2(t)$.}$$
Since $t\in\left(-\frac{\pi}{2},\frac{\pi}{2}\right) $ is arbitrary, we define the following reparametrizations: 
$$\zeta^1_{\ptb,\rho}(s)=\zeta_{\ptb,\rho}(s_1(t)) \text{ and } \zeta^2_{\ptb,\rho}(s)=\zeta_{\ptb,\rho}(s_2(t))\text{ for } t\in\left(-\frac{\pi}{2},\frac{\pi}{2}\right) .$$ 
At the extremities $t=\pm \frac{\pi}{2}$, for $i=1,2$, we set: 
\begin{align*}
\zeta^i_{\ptb,\rho}\left(-\frac{\pi}{2}\right) & =\big(\cos(\rho+\rhot),0,-\sin(\rho+\rhot)\big),\\
\zeta^i_{\ptb,\rho}\left(\frac{\pi}{2}\right) & =\big(\cos(\rho+\rhot),0,\sin(\rho+\rhot)\big).
\end{align*}
By construction, it is up to a direct computation to verify that the above reparametrizations satisfy the properties of the lemma.
\end{proof}

\subsection{Definition of $\mQ_i:\mathbb{R}\to\SO$ for $i\in\{0,1,\ldots,n_\mQ+1\}$}

To define $F$ we need to construct certain curves which will be in the image of $F$. These curves are made from concatenation of several arcs of circles. Here we describe the curves $\spL_i$ and $\spR_i$ which denote the positions of the centers of these circles (see Figure \ref{fig:auxcurves} below). We use Lemmas \ref{timming} and \ref{reparametrization} to define:
\begin{gather*}
\spL_1(s) \coloneqq \left\{\begin{array}{ll}  
\zeta_{\ptb,2\rhot}^1(s-s_{2\rhot,2,r}), & \quad  \text{ if $s\leq 0$.} \\
\pta, & \quad  \text{ if $s\geq 0$.}
\end{array}\right. \\
\spR_1(s) \coloneqq \left\{\begin{array}{ll}  
\ptb, & \quad  \text{ if $s\leq 0$.} \\
\zeta_{\pta,2\rhot}(s-s_{2\rhot,1,r}), & \quad  \text{ if $s\geq 0$.}  
\end{array}\right. 
\end{gather*}

\noindent In the definition above, the number $s_{2\rhot,1,r}$ is given by Lemma \ref{timming} and the curve $\zeta_{\ptb,2\rhot}^1$ comes from Lemma \ref{reparametrization}. It follows from the definitions that $d\left(\spL_1(s),\spR_1(s)\right)=2\rhot$ for all $s\in\mathbb{R}$.
 
Next, for each even number $2\leq i\leq n_\mQ$ , we use Lemmas \ref{timming} and \ref{reparametrization} to define:
\begin{gather*}
\spL_i(s) \coloneqq   
\zeta_{\pta,2i\rhot}(s-s_{2i\rhot,1,l}).\\
\spR_i(s) \coloneqq \left\{\begin{array}{ll} \zeta_{\ptb,2i\rhot}^1(s-s_{2i\rhot,2,l}), \quad & s\in\bigcup_{n\in\mathbb{Z}} J_{2n}.  \\
\zeta_{\ptb,2i\rhot}^2(s-s_{2i\rhot,2,l}), \quad & s\in\bigcup_{n\in\mathbb{Z}}J_{2n+1}.
\end{array}
\right.
\end{gather*}

\noindent Here $J_n=\left[n\pi-\frac{\pi}{2}+s_{2i\rhot,2,l},n\pi+\frac{\pi}{2}+s_{2i\rhot,2,l}\right]$. Finally, for $3\leq i\leq n_\mQ$ an odd number, we use Lemmas \ref{timming} and \ref{reparametrization} to define:
\begin{gather*}
\spL_i(s) \coloneqq \left\{\begin{array}{ll} \zeta_{\ptb,2i\rhot}^2(s-s_{2j\rhot,2,r}) \quad & s\in J_0 \\
\zeta_{\ptb,2i\rhot}^1(s-s_{2i\rhot,2,r}) \quad & s\in\bigcup_{n\in\mathbb{Z}^*}J_{2n} \\
\zeta_{\ptb,2i\rhot}^2(s-s_{2i\rhot,2,r}) \quad & s\in\bigcup_{n\in\mathbb{Z}}J_{2n+1}
\end{array}
\right.\\ 
\spR_i(s) \coloneqq   
\zeta_{\pta,2i\rhot}(s-s_{2i\rhot,1,r}).
\end{gather*}

\noindent Here $J_n=\left[n\pi-\frac{\pi}{2}+s_{2i\rhot,2,r},n\pi+\frac{\pi}{2}+s_{2i\rhot,2,r}\right]$. Again, from Lemma \ref{reparametrization}, the spherical distance $d\left(\spL_i(s),\spR_i(s)\right)=2\rhot$ for all $s\in\mathbb{R}$ and $i\in\mathbb{N}$. This means that if we draw a circle with curvature $+\kappa_0$ centered at $\spL_i(s)$ and another circle with curvature $-\kappa_0$ centered at $\spR_i(s)$, these circles touch each other at a unique point with common tangent vector, we denote the common Frenet frame at that point by $\mQ_i(s)$, with $s\in\mathbb{R}$. Thus we have defined a family of continuous applications:
$$ \mQ_i:\mathbb{R}\to\SO , \quad \text{with $i\in\{1,2,\ldots ,n_\mQ\}$.} $$

\noindent We also define $\mQ_0,\mQ_{n_\mQ+1}:\mathbb{R}\to\SO $ with $\mQ_0\equiv \mI$ and $\mQ_{n_\mQ+1}\equiv \mQ$, where $\mQ$ is the matrix in the definition of $\Lspace$. 

Also note that the following relation is an immediate consequence of the definition.
\begin{proposition}\label{propdist}
For each $i\in\{0,1,2,\ldots,n_\mQ\}$, the following inequalities are satisfied
 $$d\big(\spL_i(t_1),\spR_{i+1}(t_2)\big)\geq 2\rhot \quad\text{and}\quad d\big(\spR_i(t_1),\spL_{i+1}(t_2)\big)\geq 2\rhot \quad \forall t_1,t_2\in\mathbb{R}.$$ 
Moreover, for each $t_1\in\mathbb{R}$ and $k\in\mathbb{Z}$, there exist unique $t_2$ and $t_3\in \big[2k\pi,2(k+1)\pi\big)$ such that $d\big(\spL_i(t_1),\spR_{i+1}(t_2)\big)= 2\rhot$ and $d\big(\spR_i(t_1),\spL_{i+1}(t_3)\big)= 2\rhot$.
\end{proposition}

\begin{figure}
\centering
\begin{tikzpicture}
\begin{scope}[scale=.7]
\coordinate (p) at (0,1);
\coordinate (q) at (0,-1);
\draw[dashed,blue] (p) circle (1);
\draw[dashed,red] (q) circle (1); 
\draw[dashed,blue] (4,1) circle (1);
\draw[dashed,red] (4,-1) circle (1); 
\draw[dashed,blue] (5.65,1) circle (1);
\draw[dashed,red] (5.65,-1) circle (1); 
\draw[dashed,blue] (8,1) circle (1);
\draw[dashed,red] (8,-1) circle (1); 
\draw[dashed,blue] (9.85,1) circle (1);
\draw[dashed,red] (9.85,-1) circle (1); 
\draw[red] (p) circle (2);
\draw[blue] (q) circle (2);
\draw[blue] (p) circle (4);
\draw[red] (q) circle (4);
\draw[red] (p) circle (6);
\draw[blue] (q) circle (6);
\draw[blue] (p) circle (8);
\draw[red] (q) circle (8);
\draw[red,dashed] (p) circle (10);
\draw[blue,dashed] (q) circle (10);
\draw[blue,fill=blue] (p) circle (.1);
\draw[red,fill=red] (q) circle (.1);
\draw[fill=black] (9.85,1) circle (.1);
\draw[fill=black] (9.85,-1) circle (.1);
\draw[blue,fill=blue] (4,1) circle (.1);
\draw[red,fill=red] (4,-1) circle (.1);
\draw[blue,fill=blue] (8,1) circle (.1);
\draw[red,fill=red] (8,-1) circle (.1);
\draw[blue,fill=blue] (5.65,1) circle (.1);
\draw[red,fill=red] (5.65,-1) circle (.1);
\draw[thick] (0,0) -- (9.85,0);
\node[above] at (p) {$\spL_1(0)=\tilde{p}_1$};
\node[below] at (q) {$\spR_1(0)=\tilde{p}_2$};
\node[above] at (9.85,1) {$\tilde{q}_1$};
\node[below] at (9.85,-1) {$\tilde{q}_2$};
\node[above] at (0,3) {\color{red}$\zeta_{\tilde{p}_2,4\rhot}$};
\node[below] at (0,-3) {\color{blue}$\zeta_{\tilde{p}_1,4\rhot}$};
\node[below] at (0,3) {\color{red}$\zeta_{\tilde{p}_1,2\rhot}$};
\node[above] at (0,-3) {\color{blue}$\zeta_{\tilde{p}_2,2\rhot}$};
\node[above] at (0,5) {\color{blue}$\zeta_{\tilde{p}_2,6\rhot}$};
\node[below] at (0,-5) {\color{red}$\zeta_{\tilde{p}_1,6\rhot}$};
\node[above] at (0,7) {\color{red}$\zeta_{\tilde{p}_2,8\rhot}$};
\node[below] at (0,-7) {\color{blue}$\zeta_{\tilde{p}_1,8\rhot}$};
\node[above] at (4,1) {$\spL_2(0)$};
\node[below] at (4,-1) {$\spR_2(0)$};
\node[above] at (5.65,1) {$\spL_3(0)$};
\node[below] at (5.65,-1) {$\spR_3(0)$};
\node[above] at (8,1) {$\spL_4(0)$};
\node[below] at (8,-1) {$\spR_4(0)$};

\end{scope}
\end{tikzpicture}
\caption{Illustration of application $\bar{F}$ on $\mathbb{S}^2$. Each red circle represents the trajectory of the center of a circle osculating the curve from the right and each blue circle represents the trajectory of the center of a circle osculating the curve from the left. The ten small dashed circles are control circles on left in blue and on right in red. The two big dashed circles are $\zeta_{\tilde{p}_2,10\rhot}$ (in blue) and $\zeta_{\tilde{p}_1,10\rhot}$ (in red) which cannot be used as trajectory for control circles because they are too close to $\tilde{q}_2$ and $\tilde{q}_1$ respectively. 
So, in this picture the index is $n_\mQ=4$, and there are four pairs of control circles which we can freely move along the trajectories $\spL_i$ and $\spR_i$ described without interfering with each other. The crucial point is that the distance from each blue circle to the red circle with different radius is greater or equal to $2\rhot$.}
\label{fig:auxcurves}
\end{figure}
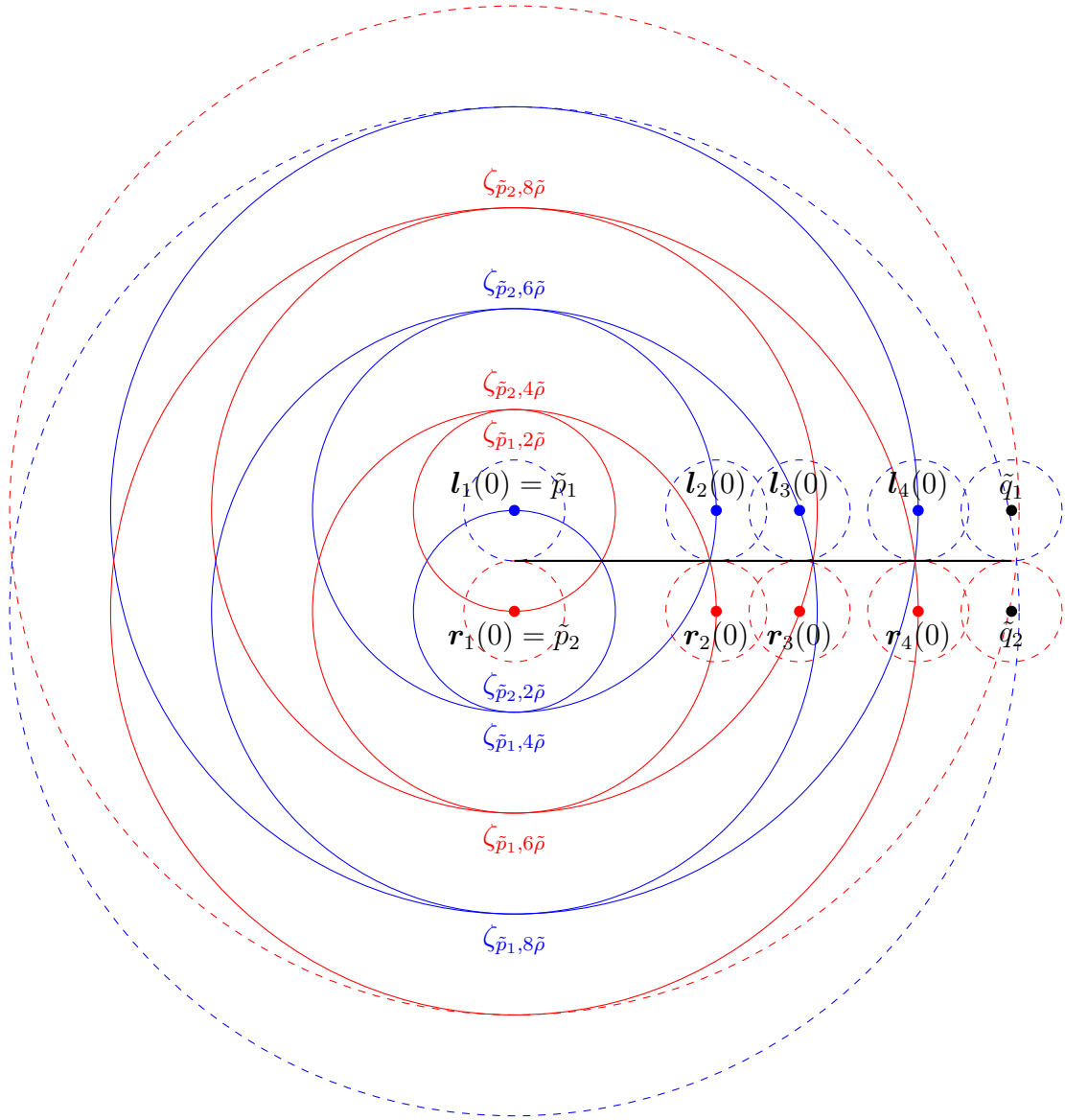

\subsection{Definition of $\bar{F}:\mathbb{R}^{n_\mQ}\to\Lspace$}
Summarizing this subsection, we shall define a map $\bar{F}:\mathbb{R}^{n_\mQ}\to\mathcal{L}_{\rho_0}(\mQ)$. For each $(x_1,x_2,\ldots,x_{n_\mQ})\in\mathbb{R}^{n_\mQ}$, we associate it to $n_\mQ+1$ curves in the following spaces, respectively, 
$$\mathcal{L}_{\rho_0}\big(\mQ_0,\mQ_1(x_1)\big),\mathcal{L}_{\rho_0}\big(\mQ_1(x_1),\mQ_2(x_2)\big),\mathcal{L}_{\rho_0}\big(\mQ_2(x_2),\mQ_3(x_3)\big),\ldots,\mathcal{L}_{\rho_0}\big(\mQ_{n_\mQ}(x_{n_\mQ}),\mQ_{n_\mQ+1}\big).$$
Then we concatenate these $n_\mQ+1$ curves to obtain a curve in $\mathcal{L}_{\rho_0}(\mI,\mQ)$ that will be defined as $\bar{F}(x_1,x_2,\ldots,x_{n_\mQ})$. 

For $a\in\mathbb{S}^2$, $r\in (0,\pi)$, we denote $\zeta_{a,r,\circlearrowleft}$ to be the circle of radius $r$ centered at $a$: 
\begin{equation}\label{circle1}
\zeta_{a,r,\circlearrowleft}(s)=\cos r \cdot a + \sin r\big(\cos s \cdot u_1(a)-\sin s\cdot u_2(a)\big) .
\end{equation}
Analogously, we denote $\zeta_{\pta,r,\circlearrowright}$ as:
\begin{equation}\label{circle2}
\zeta_{a,r,\circlearrowright}(s)=\cos r \cdot a + \sin r\big(\cos s \cdot u_1(a)+\sin s\cdot u_2(a)\big) .
\end{equation}
Given $k\in\mathbb{N}$, we say that a curve \emph{traverses} the circle $\zeta_{a,r,\circlearrowleft}$ $k$ times, if this curve is one of reparametrizations of $\zeta_{a,r,\circlearrowleft}(s)$ with the domain $s\in [0,2k\pi]$. We will use the same term for $\zeta_{a,r,\circlearrowright}$. 
Exclusively in this subsection, we will also use the following notation, for each $M\in\SO$:
\begin{equation*}
\pta(M) = M(\cos\rhot,0,\sin\rhot) \quad\text{and} \quad  \ptb(M) = M(\cos\rhot,0,-\sin\rhot). 
\end{equation*}


\begin{lemma}For each $i\in \{0,1,2,\ldots,n_\mQ\}$ consider $(x_i,x_{i+1})$ and $\mQ_i,\mQ_{i+1} \in\SO$ as defined above. There exists a unique continuous choice, depending on $(x_i,x_{i+1})$, of CSC curve in the space $\mathcal{L}_{\rhot}(\mQ_i,\mQ_{i+1})$, denoted by $\gamma_{0,\rhot}(\mQ_i,\mQ_{i+1})$ satisfying the following property. If $(x_i,x_{i+1})$ is such that 
$$d\big(\pta(\mQ_i(x_i)),\ptb(\mQ_{i+1}(x_{i+1}))\big)=2\rhot \quad \text{or} \quad d\big(\ptb(\mQ_i(x_i)),\pta(\mQ_{i+1}(x_{i+1}))\big)=2\rhot ,$$ 
then $\gamma_{0,\rhot}(\mQ_i,\mQ_{i+1})$ is of type CC. That is, $\gamma_{0,\rhot}(\mQ_i,\mQ_{i+1})$ is concatenation of two arcs of circles of radius $\rhot$.
\end{lemma}

\begin{proof}
 From Theorem \ref{perez} and Proposition \ref{propdist} for $\rho=\rhot$, $\mP=\mQ_i$ and $\mQ=\mQ_{i+1}$, the length-minimizing curve is of type CSC. The continuity can be proven by using the same argument as J. Ayala and H. Rubinstein's argument in \cite{ayala2} for the plane case. The idea is to define a region $\Omega$ that depends continuously on $\mQ_i$ and $\mQ_{i+1}$. $\mQ_i$ and $\mQ_{i+1}$ satisfy the condition D in their article. The length-minimizing curve in $\bar{\mathcal{L}}_\rhot(\mQ_i,\mQ_{i+1})$ can be verified to lie in $\Omega$, is unique and continuous.

This argument is similar to the demonstration of Corollary \ref{cor16}.
\end{proof}

For each $i\in\{0,1,2,\ldots,n_\mQ\}$ we define a curve $\alpha_i(x)\in\mathcal{L}_{\rho_0}\big(\mQ_i(x_i),\mQ_{i+1}(x_{i+1})\big)$ by following the construction below

\begin{lemma}[Definition of $\alpha_i$'s and its properties] \label{segment} For each $i\in\{0,1,\ldots,n_\mQ\}$ and each $(x_1,x_2,\ldots,x_{n_\mQ})\in\mathbb{R}^{n_\mQ}$, there exist real functions $\tilde{x}_{i+1}^1,\tilde{x}_{i+1}^2:\mathbb{R}\to\mathbb{R}$, continuous functions $s_0$, $s_1$, $s_2:\mathbb{R}^2\times\{0,1,\ldots,n_\mQ\}\to\mathbb{R}$ and a curve $\alpha_i$ satisfying the following properties. When $i$ is even \footnote{For simplicity, we denote $s_j(x_i,x_{i+1},i)=s_j$ for $j=1,2,3$, $\tilde{x}_{i+1}^j(x_i)=\tilde{x}_{i+1}^j$ for $i=1,2$, $\mQ_i(x_i)=\mQ_i$ and $\mQ_{i+1}(x_{i+1})=\mQ_{i+1}$.}:
\begin{gather*}
 \text{ If $x_{i+1}\leq \tilde{x}_{i+1}^1$,} \quad \alpha_i(s)=\left\{\begin{array}{ll} \zeta_{\ptb(\mQ_i),\rhot,\circlearrowright}(s), & s\in [0,s_0]. \\
\zeta_{\ptb,(2i-1)\rhot,\circlearrowright}(s), & s\in [s_0,s_1].\\
\zeta_{\pta(\mQ_{i+1}),\rhot,\circlearrowleft}(s), & s\in [s_1,s_2].
\end{array} \right. \\
 \text{ If $\tilde{x}_{i+1}^1\leq x_{i+1} \leq \tilde{x}_{i+1}^2$,}\quad\alpha_i= \gamma_{0,\rhot}(\mQ_i,\mQ_{i+1}).\\
 \text{ If $x_{i+1}\geq \tilde{x}_{i+1}^2$,}\quad\alpha_i(s)=\left\{\begin{array}{ll} \zeta_{\pta(\mQ_i),\rhot,\circlearrowleft}(s), & s\in [0,s_0]. \\
\zeta_{\pta,(2i-1)\rhot,\circlearrowleft}(s), & s\in [s_0,s_1].\\
\zeta_{\ptb(\mQ_{i+1}),\rhot,\circlearrowright}(s), & s\in [s_1,s_2].
\end{array} \right. 
\end{gather*}

When $i$ is odd:
\begin{gather*}
\text{ If $x_{i+1}\leq \tilde{x}_{i+1}^1$,}\quad\alpha_i(s)=\left\{\begin{array}{ll} \zeta_{\ptb(\mQ_i),\rhot,\circlearrowright}(s), & s\in [0,s_0]. \\
\zeta_{\pta,(2i-1)\rhot,\circlearrowright}(s), & s\in [s_0,s_1].\\
\zeta_{\pta(\mQ_{i+1}),\rhot,\circlearrowleft}(s), & s\in [s_1,s_2].
\end{array} \right. \\
\text{ If $ \tilde{x}_{i+1}^1\leq x_{i+1} \leq \tilde{x}_{i+1}^2$,}\quad\alpha_i= \gamma_{0,\rhot}(\mQ_i,\mQ_{i+1}).  \\
\text{ If $x_{i+1}\geq \tilde{x}_{i+1}^2$.}\quad\alpha_i(s)=\left\{\begin{array}{ll} \zeta_{\pta(\mQ_i),\rhot,\circlearrowleft}(s), & s\in [0,s_0]. \\
\zeta_{\ptb,(2i-1)\rhot,\circlearrowleft}(s), & s\in [s_0,s_1].\\
\zeta_{\ptb(\mQ_{i+1}),\rhot,\circlearrowright}(s), & s\in [s_1,s_2].
\end{array} \right.
\end{gather*}
Moreover, the parameter $s$ in each case above is chosen such that $\alpha_i(0)=\mQ_i \cdot e_1$, $\alpha_i(s_2)=\mQ_{i+1} \cdot e_1$. $\alpha_i(s_0)$ and $\alpha_i(s_1)$ are well defined and continuous with respect to the pair $(x_i,x_{i+1})$. In other words, the following function is continuous:
$$F_{i+1}:(x_i,x_{i+1})\in\mathbb{R}^2\mapsto \alpha_i\in\mathcal{L}_{\rho_0}\big(\mQ_i(x_i),\mQ_{i+1}(x_{i+1})\big).$$
\end{lemma}

Moreover, during the proof of the lemma above, we will also verify some of properties listed on the construction below.
\begin{construction}[A more detailed description of $\alpha_i$'s]\label{propsegm}
\upshape For each $i\in\{0,1,\ldots,n_Q\}$, the application $F_{i+1}$ defined in Lemma \ref{segment} satisfies the following relation:
\begin{itemize}
\item $\length\big(F_{i+1}(x_{i+1}+2k\pi)\big)=\length\big(F_{i+1}(x_{i+1})\big)+2k\pi\sin\big((2i+1)\rhot\big)$ for all $k\in\mathbb{N}$, $x_{i+1}\in \left[\tilde{x}_{i+1}^1,\tilde{x}_{i+1}^1+2\pi\right)$.
\item $\length\big(F_{i+1}(x_{i+1}-2k\pi)\big)=\length\big(F_{i+1}(x_{i+1})\big)+2k\pi\sin\big((2i+1)\rhot\big)$ for all $k\in\mathbb{N}$, $x_{i+1}\in \left(\tilde{x}_{i+1}^2-2\pi,\tilde{x}_{i+1}^2\right]$.
\end{itemize}
Furthermore, we describe $\alpha_i$ with more details. In the case that $i$ is even:
\begin{enumerate}
\item For $x_{i+1}\leq \tilde{x}^2_{i+1}$, $\alpha$ is concatenation of the following $3$ curves:
\begin{enumerate}
\item Shortest arc on $\zeta_{\ptb(\mQ_i),\rhot,\circlearrowright}$ that travels from $\mQ_i\cdot e_1$ to the unique point $a_1$ in $\zeta_{\ptb(\mQ_i),\rhot}\cap\zeta_{\ptb,(2i+1)\rhot}$.
\item Arc on $\zeta_{\ptb,(2i-1)\rhot,\circlearrowright}$ that travels from $a_1$ to the unique point $b_1$ in $\zeta_{\ptb,(2i+1)\rhot}\cap\zeta_{\pta(\mQ_{i}),\rhot}$. This arc is concatenation of shortest arc from $a_1$ to $b_1$ and circle that traverses $\zeta_{\ptb,(2i+1)\rhot,\circlearrowright}$ $k$ times, where $k\in\mathbb{N}$ satisfy $(x_{i+1}+2k\pi)\in \left(\tilde{x}_{i+1}^2-2\pi,\tilde{x}_{i+1}^2\right]$.
\item Shortest arc on $\zeta_{\pta(\mQ_{i+1}),\rhot,\circlearrowleft}$ that travels from $b_1$ to $\mQ_{i+1}\cdot e_1$.
\end{enumerate}
\item For $\tilde{x}^2_{i+1}\leq x_{i+1} \leq \tilde{x}^1_{i+1}$, $\alpha$ is a type CSC curve in $\bar{\mathcal{L}}_{\rhot}\big(\mQ_i(x_i),\mQ_{i+1}(x_{i+1})\big)$.
\item For $x_{i+1}\geq \tilde{x}^1_{i+1}$, $\alpha$ is concatenation of the following $3$ curves:
\begin{enumerate}
\item Shortest arc on $\zeta_{\pta(\mQ_i),\rhot,\circlearrowleft}$ that travels from $\mQ_i\cdot e_1$ to the unique point $a_1$ in $\zeta_{\pta(\mQ_i),\rhot}\cap\zeta_{\pta,(2i+1)\rhot}$.
\item Arc on $\zeta_{\pta,(2i+1)\rhot,\circlearrowleft}$ that travels from $a_1$ to the unique point $b_1$ in $\zeta_{\pta,(2i+1)\rhot}\cap\zeta_{\ptb(\mQ_{i+1}),\rhot}$. This arc is concatenation of shortest arc from $a_1$ to $b_1$ and circle that traverses $\zeta_{\pta,(2i+1)\rhot,\circlearrowleft}$ $k$ times, where $k\in\mathbb{N}$ satisfy $(x_{i+1}-2k\pi)\in \left[\tilde{x}_{i+1}^1,\tilde{x}_{i+1}^1+2\pi\right)$.
\item Shortest arc on $\zeta_{\pta(\mQ_{i+1}),\rhot,\circlearrowright}$ that travels from $b_1$ to $\mQ_{i+1}\cdot e_1$.
\end{enumerate}
\end{enumerate} 
In the case that $i$ is odd:
\begin{enumerate}
\item For $x_{i+1}\leq \tilde{x}^2_{i+1}$, $\alpha$ is concatenation of the following $3$ curves:
\begin{enumerate}
\item Shortest arc on $\zeta_{\ptb(\mQ_i),\rhot,\circlearrowright}$ that travels from $\mQ_i\cdot e_1$ to the unique point $a_1$ in $\zeta_{\ptb(\mQ_i),\rhot}\cap\zeta_{\pta,(2i+1)\rhot}$.
\item Arc on $\zeta_{\pta,(2i+1)\rhot,\circlearrowright}$ that travels from $a_1$ to the unique point $b_1$ in $\zeta_{\pta,(2i+1)\rhot}\cap\zeta_{\pta(\mQ_{i+1}),\rhot}$. This arc is concatenation of shortest arc from $a_1$ to $b_1$ and circle that traverses $\zeta_{\pta,(2i+1)\rhot,\circlearrowright}$ $k$ times, where $k\in\mathbb{N}$ satisfy $(x_{i+1}+2k\pi)\in \left(\tilde{x}_{i+1}^2-2\pi,\tilde{x}_{i+1}^2\right]$.
\item Shortest arc on $\zeta_{\pta(\mQ_{i+1}),\rhot,\circlearrowleft}$ that travels from $b_1$ to $\mQ_{i+1}\cdot e_1$.
\end{enumerate}
\item For $\tilde{x}^2_{i+1}\leq x_{i+1} \leq \tilde{x}^1_{i+1}$, $\alpha$ is a type CSC curve in $\bar{\mathcal{L}}_{\rhot}\big(\mQ_i(x_i),\mQ_{i+1}(x_{i+1})\big)$.
\item For $x_{i+1}\geq \tilde{x}^1_{i+1}$, $\alpha$ is concatenation of the following $3$ curves:
\begin{enumerate}
\item Shortest arc on $\zeta_{\pta(\mQ_i),\rhot,\circlearrowleft}$ that travels from $\mQ_i\cdot e_1$ to the unique point $a_1$ in $\zeta_{\pta(\mQ_i),\rhot}\cap\zeta_{\ptb,(2i+1)\rhot}$.
\item Arc on $\zeta_{\ptb,(2i+1)\rhot,\circlearrowleft}$ that travels from $a_1$ to the unique point $b_1$ in $\zeta_{\ptb,(2i+1)\rhot}\cap\zeta_{\ptb(\mQ_{i+1}),\rhot}$. This arc is concatenation of shortest arc from $a_1$ to $b_1$ and circle that traverses $\zeta_{\ptb,(2i+1)\rhot,\circlearrowleft}$ $k$ times, where $k\in\mathbb{N}$ satisfy $(x_{i+1}-2k\pi)\in \left[\tilde{x}_{i+1}^1,\tilde{x}_{i+1}^1+2\pi\right)$.%
\item Shortest arc on $\zeta_{\pta(\mQ_{i+1}),\rhot,\circlearrowright}$ that travels from $b_1$ to $\mQ_{i+1}\cdot e_1$.
\end{enumerate}
\end{enumerate} 
\end{construction}

\begin{figure}
\centering
\begin{tikzpicture}
   \node[anchor=west,inner sep=0] at (0,-10.5) {\includegraphics[width=8.5cm]{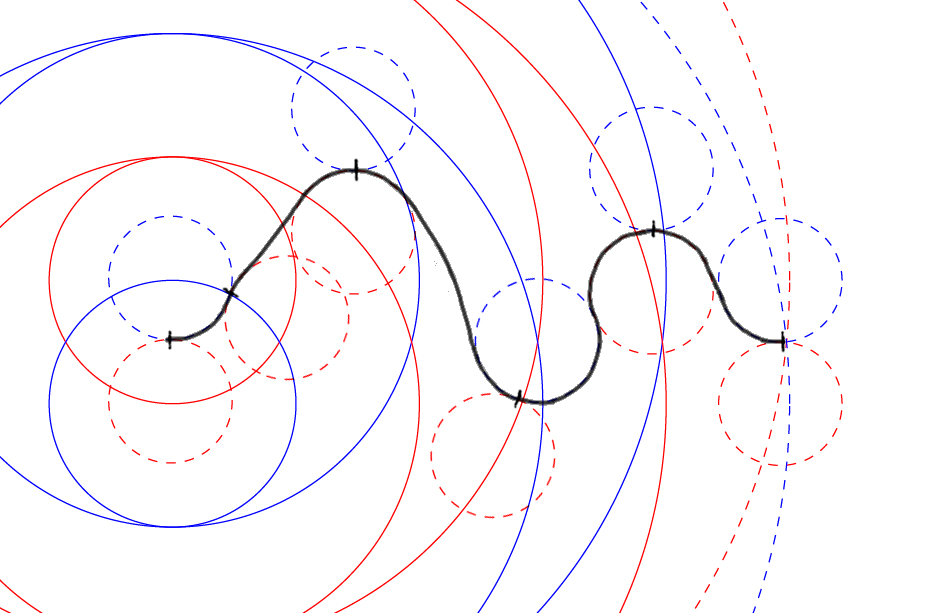}};
  \node[anchor= west,inner sep=0] at (8.2,-10.5) {\includegraphics[width=8.5cm]{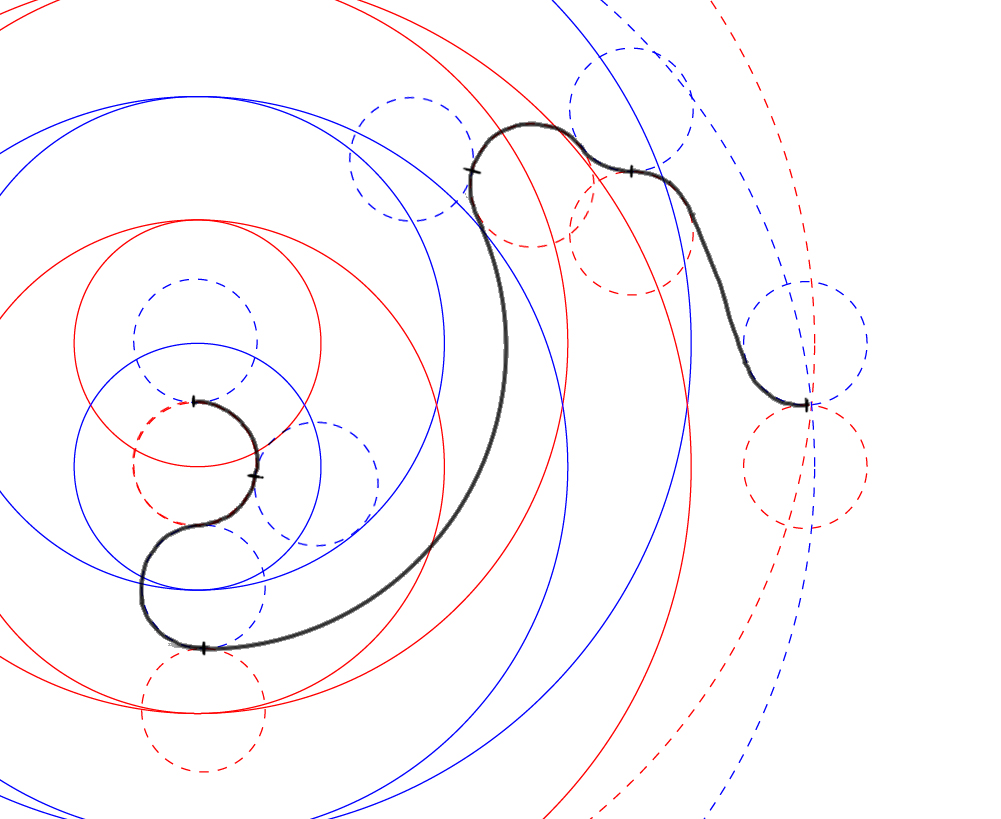}};
\end{tikzpicture}
\caption{These are examples of curves by map $\bar{F}$ for the case $n_\mQ=4$. In each figure, the thick dark curve is $\bar{F}(x)$, the marked points on curves are endpoints of $\alpha_i$, $i\in\{0,1,2,3,4\}$, the blue and red small dashed circles represent all $6$ pairs of control circles.}
\label{fig:auxcurves2}
\end{figure}

\begin{proof}[Proof of Lemma \ref{segment} and assertions on Construction \ref{propsegm}] We shall prove that such $F_{i+1}$, $\tilde{x}_{i+1}^1$ and $\tilde{x}_{i+1}^2$ exists by explicitly constructing them based on descriptions given in Construction \ref{propsegm}. If $i=0$, we set $\tilde{x}_1^1=\tilde{x}_1^2=0$, and
\begin{gather*}
\alpha_1(s)=\zeta_{\pta,\rhot,\circlearrowleft}(s), \text{ for $s\in [0,s_0]$}, \quad \text{ if $ 0\leq x_1 < 2\pi$}.\\
\alpha_1(s)=\zeta_{\ptb,\rhot,\circlearrowright}(s), \text{ for $s\in [0,s_0]$}, \quad \text{ if $-2\pi < x_{1}\leq 0$}.
\end{gather*}
In the equations above, $s_0$ is the continuous real function such that for $x_{1}\geq 0$, $\zeta_{\pta,\rhot,\circlearrowleft}(s_0)=\mQ_1(x_{1})\cdot e_1$, for $x_{1}\leq 0$, $\zeta_{\pta,\rhot,\circlearrowright}(s_0)=\mQ_1(x_{1})\cdot e_1$ and $s_0(0)=0$. Also we put $s_1\equiv s_2\equiv s_0$, so the second and the third segment of $\alpha_1$ on definition listed above are degenerate. For each integer $k\geq 1$, we also set
\begin{gather*}
\alpha_1(s)=\zeta_{\pta,\rhot,\circlearrowleft}(s), \text{ for $s\in [0,2k\pi+s_0]$}, \quad \text{ if $ 2k\pi\leq x_1 < 2(k+1)\pi$}.\\
\alpha_1(s)=\zeta_{\ptb,\rhot,\circlearrowright}(s), \text{ for $s\in [0,2k\pi+s_0]$}, \quad \text{ if $ -2(k+1)\pi < x_1\leq -2k\pi$}.
\end{gather*}

For each $i\geq 1$, we use an inductive process, set:  
\begin{gather*}
\tilde{x}^2_{i+1}=\min\left\{t\geq x_{i} ; d\big(\spL_i(t),\spR_{i+1}(t)\big)=2\rhot\right\},\\
\tilde{x}^1_{i+1}=\max\left\{t\leq x_{i} ; d\big(\spR_i(t),\spL_{i+1}(t)\big)=2\rhot\right\}.
\end{gather*}
Thus in particular, $\tilde{x}_{i+1}^1\leq x_i \leq \tilde{x}_{i+1}^2 $. We define $\alpha_i$ as in the statement of Proposition \ref{propsegm}. To verify that the definition is valid, we separate the argument into two cases.\vspace{.7em}

\noindent\textbf{Even Case with $i\geq 1$:} For each even integer $i\geq 2$, from the definition, we deduce: 
$$d(\ptb,\ptb(\mQ_i))=d(\pta,\pta(\mQ_i))=2i\rhot \quad \text{and} \quad d(\ptb,\pta(\mQ_{i+1}))=d(\pta,\ptb(\mQ_{i+1}))=(2i+2)\rhot. $$

\noindent So:
\begin{enumerate}
\item The following conclusion is obtained for the pair of curves $(\zeta_{\ptb(\mQ_i),\rhot,\circlearrowright},\zeta_{\ptb,(2i+1)\rhot,\circlearrowright})$. The intersection $\zeta_{\ptb(\mQ_i),\rhot}\cap \zeta_{\pta,(2i+1),\rhot}$ consists of exactly one point, namely $a_1$. Furthermore, the tangent vector of $\zeta_{\ptb(\mQ_i),\rhot,\circlearrowright}$ coincides with the tangent vector of $\zeta_{\ptb,(2i+1)\rhot,\circlearrowright}$ at $a_1$. 
\item The analogous conclusion is obtained for the following pairs of oriented circles 
\begin{enumerate}
\item $(\zeta_{\pta(\mQ_i),\rhot,\circlearrowleft},\zeta_{\pta,(2i+1)\rhot,\circlearrowleft})$, 
\item $(\zeta_{\pta(\mQ_{i+1}),\rhot,\circlearrowleft},\zeta_{\ptb,(2i+1)\rhot,\circlearrowright})$, 
\item $(\zeta_{\pta(\mQ_{i+1}),\rhot,\circlearrowright},\zeta_{\pta,(2i+1)\rhot,\circlearrowleft})$.
\end{enumerate}
\end{enumerate}

\noindent This makes the concatenation of segments in described on Items (1) and (3) of Construction \ref{propsegm} possible, unique and from the concatenation we obtain indeed a $C^1$ curve in $\mathcal{L}_{\rho_0}(\mQ_i,\mQ_{i+1})$.

For the proof continuity of $F_{i+1}$ at $\tilde{x}^1_{i+1}$, note that since $a_1=b_1$, the middle segment $\zeta_{a,(2i+1)\rhot,c}$, $a\in\{\pta,\ptb\}$, $c\in\{\circlearrowright,\circlearrowleft\}$ of concatenation in Item (1)(b) $\mathcal{L}_\rhot\big(\mQ_i(\tilde{x}_{i+1}^1),\mQ_{i+1}(\tilde{x}_{i+1}^2)\big)$ is degenerate. So the curve formed by concatenation of arcs constructed in Item (1) coincides with the length-minimizing curve in $\bar{\mathcal{L}}_\rhot(\mQ_{i},\mQ_{i+1})$. For continuity at $\tilde{x}^2_{i+1}$, the argument is analogous.\vspace{.7em}

\noindent\textbf{Odd Case with $i\geq 1$:} For each odd integer $i\geq 1$, the procedure is the same as the Even Case. We note that:
$$d(\pta,\ptb(\mQ_i))=d(\ptb,\pta(\mQ_i))=2i\rhot \quad \text{and} \quad d(\pta,\pta(\mQ_{i+1}))=d(\ptb,\ptb(\mQ_{i+1}))=(2i+2)\rhot.$$ 
\noindent Using the same arguments as in Even Case for pairs:
\begin{align*}
(\zeta_{\ptb(\mQ_i),\rhot,\circlearrowright},\zeta_{\pta,(2i+1)\rhot,\circlearrowright}), (\zeta_{\pta(\mQ_i),\rhot,\circlearrowleft},\zeta_{\ptb,(2i+1)\rhot,\circlearrowleft}),\\ (\zeta_{\pta(\mQ_{i+1}),\rhot,\circlearrowleft},\zeta_{\pta,(2i+1)\rhot,\circlearrowright})\quad \text{and} \quad(\zeta_{\pta(\mQ_{i+1}),\rhot,\circlearrowright},\zeta_{\ptb,(2i+1)\rhot,\circlearrowleft}),
\end{align*}
\noindent we obtain that the concatenation of segments in Items 1 and 3 is possible and is indeed a $C^1$ curve in $\mathcal{L}_{\rho_0}(\mQ_i,\mQ_{i+1})$.
The justifications for the continuity at $\tilde{x}^1_{i+1}$ and $\tilde{x}^2_{i+1}$ are also the same as in Even Case.

This proves that $F$ is well defined and continuous. The relation about the length in Construction \ref{propsegm} is an immediate consequence of its description.
\end{proof}

So by Lemma \ref{segment}, for each vector $(x_1,x_2,\ldots,x_{n_\mQ})$ we associate it to $n_\mQ+1$ curves namely:
$$\alpha_i\in\bar{\mathcal{L}}_{\rhot}(\mQ_i,\mQ_{i+1})\subset\mathcal{L}_{\rho_0}(\mQ_i,\mQ_{i+1}),\quad i=0,1,\ldots,n_\mQ . $$
Since the final frame of each $\alpha_i$ coincides with the initial frame of $\alpha_{i+1}$, the concatenation of all $\alpha_i$ results into a curve in $\mathcal{L}_{\rho_0}(\mI,\mQ)$. We define this curve as the image of $(x_1,x_2,\ldots,x_{n_\mQ})$ under $\bar{F}$:
$$\bar{F}(x_1,x_2,\ldots,x_{n_\mQ})=\bigoplus_{i=0}^{n_\mQ}\alpha_i .$$

Now we have defined a continuous application $\bar{F}:\mathbb{R}^{n_\mQ}\to\Lspace$, and we will modify it into our desired $F:\mathbb{S}^{n_\mQ}\to\Lspace$ in the next subsection. 

\begin{remark}\upshape In general, the length-minimizing CSC curve $\gamma_0$ does not lie in $\img(\bar{F})$. Only in very specific cases we have $\bar{F}(0,0,\ldots,0)=\gamma_0$. This happens in the case in which
$$\mQ = \left(\begin{array}{ccc} \cos\theta & -\sin\theta & 0 \\
\sin\theta & \cos\theta & 0 \\
0 & 0 & 1
\end{array}\right) .$$
This case is shown in Figure \ref{fig:auxcurves}.
\end{remark}

\subsection{Adding loops}

First we define the concept of geodesic loops added to a given curve $\gamma\in\cLspacea$.

\begin{definition} Consider the space $\imspace$. Given a curve $\gamma\in\imspace$, parametrized so that $\gamma:[0,1]\to\mathbb{S}^2$ and $t_0\in (0,1)$. Let $n\geq 1$ be an integer. We denote by $\gamma^{[t_0\# (2n)]}$ the following curve:
$$\gamma^{[t_0\#(2n)]}(t)=\left\{ \begin{array}{ll} \gamma(t) \quad & t\in [0,t_0-2\epsilon]\\
\gamma(t_0-2\epsilon+2(t-t_0+2\epsilon)) \quad & t\in [t_0-2\epsilon,t_0-\epsilon] \\
\mathfrak{F}_\gamma(t_0)\zeta\left(\frac{2n\pi(t-t_0+\epsilon)}{\epsilon}\right) \quad & t\in [t_0-\epsilon,t_0+\epsilon] \\
\gamma(t_0+2(t-t_0-\epsilon)) \quad & t\in [t_0+\epsilon,t_0+2\epsilon] \\
\gamma(t) \quad & t\in [t_0+2\epsilon,1] \\
\end{array}
\right.$$
In the equation above, $\epsilon$ is taken sufficiently small so that $(t_0-2\epsilon,t_0+2\epsilon)\subset [0,1]$. The curve $\zeta$ is given by $\left(\cos(t),\sin(t),0\right)$. \footnote{Two such curves with loops added by different choices of $\epsilon$ satisfying $(t_0-2\epsilon,t_0+2\epsilon)\subset [0,1]$ are in fact the same curve via the equivalence $\sim$ defined on page \pageref{equivpg}.}

For $t_0=0$ and $k\geq 1$ an integer, we define:
$$ \gamma^{[0\# k]}(t) = \left\{ \begin{array}{ll} \zeta\left(\frac{2k\pi t}{\epsilon}\right) \quad & t\in \left[0,\epsilon\right] \\
\gamma(2(t-\epsilon)) \quad & t\in [\epsilon,2\epsilon] \\
\gamma(t) \quad & t\in [2\epsilon,1].
\end{array} \right.$$

For $t_0=1$ and $k\geq 1$ an integer, we define:
$$ \gamma^{[1\# k]}(t) = \left\{ \begin{array}{ll} 
\gamma(t) \quad & t\in [0,1-2\epsilon] \\
\gamma(1-2\epsilon+2(t-1+2\epsilon)) \quad & t\in [1-2\epsilon,1-\epsilon] \\
\zeta\left(\frac{2k\pi (t-(1-\epsilon))}{\epsilon}\right) \quad & t\in \left[1-\epsilon,1\right] .
\end{array} \right.$$

\end{definition}

\begin{figure}[H]
\begin{tikzpicture}
\begin{scope}[xshift = -7cm,scale=.6]
\draw[blue,dashed] (2,0) circle (1) ;
\draw[red,dashed] (2,0) circle (2) ;
\draw[gray,dashed] (2,.5) circle (.5) ;
\draw[gray,dashed] (2,-.5) circle (.5) ;
\draw[thick] (-0.5,0) -- (4.5,0);
\end{scope}
\begin{scope}[xshift = -3.5cm,scale=.6]
\draw[blue,dashed] (2,0) circle (1) ;
\draw[red,dashed] (2,0) circle (2) ;
\draw[gray,dashed] (2.5,0) circle (.5) ;
\draw[gray,dashed] (1.5,0) circle (.5) ;
\draw[thick] (-.5,0) -- (.65,0);
\draw[thick] (4.5,0) -- (3.35,0);
\draw[thick] (.63,0) arc (-90:-30:.5) ;
\draw[thick] (1.5,0)++(0:.5) arc (-0:150:.5) ;
\draw[thick] (2.5,0)++(180:.5) arc (-180:-30:.5) ;
\draw[thick] (3.37,0) arc (90:150:.5) ;
\end{scope}
\begin{scope}[xshift = +0cm,scale=.6]
\draw[blue,dashed] (2,0) circle (1) ;
\draw[red,dashed] (2,0) circle (2) ;
\draw[gray,dashed] (2,.5) circle (.5) ;
\draw[gray,dashed] (2,-.5) circle (.5) ;
\draw[thick] (-.5,0) -- (.60,0);
\draw[thick] (4.5,0) -- (3.40,0);
\draw[thick] (.58,0) arc (-90:-50:.5) -- (1.7,.9) ;
\draw[thick] (2,0) arc (-90:125:.5)  ;
\draw[thick] (2,0) arc (90:305:.5)  ;
\draw[thick] (3.42,0) arc (90:130:.5) -- (2.3,-0.9) ;
\end{scope}
\begin{scope}[xshift = +3.5cm, scale=.6]
\draw[blue,dashed] (2,0) circle (1) ;
\draw[red,dashed] (2,0) circle (2) ;
\draw[gray,dashed] (2.5,0) circle (.5) ;
\draw[gray,dashed] (1.5,0) circle (.5) ;
\draw[thick] (-0.5,0) -- (1.2,0);
\draw[thick] (1.2,0) arc (-90:-70:1) -- (2.3,0.46) ;
\draw[thick] (2.8,0) arc (90:110:1) -- (1.7,-0.46);
\draw[thick] (1.5,0) ++ (0:.5) arc (0:300:.5);
\draw[thick] (2.5,0) ++ (-180:.5) arc (-180:120:.5);
\draw[thick] (2.8,0) -- (4.5,0);
\end{scope}
\begin{scope}[xshift = +7cm, scale=.6]
\draw[blue,dashed] (2,0) circle (1) ;
\draw[red,dashed] (2,0) circle (2) ;
\draw[thick] (-0.5,0) -- (4.5,0);
\draw[thick] (2,.5) circle (.5) ;
\draw[thick] (2,-.5) circle (.5) ;
\end{scope}
\end{tikzpicture}
\caption{For a curve $\gamma\in\bar{\mathcal{L}}_\rho(\mI,\mQ)$, if the curvature is small on a sufficiently long piece of $\gamma$, then $\gamma$ is homotopic to $\gamma^{[t_0\# 2]}$ in $\gamma\in\bar{\mathcal{L}}_\rho(\mI,\mQ)$. Under this deformation, the curve remains unchanged outside of the dashed red circle which has radius $8\rho$. In $\mathcal{I}(\mI,\mQ)$, since there is no restriction on the curvature, this deformation can be done on an arbitrary small segment of $\gamma$.}
\label{fig:homotopyloop}
\end{figure}
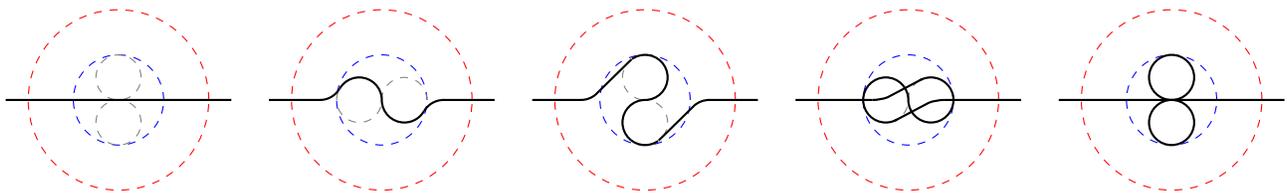

Analogously, for integers $n,m\geq 1$, we may define $\gamma^{[t_0\#2n;t_1\#2m]}$ as the curve $\gamma$ with $2n$ loops attached at $\gamma(t_0)$ and $2m$ loops attached at $\gamma(t_1)$.
\begin{definition} Consider the space $\imspace$. Given a curve $\gamma\in\imspace$ parametrized with constant speed $\gamma:[0,1]\to\mathbb{S}^2$. Given an integer $n\geq 1$, we denote by $\gamma^{[\# 2n]}$ the following curve:
$$ \gamma^{[\# 2n]}=\gamma^{[t_0\# 1;t_1\#2;\ldots;t_n\# 1]},$$
\noindent where $t_k = \frac{k}{n} $ for $k\in\{0,1,\ldots,n\}$.
\end{definition}

For $n$ sufficiently large, we define $\gamma^{[\flat (2n)]}$ by modifying the curve $\gamma^{[\# (2n)]}$. Assume that the same $\epsilon>0$ is used for each loop so that for all $t\in[0,1]$ such that $t-t_j\leq\epsilon$ we have:
$$ \gamma^{[\#(2n)]}(t)=\mathfrak{F}_{\gamma}(t_j)\zeta\left(\frac{t-t_j}{\epsilon}\right).$$

\noindent For each each $j\in\{0,1,\ldots,n\}$, let 
$$t_{j,0}=t_j+\frac{7}{8}\epsilon, \quad t_{j,\frac{1}{2}}=\frac{t_j+t_{j+1}}{2} \quad \text{and} \quad t_{j,1}=t_{j+1}-\frac{7}{8}\epsilon .$$

\noindent We also consider the unique length-minimizing CSC curve $\beta_j$ in $\bar{\mathcal{L}}_\rhot(\mathfrak{F}_\gamma(t_{j,0}),\mathfrak{F}_\gamma(t_{j,1}))$. For convenience, parametrize its domain as $\beta_j:[t_{j,0},t_{j,1}]\to\mathbb{S}^2$.

\begin{definition} Given a curve $\gamma\in\imspace$ and $\rhot\in\left(0,\frac{\pi}{4}\right)$, take an $m$ sufficiently large. For all $n\geq m$, we define $\gamma^{[\flat (2n)]}$ by:
\begin{equation} \label{eqflat1}
\gamma^{[\flat (2n)]}(t) = \left\{\begin{array}{ll}
\gamma^{[\# (2n)]}(t),\quad & \text{for } t\in [0,1]\setminus \bigcup_{j=0}^{n}(t_{j,0},t_{j,1}),\\
\beta_j(t) \quad & \text{for } t\in [t_{j,0},t_{j,1}].
\end{array}\right.
\end{equation}
\end{definition}

Below is Lemma 6.1 of \cite{sald}, the proof is based on Figure \ref{fig:homotopyloop}.
\begin{lemma}\label{lem61} Let $K$ be a compact set, $\mQ\in\SO$ and $n\geq 1$ an integer. Let $f:K\to \mathcal{I}(\mI,\mQ)$ and $t_0:K\to (0,1)$ be continuous functions. Then $f$ and $f^{[t_0\# 2n]}$ are homotopic in $\mathcal{I}(\mI,\mQ)$.
\end{lemma}

Now we introduce a simple technical lemma:
\begin{lemma}\label{simplel} Let $k\in\mathbb{N}$. If $x=(x_1,x_2,\ldots,x_{k})\in\mathbb{R}^k$, is such that:
$$\max\big\{|x_i|;i\in\{1,2,\ldots,k\}\big\}\geq 2\left\lceil\frac{k}{2}\right\rceil\pi,$$ 
\noindent then at least one of the following items is satisfied.
\begin{enumerate}
\item $|x_1|\geq 2\pi$.
\item There exists an $i\in\{1,2,\ldots,k-1\}$ such that $|x_i-x_{i+1}|\geq 2\pi$.
\item $|x_k|\geq 2\pi$.
\end{enumerate}
\end{lemma}

\begin{proof} For $k=1,2$, it is obvious. Now suppose that $k\geq 3$. Let $m\in\{1,2,\ldots,k\}$ satisfy $|x_m|=\max\big\{|x_i|;i\in\{1,2,\ldots,k\}\big\}$. We first consider the case that $m\leq\left\lceil\frac{k}{2}\right\rceil$. By triangular inequality:
$$ \max\{|x_i|\}\leq |x_1| + \sum_{i=1}^{\left\lceil\frac{k}{2}\right\rceil - 1}|x_i-x_{i+1}|.$$
Note that the left-hand side is greater or equal to $ 2\left\lceil\frac{k}{2}\right\rceil\pi$, and the right-hand side has exactly $ \left\lceil\frac{k}{2}\right\rceil$ non-negative terms. This implies $|x_1|\geq 2\pi$ or $|x_i-x_{i+1}|\geq 2\pi$ for some $i\in\left\{1,2,\ldots,\left\lceil\frac{k}{2}\right\rceil-1\right\}$. 

The case $m\geq\left\lceil\frac{k}{2}\right\rceil$ is analogous. This concludes that $|x_k|\geq 2\pi$ or $|x_i-x_{i+1}|$ for some $i\in\left\{\left\lfloor\frac{k}{2}\right\rfloor+1,\left\lfloor\frac{k}{2}\right\rfloor+2 , \ldots , k-1\right\}$.
\end{proof}

In the previous subsection,
 $$\alpha=\bar{F}(x)=\bigoplus_{i=0}^{n_\mQ}\alpha_i .$$
  Then if $x$ is such that $\max\{|x_i|\}\geq 2\left\lceil\frac{n_\mQ}{2}\right\rceil$, applying Lemma \ref{simplel} (with $k=n_\mQ$), we obtain one of the following cases:
 \begin{enumerate}
 \item If $|x_1|\geq 2\pi$, then $\alpha_0 $ has a segment of constant curvature with radius equal to $\rhot$ and with length of that segment greater than $\pi\sin\rhot$. 
 \item If $|x_i-x_{i+1}|\geq 2\pi$, then $\alpha_i $ has a segment of constant curvature with radius equal to $(2i+1)\rhot$ and length greater than $\pi\sin \left((2i+1)\rhot\right)$. 
 \item If $|x_{n_\mQ}|\geq 2\pi$, then $\alpha_{n_\mQ}$ has a segment of constant curvature with radius equal to $(2k+1)\rhot$ and length greater than $\pi\sin \left((2n_\mQ+1)\rhot\right)$. 
\end{enumerate}  

This follows directly from Construction \ref{propsegm}. For each of 3 cases above we will add a huge number of small loops on that part of the curve without changing the other parts of the curve. Now we present a construction to explain how these loops are added:

\begin{construction} Given a real number $r\in[\rho_0,\pi-\rho_0]$, consider an arc of circle of radius $r$ with angle $\theta\geq\pi$. If $r\in\left[\rhot,\frac{\pi}{2}\right]$, we add loops by following the process described in Figure \ref{fig:loop1}.
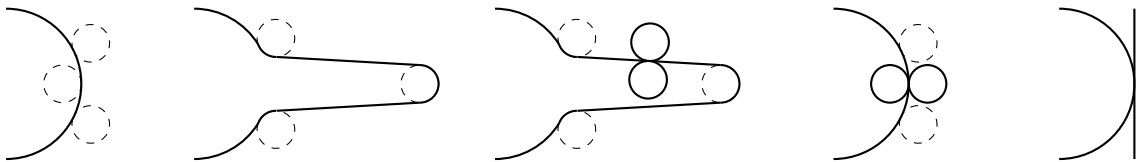
\begin{figure}
\centering
\begin{tikzpicture}
\begin{scope}[xshift=-4cm,scale=.5]
\draw[thick] (0,0) ++(-90:2) arc (-90:90:2) node[pos=.35](a){} node[pos=.5](c){} node[pos=.65](b){};
\draw[dashed] (a) arc (-200:160:.5);
\draw[dashed] (b) arc (-160:200:.5);
\draw[dashed] (c) arc (0:360:.5);
\end{scope}
\begin{scope}[xshift=-1.5cm,scale=.5]
\draw[thick] (0,0) ++(-90:2) arc (-90:-31.5:2) node[pos=1](a){};
\draw[thick] (0,0) ++(31.5:2) arc (31.5:90:2) node[pos=0](b){};
\draw[dashed] (a) arc (-200:160:.5) node[pos=.8](f){};
\draw[thick] (a.center) arc (-200:-270:.5);
\draw[dashed] (b) arc (-160:200:.5) node[pos=.2](g){};
\draw[thick] (b.center) arc (-160:-90:.5);
\draw[dashed] (6,0) ++(0:.5) arc (0:360:.5) node[pos=.25](d){} node[pos=.75](e){};
\draw[thick] (6,0) ++(-90:.5) arc (-90:90:.5);
\draw[thick] (g.center) -- (d.center);
\draw[thick] (f.center) -- (e.center);
\end{scope}
\begin{scope}[xshift=2.5cm,scale=.5]
\draw[thick] (0,0) ++(-90:2) arc (-90:-31.5:2) node[pos=1](a){};
\draw[thick] (0,0) ++(31.5:2) arc (31.5:90:2) node[pos=0](b){};
\draw[dashed] (a) arc (-200:160:.5) node[pos=.8](f){};
\draw[thick] (a.center) arc (-200:-270:.5);
\draw[dashed] (b) arc (-160:200:.5) node[pos=.2](g){};
\draw[thick] (b.center) arc (-160:-90:.5);
\draw[dashed] (6,0) ++(0:.5) arc (0:360:.5) node[pos=.25](d){} node[pos=.75](e){};
\draw[thick] (6,0) ++(-90:.5) arc (-90:90:.5);
\draw[thick] (g.center) -- (d.center) node[pos=.5](h){};
\draw[thick] (f.center) -- (e.center) node[pos=.5](i){};
\draw[thick] (h) arc (-93:260:.5);
\draw[thick] (h) arc (87:440:.5);
\end{scope}
\begin{scope}[xshift=7cm,scale=.5]
\draw[thick] (0,0) ++(-90:2) arc (-90:90:2) node[pos=.35](a){} node[pos=.5](c){} node[pos=.65](b){};
\draw[dashed] (a) arc (-200:160:.5);
\draw[dashed] (b) arc (-160:200:.5);
\draw[thick] (1.5,0) circle (.5);
\draw[thick] (2.5,0) circle (.5);
\end{scope}
\begin{scope}[xshift=10cm,scale=.5]
\draw[thick] (0,0) ++(-90:2) arc (-90:90:2) node[pos=.35](a){} node[pos=.5](c){} node[pos=.65](b){};
\draw[thick] (2,2) -- (2,-2) ;
\end{scope}
\end{tikzpicture}
\caption{This figure describes how loops were added. Dashed circles have radius $\rhot$. In the image of left side, start with an arc of circle of radius greater or equal to $\rhot$. Pull the curve by using an rotation in sphere so that the curve will have sufficiently long arcs as in the image of center left. Then, in the center image, small loops were added on long arcs by the deformation of Figure \ref{fig:homotopyloop}. Finally, the curve is deformed back to the original position except for two loops that we added. Additionally, we enlarge the radius these two loops transforming them into great circles.}
\label{fig:loop1}
\end{figure}

Analogously, if $r\in \left(\frac{\pi}{2},\pi-\rhot\right]$, we just instead of pulling the curve to right, we pull the curve to the left as if it is a mirrored version of previous case.
\end{construction}

\begin{remark}\upshape
The construction above is just one of many choices that will satisfy our future needs, one may come to several other ways to add loops that will also work. After adding enough loops, in the next step for each curve, we spread the loops along the curve. So each curve would look like a phone wire, and finally we construct a homotopy of these curves into a single curve.
\end{remark}

\begin{figure}
\centering
\begin{tikzpicture}
\begin{scope}[scale=3]
    \pgfmathsetmacro{\CameraX}{sin(\ViewAzimuth)*cos(\ViewElevation)}
    \pgfmathsetmacro{\CameraY}{-cos(\ViewAzimuth)*cos(\ViewElevation)}
    \pgfmathsetmacro{\CameraZ}{sin(\ViewElevation)}
    \path[use as bounding box] (-1,-1) rectangle (1,1); 
    \begin{scope}
        \clip (0,0) circle (1);
        \begin{scope}[transform canvas={rotate=-20}]
            \shade [ball color=white] (0,0.5) ellipse (1.8 and 1.5);
        \end{scope}
    \end{scope}
    \begin{axis}[
        hide axis,
        view={\ViewAzimuth}{\ViewElevation},     
        every axis plot/.style={very thin},
        disabledatascaling,                      
        anchor=origin,                           
        viewport={\ViewAzimuth}{\ViewElevation}, 
    ]
        \addFGBGplot[domain=0:2*pi, samples=100, samples y=1] ({sin(deg(pi/3))*cos(deg(x))}, {sin(deg(pi/3))*sin(deg(x))}, {cos(deg(pi/3))});
        \addFGBGplot[domain=0:2*pi, samples=100, samples y=1] ({cos(deg(x))}, {sin(deg(x))}, 0);
        \addFGBGplot[domain=0:2*pi, samples=100, samples y=1] ({sin(deg(pi/3))*cos(deg(x))}, {sin(deg(pi/3))*sin(deg(x))}, {-cos(deg(pi/3))});
    \end{axis}
    \draw [decorate,decoration={brace,amplitude=5pt}]
(-1.2,-.95) -- (-1.2,-.55) node [black,midway,xshift=-0.2cm,anchor=east] {$\tilde{F}(U)$};
	\draw [decorate,decoration={brace,amplitude=5pt}]
(-1.2,-.45) -- (-1.2,-.05) node [black,midway,xshift=-0.2cm,anchor=east] {Adding Loops};
	\draw [decorate,decoration={brace,amplitude=5pt}]
(-1.2,.05) -- (-1.2,.45) node [black,midway,xshift=-0.2cm,anchor=east] {Spreading Loops};
	\draw [decorate,decoration={brace,amplitude=5pt}]
(-1.2,.55) -- (-1.2,.95) node [black,midway,xshift=-0.2cm,anchor=east] {Homotopy into a point};
	\draw[dashed] (-1.2,-1) -- (-.2,-1);
	\draw[dashed] (-1.2,-.5) -- (-.85,-.5);
	\draw[dashed] (-1.2,0) -- (-1,0);
	\draw[dashed] (-1.2,.5) -- (-.85,.5);
	\draw[dashed] (-1.2,1) -- (-.2,1);
\end{scope}
\end{tikzpicture}
\caption{General behavior of map $F:\mathbb{S}^{n_\mQ}\to\Lspace$.}
\label{fig:MapF}
\end{figure}
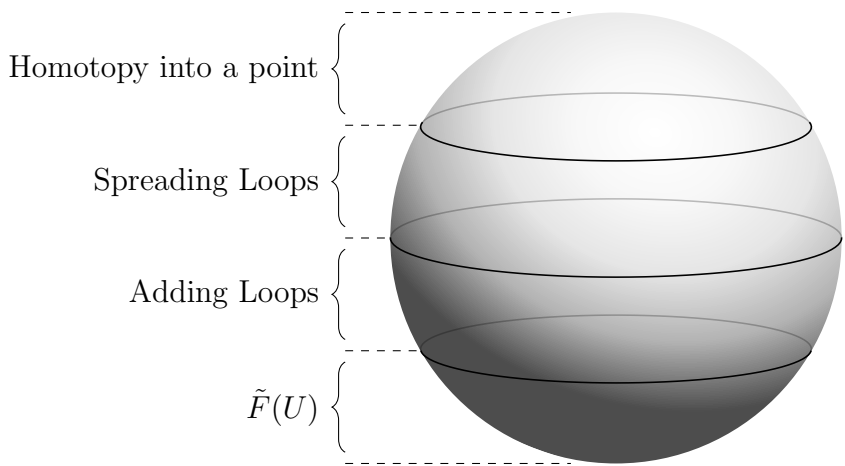

The following results are adapted directly from \cite{sald} of N. C. Saldanha. The following result corresponds to the Lemma 6.2 in this article.

\begin{lemma}\label{lem62} Let $K$ be a compact set, $\mQ\in\SO$ and $n\geq 1$ a integer. Let $t_0: K\to (0,1)$ and $f:K\to\Lspace$ be continuous functions. Then $f^{[t_0\# 2n]}$ and $f^{[t_0\#2(n+1)]}$ are homotopic, i.e., there exists $H:[0,1]\times K\to\Lspace$ with $H(0,p)=f^{[t_0\# 2n]}(p)$, $H(1,p)=f^{[t_0\#2(n+1)]}(p)$.
\end{lemma}

\begin{proof} For $n=1$, we use the deformation described in Figure \ref{fig:loop1} on one of loops on $f^{[t_0\# 2]}$. This defines a homotopy between $f^{[t_0\# 2]}$ and $f^{[t_0\# 4]}$. For general case, consider $g=(f^{[t_0\# 2(n-1)]})$. By the previous case, $g^{[t_0\# 2]}$ is homotopic to $g^{[t_0\# 4]}$. This implies that $f^{[t_0\# 2n]}$ and $f^{[t_0\#2(n+1)]}$ are homotopic.
\end{proof}

The following lemma is a direct adaptation of Lemma 6.3 in \cite{sald}.
\begin{lemma}\label{lemspre} Let $K$ be a compact set, $f:K\to\Lspace$ and $t_0:K\to (0,1)$ continuous maps. Then, for a sufficiently large $n$, the function $f^{[\flat(2n)]}$ is homotopic to $f^{[t_0\#(2n)]}$, i.e., there exists an application $H:[0,1]\times K\to\Lspace$ such that $H(0,\cdot)=f^{[\flat(2n)]}$ and $H(1,\cdot)=f^{[t_0\#(2n)]}$.
\end{lemma}

\begin{proof} Notice that the functions $f^{[t_0\# (2n)]}$ and $f^{[\#(2n)]}$ are homotopic: the homotopy consists of merely rolling loops along the curve. More precisely, for $\tilde{t}_j(s) = \frac{sj}{n}+(1-s)t_0$, this homotopy is defined by
$$ H_1(s,p) = (f(p))^{[\tilde{t}_0(s)\# 1;\tilde{t}_1(s)\# 2;\ldots;\tilde{t}_{n-1}(s)\# 2;\tilde{t}_n(s)\# 1]} .$$

We next verify that, for sufficiently large $n$, the functions $f^{[\# (2n)]}$ and $f^{[\flat(2n)]}$ are homotopic. Let $\tilde{\mQ}_j(p)=(\mathfrak{F}_{f(p)}(t_{j,\frac{1}{2}}))^{-1}\in\SO$, where $t_{j,0}, t_{j,\frac{1}{2}}, t_{j,1}$ are as in the construction of $f^{[\flat(2n)]}$. We have
\begin{align*}
\tilde{\mQ}_j(p)\mathfrak{F}_{(f(p))^{[\flat(2n)]}}(t_{j,0}) & =  \tilde{\mQ}_j(p)\mathfrak{F}_{(f(p))^{[\#(2n)]}}(t_{j,0})  \\
\tilde{\mQ}_j(p)\mathfrak{F}_{(f(p))^{[\flat(2n)]}}(t_{j,1}) & =  \tilde{\mQ}_j(p)\mathfrak{F}_{(f(p))^{[\#(2n)]}}(t_{j,1})
\end{align*}
Thus, for sufficiently large $n$, the arcs
$$ \tilde{\mQ}_j(p)(f(p))^{[\flat(2n)]},\tilde{\mQ}_j(p)(f(p))^{[\#(2n)]}:[t_{j,0},t_{j,1}]\to\mathbb{S}^2 $$
are \emph{graphs}, in the sense that the first coordinate $x: [t_{j,0},t_{j,1}]\to [x_-,x_+]$ is an increasing diffeomorphism (with $x_\pm \approx \pm\frac{1}{2} $), and $y$ and $z$ can be considered functions of $x$. Since the space of increasing diffeomorphisms of an interval is contractible, we may construct a homotopy from $f^{[\#(2n)]}$ to a suitable reparametrization $f_1$ of $f^{[\#(2n)]}$ in each $[t_{j,0},t_{j,1}]$ for which the function $x$ above is the same as for $f^{[\flat(2n)]}$. We may then join $f_1$ and $f^{[\flat(2n)]}$ by performing a convex combination followed by projection to $\mathbb{S}^2$. We observe that if the curves $f(p)$ are in $\Lspace$, then both constructions above remain in $\Lspace$.
\end{proof}

The following lemma, which is a direct adaptation of Lemma 6.6 in \cite{sald}, guarantees the continuity of the choice on Step 1:
\begin{lemma}\label{lemcont} Let $\mQ\in\SO$. Let $K$ be a compact manifold and $f:K\to\Lspace$ a continuous map. Assume that the following three properties are satisfied:
\begin{enumerate}
\item $t_0\in (0,1)$ and $t_1,t_2,\ldots,t_n:K\to (0,1)$ are continuous functions with $t_0<t_1<\cdots <t_n$;
\item $K=\bigcup_{1\leq i\leq n}U_i$, where $U_i\subset K$ are open sets;
\item there exist continuous functions  $g_i:U_i\to\Lspace$ such that, for all $p\in U_i$, we have $f(p)=\big(g_i(p)\big)^{[t_i(p)\#2]}$.
\end{enumerate}
Then there exists $H:[0,1]\times K \to \Lspace$ with $H(0,p)=f(p)$, $H(1,p)=\big(f(p)\big)^{[t_0\# 2]}$.
\end{lemma}

\begin{proof}
We proceed by induction on $n$. For $n=1$ we have $U_1=K$ and therefore $f=g_1^{[t_1\# 2]}$. The conclusion follows from the Lemma \ref{lem62}. Assume now that $n>1$. Let $W\subset U_n$ be an open set whose closure is contained in $U_n$ and such that $K=W\cup \bigcup_{1\leq i\leq n-1}U_i$. We now slide the loop in the position $t_n$ to the position $t_{n-1}$ in $W$, allowing for the loop to stop elsewhere for $p\in U_n\smallsetminus W$. More precisely, let $u: K \to [0,1]$ be a continuous function with $u(p)=1$ for $p\in W$ and $u(p) =0$ for $p\not\in U_n$. Define $H_n:[0,1]\times K\to\Lspace$ by
$$ H_n(s,p)=\left\{\begin{array}{ll} f(p), \quad & p\not\in U_n\\
g_n(p)^{[((1-u(p)s)t_n(p)+u(p)st_{n-1}(p))\# 2]}, \quad & p \in U_n
\end{array}\right.$$
Let $ \bar{f}(p)=H_n(1,p)$, $H_n$ defines a homotopy between $\bar{f}$ and $f$. Let $\bar{U}_i=U_i$ for $i<n-1$ and $\bar{U}_{n-1}=U_{n-1}\cup W$; the hypothesis of the Lemma apply to $\bar{f}$ with a smaller value of $n$ and therefore $\bar{f}$ is homotopic to $\bar{f}^{[t_0\# 2]}$. Therefore, so is $f$. 
\end{proof}

Here we give an explicit construction of $F$.

\vspace{.7em}\noindent\textbf{Step 1:} Consider the ball $B_R(0)\in\mathbb{R}^{n_\mQ}$, with $R=2\left\lceil\frac{n_\mQ}{2}\right\rceil\pi$, and take the boundary $\Theta_1=\partial B_R(0)$ of the ball which is a sphere of dimension $n_\mQ-1$. By the construction given in Lemma \ref{segment}, every curve in $\Theta_1$ has at least one arc of circle with radius $r$ in the interval $\big[\rhot,(2k+1)\rhot\big]$ with length greater than $\pi\sin r$. 

Define $f:\Theta_1\to\Lspace$ as $\bar{F}$ with two loops added to each of its long arcs. To preserve the continuity, for each arc that is very close to become a long arc draw something intermediary as shown in the Figure \ref{fig:homotopyloop}. There is a homotopy as below:
$$\bar{F}_2:\Theta_1\times[0,1]\to\Lspace ,$$
where $\bar{F}_2(\cdot,0)=\bar{F}(\cdot)$ and $\bar{F}_2(\cdot,1)=f(\cdot)$. 

\vspace{.7em}\noindent\textbf{Step 2:}  We look more carefully into the construction of $\bar{F}$ (see Figure \ref{fig:auxcurves2}). There are:
\begin{enumerate}
\item $n_\mQ+1$ open sets $U_i\in\Theta_1$, $i\in\{0,1,\ldots,n_\mQ\}$ corresponding to curves that have at least one arc of circle with radius $r$ in the interval $\big[\rhot,(2k+1)\rhot\big]$ with length greater than $\pi\sin r$ at $\alpha_i$. As seen above, $\bigcup_i U_i=\Theta_1$.
\item $n_\mQ+1$ continuous functions $t_i:\Theta_1\to (0,1)$, $i\in\{0,1,\ldots,n_\mQ\}$ corresponding to the precise parameter of the curve $\gamma= \bar{F}(p)$ in which we add loops to each of long arcs of $\alpha_i$ (when it is available). Since arcs were added in $\alpha_i$ and $\gamma$ is concatenation of $\alpha_i$'s, it is clear that $t_0<t_1<t_2<\ldots<t_{n_\mQ}$.
\item Since $t_0:\Theta_1\to (0,1)$ is continuous and $\Theta_1$ is compact, there exists a constant $t_{-1}<t_0(p)$, for all $p\in\Theta_1$ .
\item Define $g_i:U_i\to\Lspace $ as $\bar{F}\vert_{U_i}$.
\end{enumerate}

Applying Lemma \ref{lemcont} for $K=\Theta_1$ and $f:\Theta_1\to\Lspace$ we obtain the following homotopy. We obtain:
$$\bar{F}_3:\Theta_1\times[1,2]\to\Lspace ,$$
where $\bar{F}_3(\cdot,1)=f(\cdot)=\bar{F}_2(\cdot,1)$, and $\bar{F}_3(\cdot,2)=f(\cdot)^{[t_{-1}\# 2]}$ is the same curve with at least two loops added at $t_{-1}$.

\vspace{.7em}\noindent\textbf{Step 3:} Finally we prove and use the following proposition, which also is a direct adaptation of Proposition 6.4 in \cite{sald} to obtain:
$$ \bar{F}_4:\Theta_1\times[2,3]\to\Lspace ,$$
where $\bar{F}_4(\cdot,2)=\bar{F}_3(\cdot,2)$, and $\bar{F}_4(a_1,3)=\bar{F}_4(a_2,3)=\tilde{\gamma}$, for all $a_1,a_2\in\Theta_1$.

\begin{proposition}\label{prop64} Let $n$ be a positive integer. Let $K$ be a compact set and let $f:K\to \Lspace\subset \mathcal{I}(\mI,\mQ)$ be a continuous function. Then $f$ is homotopic to a constant in $\mathcal{I}(\mI,\mQ)$ if and only if $f^{[t_0\# 2n]}$ is homotopic to a constant in $\Lspace$.
\end{proposition}

\begin{proof} $(\Leftarrow)$ It is trivial. In $\mathcal{I}(\mI,\mQ)$, $f$ and $f^{[t_0\# 2n]}$ are homotopic.

\noindent $(\Rightarrow)$ Let $H:K\times [0,1]\to\mathcal{I}(\mI,\mQ)$ be a homotopy with $H(\cdot,0)=f$ and $H(\cdot,1)$ is a constant function. The image of $H^{[\flat(2m)]}$ is contained in $\Lspace$. For a sufficiently large number $m$, $H^{[\flat(2m)]}$ is also continuous. This implies that $f^{[\flat \# 2m]}(\cdot)=H^{[\flat (2m)]}(\cdot,0)$ is homotopic in $\Lspace$ to a constant. By Lemma \ref{lemspre}, $f^{[t_0\# (2m)]}$ is homotopic to $f^{[\flat \# 2m]}$ in $\Lspace$ and therefore the proposition is proved for large $n$. The general case now follows from Lemma \ref{lem62}.
\end{proof}

\vspace{.7em}\noindent\textbf{Step 4:} Now we concatenate $\bar{F}$, $\bar{F}_2$, $\bar{F}_3$, $\bar{F}_4$ to obtain $F:\mathbb{S}^{n_\mQ}\to\Lspace$. First, we divide the sphere into $\mathbb{S}^{n_\mQ}=\Theta_2\sqcup\Theta_3\sqcup\Theta_4\sqcup\Theta_5\sqcup \{(0,0,\ldots,1)\}$, where
\begin{align*}
\Theta_2 & = \left\{(b_1,b_2,\ldots,b_{n_\mQ},\cos \beta );\beta\in \left[-\pi,-\frac{3\pi}{4}\right)\right\}, \\
\Theta_3 & =  \left\{(b_1,b_2,\ldots,b_{n_\mQ},\cos \beta );\beta\in \left[-\frac{3\pi}{4},-\frac{\pi}{2}\right)\right\}, \\
\Theta_4 & =  \left\{(b_1,b_2,\ldots,b_{n_\mQ},\cos \beta );\beta\in \left[-\frac{\pi}{2},-\frac{\pi}{4}\right)\right\}, \\
\Theta_5 & =  \left\{(b_1,b_2,\ldots,b_{n_\mQ},\cos \beta );\beta\in \left[-\frac{\pi}{4},0\right)\right\}.
\end{align*}

Next, $F$ is defined as the following:
$$ F(a)=\left\{ \begin{array}{ll} \bar{F}\circ\varphi(a) \quad & a\in \Theta_2  \\
\bar{F}_2\circ\varphi(a) \quad & a\in \Theta_3  \\
\bar{F}_3\circ\varphi(a) \quad & a\in \Theta_4  \\
\bar{F}_4\circ\varphi(a) \quad & a\in \Theta_5  \\
\tilde{\gamma} \quad & a=(0,0,\ldots,0,1)
\end{array} \right. $$
where $\varphi:\mathbb{S}^{n_\mQ}\to (\Theta_1\times[0,3])$ is defined as $\varphi(0,0,\ldots,0,1)=0$ and:
$$\varphi(b_1,b_2,\ldots,b_{n_\mQ},\cos \beta ) = \left(\frac{(b_1,b_2,\ldots,b_{n_\mQ})}{|(b_1,b_2,\ldots,b_{n_\mQ})|}R,\frac{4\beta}{\pi}+4\right) \quad \text{for $\beta\in [-\pi,0)$}. $$

\subsection{Triviality of $F$ in $\mathcal{I}(\mQ)$}
Note that, by the conditions on page \pageref{condq1q2} for $\mQ$, each step may be proceeded so that the final map $F$ has the following property: there exists an open ball $U_\mQ=B_a(r)\subsetneq\mathbb{S}^2$, which depends on $\mQ$, satisfying:
$$ \img \big(F(p)\big)\subset B_a(r), \quad \forall p\in\mathbb{S}^{n_\mQ}. $$
We consider the stereographic projection $h:B_a(r)\to \mathbb{R}^2$ with center $a$. So $h\big(F(p)\big):[0,1]\to\mathbb{R}^2$ is a $C^1$ immersed curve with prescribed initial and final frames for each $p\in\mathbb{S}^{n_\mQ}$. Moreover this map defines a homeomorphism between immersed curves in $B_a(r)$ and $\mathbb{R}^2$.

Each component of the space of immersed curves with prescribed initial and final frames in $\mathbb{R}^2$ is known to be contractible, refer to the introduction and Theorem 4.1 of \cite{salzuh3}. This result is proven by S. Smale in \cite{smale}. Thus $h\big(F(\cdot)\big)$ is homotopically trivial in the space of immersed curves in $\mathbb{R}^2$. This guarantees the triviality of $(\boldsymbol{i}\circ F):\mathbb{S}^{n_\mQ}\to\mathcal{I}(\mQ)$, where $\boldsymbol{i}:\Lspace\to\mathcal{I}(\mQ)$ is the set inclusion map.

\newpage
\section{Definition of the map $G:\Lspace\to\mathbb{S}^{n_\mQ}$} \label{sec:mapG}
With the application $F:\mathbb{S}^{n_\mQ}\to\Lspace$ in hand, we need another application $G:\mathcal{L}_{\rho_0}(\mQ)\to\mathbb{S}^{n_\mQ}$ such that $G\circ F:\mathbb{S}^{n_\mQ}\to\mathbb{S}^{n_\mQ}$ has degree 1. To define $G$, we shall first define a very special contractible subset $\mathcal{C}_0\in\cLspace$, which contains the length-minimizing curve $\gamma_0$, with the property of uniqueness. The boundary $\partial\mathcal{C}_0$ consists of curves that are simultaneously ``graft-able'' and is homotopic to $\mathbb{S}^{n_\mQ-1}$. We first establish a map $\bar{G}:\mathcal{C}_0\to \bar{B}_{1}(0)\in\mathbb{R}^{n_\mQ}$ which can be easily extended into $G$.
To define such map $\bar{G}$, we need to carefully extract useful information for curves in $\mathcal{C}_0$. This information describes, roughly speaking, how many times the curve bends to left and right, and how much the curve goes ``up'' and ``down'' inside a region. The precise meaning of this information will be concretized in the subsequent text.

\subsection{Preliminary definitions}
We consider the points $p_1,p_2,q_1,q_2$ given by Equation (\ref{formulapq}). We recall that $\gamma\in\Lspace$ and $\gamma:J\to\mathbb{S}^2$, where $J$ is a closed interval in $\mathbb{R}$. Given a curve $\gamma\in\Lspace$, its unit tangent vector can be viewed as a map whose image lies in $\mathbb{S}^2$, that is $\vt_\gamma:J\to\mathbb{S}^2$. Let the set $\mathcal{C}\subset\Lspace$ be a subset containing all curves whose unit tangent vector is contained in a closed half-space. In other words, 
$$ \mathcal{C}=\left\{\gamma\in\Lspace ; \begin{array}{c} \text{there exists a } v\in\mathbb{S}^2 \text{ such that } \langle \vt_\gamma(s),v\rangle \geq 0 \text{ and} \\ \langle \gamma(s),e_2\rangle >0 \text{ for all }s \in J\end{array}\right\}. $$
We call the curves in the set $\mathcal{C}$ \emph{hemispheric} curves. Throughout this section, we assume that $\mQ\in\SO$ is such that the length-minimizing curve $\gamma_0$ is hemispheric, that is: $\gamma_0\in\mathcal{C}$. We consider the following situation: $\mQ\in\SO$ and $\rho_0$ in the definition of $\Lspace$ are such that the convex quadrilateral $\mathcal{Q}_{\mQ,1}$ on sphere formed by points $ p_1,q_1,q_2,p_2 $ has the property that the length-minimizing curve $\gamma_0$ lies inside this quadrilateral. We will show that for each hemispheric curve $\gamma\in\mathcal{C}$ there is a unique vector $v_\gamma$, depending continuously on $\gamma$, satisfying Condition (\ref{vgamma}) below. We consider the quadrilateral on sphere $\mathcal{Q}_{\mQ,2} = \{v\in\mathbb{S}^2 ;\text{$\langle v,p_i\rangle \leq 0$ and $\langle v,q_i\rangle \geq 0$ for $i=1,2$}\}$. For each $v\in\mathcal{Q}_{\mQ,2}$, we consider the value: 
$$m_{\gamma}(v)=\min\{\langle \vt_\gamma(s),v\rangle ; s\in J \} .$$
Now fix a $\gamma\in\mathcal{C}$, we take $v_\gamma\in\mathcal{Q}_{\mQ,2}$ as the vector such that:
\begin{equation}\label{vgamma}
m_\gamma(v_{\gamma})\geq m_\gamma(v) \quad \text{for all $v\in\mathcal{Q}_{\mQ,2}$.} 
\end{equation}
Intuitively speaking, $v_\gamma$ is the nearest point in $\mathcal{Q}_{\mQ,2}$ to the set $\vt_{\gamma}(J)$ (in sense of Hausdorff distance). The following proposition guarantees its uniqueness and continuity.
\begin{proposition}For each $\gamma\in\mathcal{C}$, such $v_{\gamma}$ satisfying Inequality (\ref{vgamma}) mentioned above is unique and depends continuously on $\gamma$.
\end{proposition}

\begin{proof}[Proof] We start verifying the uniqueness. Suppose by contradiction that there exist $v_1$ and $v_2$, such that both satisfy Inequality (\ref{vgamma}). First we consider the case $v_1\neq\pm v_2$, take:
$$\tilde{v}=\frac{v_0}{|v_0|},\quad \text{where } v_0=\frac{v_1+v_2}{2}.$$
Then for all $p\in\vt_\gamma(J)$
$$\langle\tilde{v},p\rangle > \left\langle\frac{v_1+v_2}{2},p\right\rangle = \frac{1}{2}\big(\langle v_1,p\rangle + \langle v_2,p\rangle \big).$$
By taking minimum for $p\in \vt_\gamma(J)$ on both sides we get
$$m_\gamma(\tilde{v})>\frac{1}{2}\big(m_\gamma(v_1)+m_\gamma(v_2)\big)  ,$$
which contradicts the maximality of $v_1$.

For the case $v_1=-v_2$, since $\gamma\in\mathcal{C}$, we obtain:
$$\langle \vt_\gamma(s),v_1 \rangle \geq 0 \quad \text{and} \quad \langle \vt_\gamma(s),v_2 \rangle \geq 0 \quad \forall s\in J .$$
This implies that $\vt_\gamma(J)$ is contained in the great circle in the plane perpendicular to $v_1$. So we deduce that $\gamma$ is an arc of circle centered at $\pm v_1$, with radius $r\in (\rho_0,\pi-\rho_0)$ and length greater or equal to $\pi\sin r$. Thus $\langle \gamma(s),e_2\rangle< 0$ for some $s\in J$. This contradicts the fact that $\gamma\in\mathcal{C}$.

Now we discuss the continuous dependence of $v_\gamma$ on $\gamma$. Suppose that, by contradiction, for a pair $ (\gamma,v_\gamma)$, $\gamma\in\mathcal{C}$ there is a sequence of pairs $(\gamma_k,v_k)$, with $k\in\mathbb{N}$, and an $\epsilon>0$ such that $\gamma_k\in\mathcal{C}$, $v_k=v_{\gamma_k}$, $\lim_{k\to\infty}\gamma_k=\gamma$ and $d(v_\gamma,v_k)>\epsilon$. By compactness, we assume, without loss of generality, that the sequence $v_k$ converges to a limit $\tilde{v}\neq v_\gamma$. So
$$ \min_{s\in J} d(\vt_{\gamma_k}(s),v_\gamma)\leq \min_{s\in J} d(\vt_{\gamma_k},v_k) .$$
By taking $k\to\infty$ we obtain:
$$ \min_{s\in J} d(\vt_{\gamma}(s),v_\gamma)\leq \min_{s\in J} d(\vt_{\gamma},\tilde{v}) ,$$
which contradicts the uniqueness of $v_\gamma$.
\end{proof}

We consider these meridians with axis $v_\gamma$ passing through the points $p_1,p_2,q_1,q_2$ respectively. Let $\Theta_{1,\gamma}$ be the widest region containing $\gamma_0$ delimited by two of these meridians. Now we declare $\mathcal{C}_0$ as the following set:
$$ \mathcal{C}_0\coloneqq \{\gamma\in\Lspace ; \gamma(J)\subset \bar{B}_{\rho_0}(\Theta_{1,\gamma}) \text{ and $\gamma$ is Hemispheric.}\}. $$
We start by constructing a continuous map $\bar{G}:\mathcal{C}_0\to\mathbb{R}^{n_\mQ}$, which will satisfy $\bar{G}(\gamma)\geq R$ for some $R>0$ and all $\gamma\in\partial \mathcal{C}_0$. Then we put $G:\Lspace\to\mathbb{S}^{n_\mQ}$ as:
$$ G(\gamma)=\left\{\begin{array}{ll} p\circ \bar{G}(\gamma)  \quad & \text{if $\gamma\in\mathcal{C}_0$ and $\bar{G}(\gamma)<R$} \\
(0,0,1) \quad & \text{if $\bar{G}(\gamma)\geq R$ or $\gamma\in\Lspace\smallsetminus \mathcal{C}_0$}
\end{array}\right. ,$$
\noindent where $p$ is a homeomorphism map from the open ball $B_R(0)\subset\mathbb{R}^{n_\mQ}$ to $\mathbb{S}^{n_\mQ}\smallsetminus \{(0,0,1)\}$.
Hence the following subsections are dedicated to define the map $\bar{G}$. We will use the notation: $\sign :\mathbb{R}\to \{-1,0,+1\}$ with $\sign(x)=-1$ if $x<0$, $\sign(x)=0$ if $x=0$ and $\sign(x)=+1$ if $x>0$.

For each $\epsilon\in \left(0,\rho_0\right)$, our first step is to define a map $G_\epsilon:\Lspace\to\mathbb{R}^{n_\mQ}$. For suitable values of $\epsilon$ we will be able to use the map $G_\epsilon$ to construct the desired map $G$. Given a $\gamma\in\Lspace$, define the following two sets in $\mathbb{S}^2$:
\begin{equation}\label{defxi}
 \Xi_0(\epsilon,\gamma) \coloneqq \bar{B}_{\rho_0}(\Theta_{1,\gamma})\smallsetminus B_{\rho_0-\epsilon}(\Theta_{1,\gamma}) \quad \quad \Xi_1(\epsilon,\gamma) \coloneqq \bar{B}_{\frac{\pi}{2}}(v_\gamma)\smallsetminus B_{\frac{\pi}{2}-\epsilon}(v_\gamma) .
\end{equation}
\noindent Sometimes we omit $\epsilon$ by using notations $\Xi_0$ and $\Xi_1$ for both sets in Equation (\ref{defxi}) above.

\subsection{Extracting information based on the behavior of $\gamma\in\mathcal{C}_0$}

In this subsection we shall study the structure of intersections $\Xi_0\cap\gamma(I)$ and $\Xi_1\cap \vt_\gamma(I)$. From these intersections, for a suitable $\epsilon$, we shall construct a sequence $\big(y_1,\ldots ,y_{n_\mQ}\big)\in\mathbb{R}^{n_\mQ}$. This sequence will be used to construct the application $G$.
It follows from definition that the set $\Xi_0$ is symmetric by reflection in relation to a plane $\mathcal{P}$ passing through a $v_\gamma$-meridian and crosses the curve $\gamma_0$. We denote the upper and the lower parts of $\Xi_0$ in relation to $\mathcal{P}$ by $\Xi_0^+$ and $\Xi_0^-$, respectively. Also, we denote $\Xi_1^+$ and $\Xi_1^-$ the upper and the lower parts of $\Xi_1$ in relation to the plane $\mathcal{P}$. Furthermore we take $\Xi_0^+$ and $\Xi_1^+$ as both closed sets, by including the sections in the intersection of $\Xi_0$ and $\Xi_1$ with the plane $\mathcal{P}$. But later it will turn out this choice will not be important for our needs, because both the curve and its tangent vector will get nowhere close to $\partial\Xi_0^+\cap\partial\Xi_0^-$ and $\partial\Xi_1^+\cap\partial\Xi_1^-$ respectively. 

We use the notation $\left(\mathbb{R}^+\right)^\mathbb{N}$ to denote the space of sequences of non-negative numbers $(x_n)_{n\in\mathbb{N}}$, $x_k\geq 0$ for all $k\in\mathbb{N}$. Pick an $\epsilon\in [0,\epsilon_0)$, given a curve $\gamma\in\mathcal{C}_0$, we want to ``extract'' from the pair $(\epsilon,\gamma)$ a sequence $(x_n)_{n\in\mathbb{N}}\in (\mathbb{R}^+)^\mathbb{N}$ by the following 5 steps: 
\\[\cvspace]
\noindent\textbf{Step 1:} Consider two sets: $\mathcal{S}_1\coloneqq \Xi_0(\epsilon,\gamma)\cap \gamma$ and $\mathcal{S}_2\coloneqq \Xi_1(\epsilon,\gamma)\cap \vt_\gamma$. Also consider $J_1\coloneqq\{s\in J ; \gamma(s)\in \Xi_0\}$ $J_2\coloneqq\{s\in J ; \vt_\gamma(s)\in \Xi_1\}$.\footnote{Despite of use of $J_k$ to represent the subsets of $J$ here, $J_k$ are \emph{not} intervals in general.} Note that for $\epsilon$ sufficiently small, $J_1\cap J_2=\emptyset$ (taking $\epsilon <\frac{\rho_0}{8}$ is enough for that). 
\\[\cvspace]
\noindent\textbf{Step 2:} If both sets $\mathcal{S}_1$ and $\mathcal{S}_2$ are empty, we set $x_k=0$ for all $k\in\mathbb{N}$. So we finished defining the sequence for this particular case.
\\[\cvspace]
\noindent\textbf{Step 3:} If $\mathcal{S}_1$ or $\mathcal{S}_2$ is non-empty. We subdivide sets $J_1$ and $J_2$ into disjoint unions $J_1=J_1^+\sqcup J_1^-$ and $J_2=J_2^+\sqcup J_2^-$ by setting $J_1^+=\{s\in J;\gamma(s)\in\Xi_0^+\}$, $J_1^-=\{s\in J;\gamma(s)\in\Xi_0^-\}$, $J_2^+=\{s\in J;\gamma(s)\in\Xi_1^+\}$ and $J_2^-=\{s\in J;\gamma(s)\in\Xi_1^-\}$.
\\[\cvspace]
\noindent\textbf{Step 4:} \label{step4} Again, we subdivide these four sets into disjoint unions:
 $$J_1^+=\bigsqcup_{k\in\mathbb{N}} J_{1,k}^+ , \qquad J_1^-=\bigsqcup_{k\in\mathbb{N}} J_{1,k}^- , \qquad J_2^+=\bigsqcup_{k\in\mathbb{N}} J_{2,k}^+ , \qquad J_2^-=\bigsqcup_{k\in\mathbb{N}} J_{2,k}^- . $$
This subdivision may be done so that it satisfies the following 3 properties:
\begin{enumerate}
\item $J_{2,k}^+ < J_{1,k}^+ < J_{2,k}^- < J_{1,k}^- < J_{2,k+1}^+$ for all $k\in\mathbb{N}$. \footnote{For two disjoint subsets $J$ and $K$ of $\mathbb{R}$, we write $J<K$ when $a<b$ for all $a\in J$ and $b\in K$.}
\item If for some $k\in\mathbb{N}$, $k>0$, one of the following 4 cases occurs:
\begin{enumerate}
\item $J_{1,k}^+$, $J_{2,k}^-$, $J_{1,k}^-$ are all empty sets.
\item $J_{2,k}^-$, $J_{1,k}^-$, $J_{2,k+1}^+$ are all empty sets.
\item $J_{1,k}^-$, $J_{2,k+1}^+$, $J_{1,k+1}^+$ are all empty sets.
\item $J_{2,k+1}^+$, $J_{1,k+1}^+$, $J_{2,k+1}^-$ are all empty sets. 
\end{enumerate} 
Then for all integers $l$ such that $l> k$, the sets $J_{2,l}^+$, $J_{1,l}^+$, $J_{2,l}^-$, $J_{1,l}^-$ are all empty sets.
\item If one of the following 3 cases occurs:
\begin{enumerate}
\item $J_{2,0}^-$, $J_{1,0}^-$, $J_{2,1}^+$ are all empty sets.
\item $J_{1,0}^-$, $J_{2,1}^+$, $J_{1,1}^+$ are all empty sets.
\item $J_{2,1}^+$, $J_{1,1}^+$, $J_{2,1}^-$ are all empty sets. 
\end{enumerate} 
Then for all integers $l$ such that $l> 0$, the sets $J_{2,l}^+$, $J_{1,l}^+$, $J_{2,l}^-$, $J_{1,l}^-$ are all empty sets.
\end{enumerate} 
Intuitively, the property in Item (1) means that these sets are ordered in a strictly increasing fashion, and Items (2) and (3) say that the redundant empty sets compressed together so that non-empty sets have smallest indexes possible.
\\[\cvspace]
\noindent\textbf{Step 5:} For each $k\in\mathbb{N}$, denote by $\mathcal{A}_k^+$ the union of all closed regions delimited by $\gamma(J_{1,k}^+)$ and $\partial \Xi_0^+$ lying on the right of $\gamma$. Analogously, denote by $\mathcal{A}_k^-$ the union of all closed regions delimited by $\gamma(J_{1,k}^-)$ and $\partial \Xi_0^-$ lying on the left of $\gamma$. Lastly, for $J_{1,k}^+$ and $J_{1,k}^-$ empty, we set $\mathcal{A}_k^+ = \emptyset$. We define the sequence $(x_k)_{k\in\mathbb{N}}$:
\begin{equation} \label{eqextract}
 x_{4k}=\length \left(\vt_\gamma|_{J_{2,k}^+}\right) ,\enskip x_{4k+1}=\area \left(\mathcal{A}_k^+\right) ,\enskip x_{4k+2}=\length \left(\vt_\gamma|_{J_{2,k}^-}\right) ,\enskip x_{4k+3}=\area \left(\mathcal{A}_k^-\right).
\end{equation}
Keep in mind that the sequence $(x_k)_{k\in\mathbb{N}}$ depends on the value of pair $(\epsilon,\gamma)$. These formulas above will be crucial to define $G$, but before that we need to establish several properties. So we postpone the main construction to the next subsection. 

For such a sequence $(x_k)_{k\in\mathbb{N}}$, we define a kind of index for curves $\gamma$ in $\mathcal{C}_0$, we call it $\epsilon$-index of $\gamma$. Denote $\epsilon$-index by $\cindex_\epsilon : \mathcal{C}_0 \to \mathbb{N} $, defined as follows:
\begin{equation}\label{eqindex}
\cindex_\epsilon(\gamma) = \left\{ \begin{array}{ll}
\left\lceil\dfrac{k}{2}\right\rceil, & \text{if $x_k\neq 0$, $ x_l = 0 $ $\forall$ $l> k$ and $x_0$ or $x_1$ are non-zero.} \\[1.2em]
\left\lceil\dfrac{k}{2}\right\rceil -1, & \text{if $x_k\neq 0$, $ x_l=0 $ $\forall$ $l>k$ and $x_0=x_1=0$.} \\[1.2em]
0 , & \text{if $x_k=0$ for all $k\in\mathbb{N}$.}
\end{array}
\right. 
\end{equation}
\begin{figure}[htb]
\centering
\begin{tikzpicture}
 \node[anchor=south west,inner sep=0] at (0,0) {\includegraphics[width=8cm]{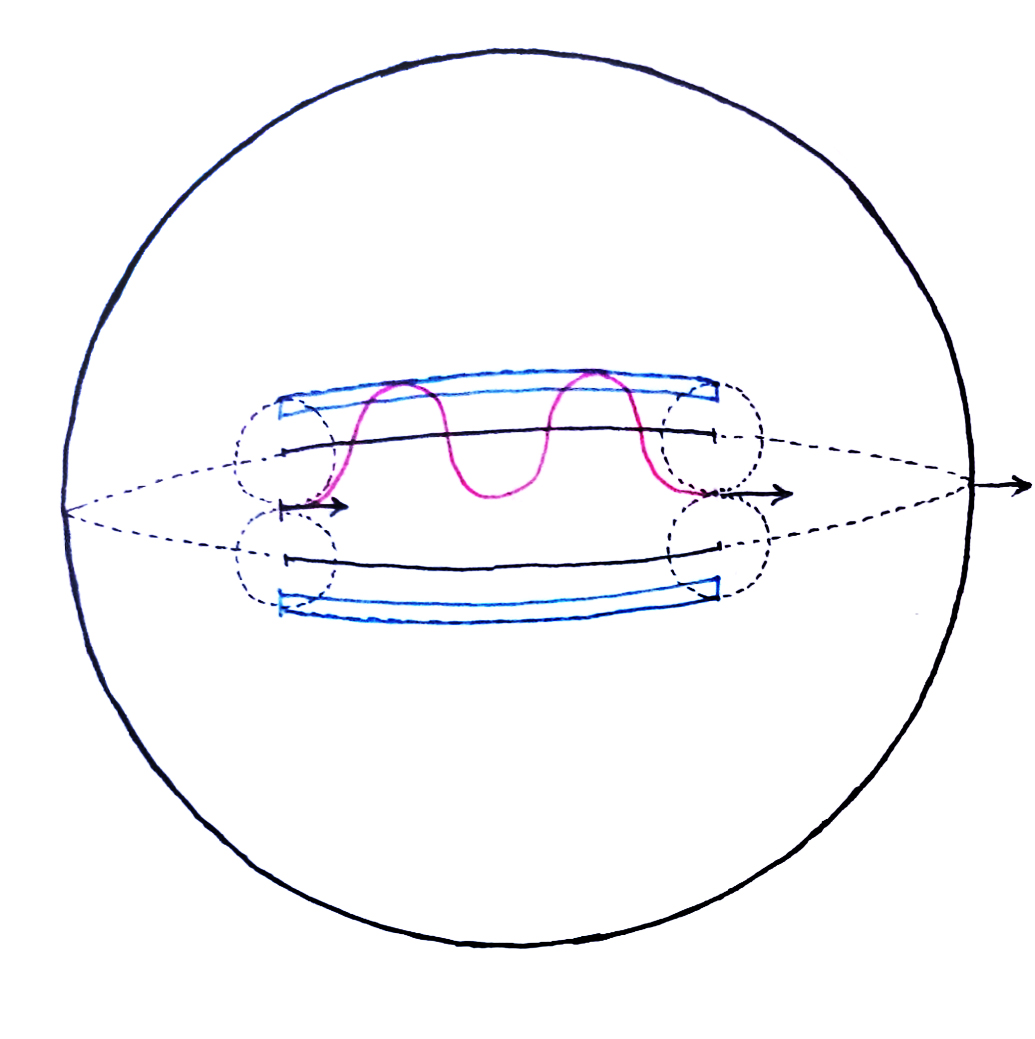}};
  \node[anchor=south west,inner sep=0] at (8.5,0) {\includegraphics[width=8cm]{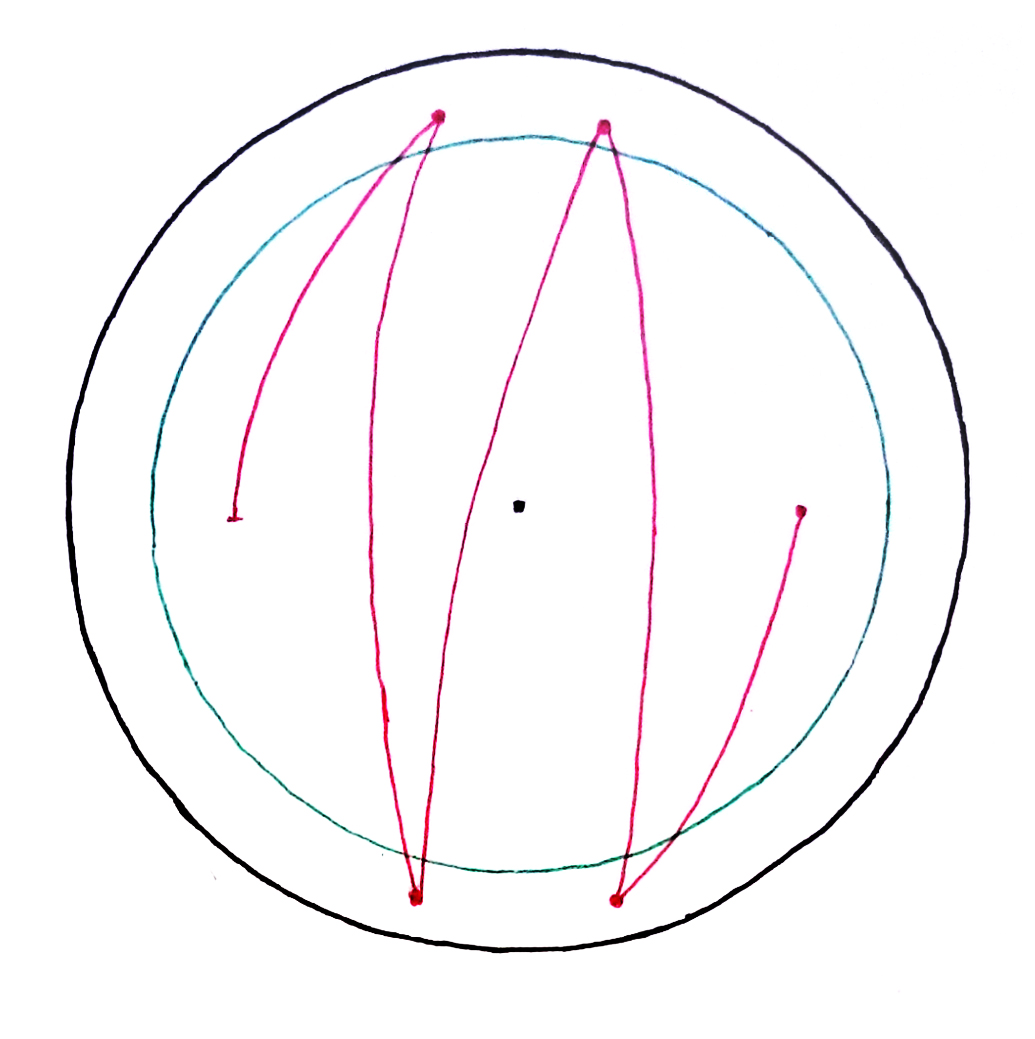}};
\node at (8,4.6) {$v_\gamma$};
\node at (12.8,4.4) {$v_\gamma$};
\node at (4,5.6) {$\Xi^+_0$};
\node at (4,2.9) {$\Xi^-_0$};
\node at (12.8,7.25) {$\Xi^+_1$};
\node at (12.8,1.0) {$\Xi^-_1$};
\end{tikzpicture}
\caption{Illustration of an example of $\gamma$ (the curve in red on the left-hand side) and its tangent vector $\vt_\gamma$ in $\mathbb{S}^2$ (the curve in red on the right-hand side). In this example we have $x_0, x_1, x_2 >0$, $x_3=0$, $x_4, x_5, x_6>0$ and $x_k=0$ for all $k\geq 7$. The $\epsilon$-index of $\gamma$ is $3$.}
\label{fig:extract}
\end{figure}

\begin{figure}[htb]
\centering
\begin{tikzpicture}
 \node[anchor=south west,inner sep=0] at (0,0) {\includegraphics[width=8cm]{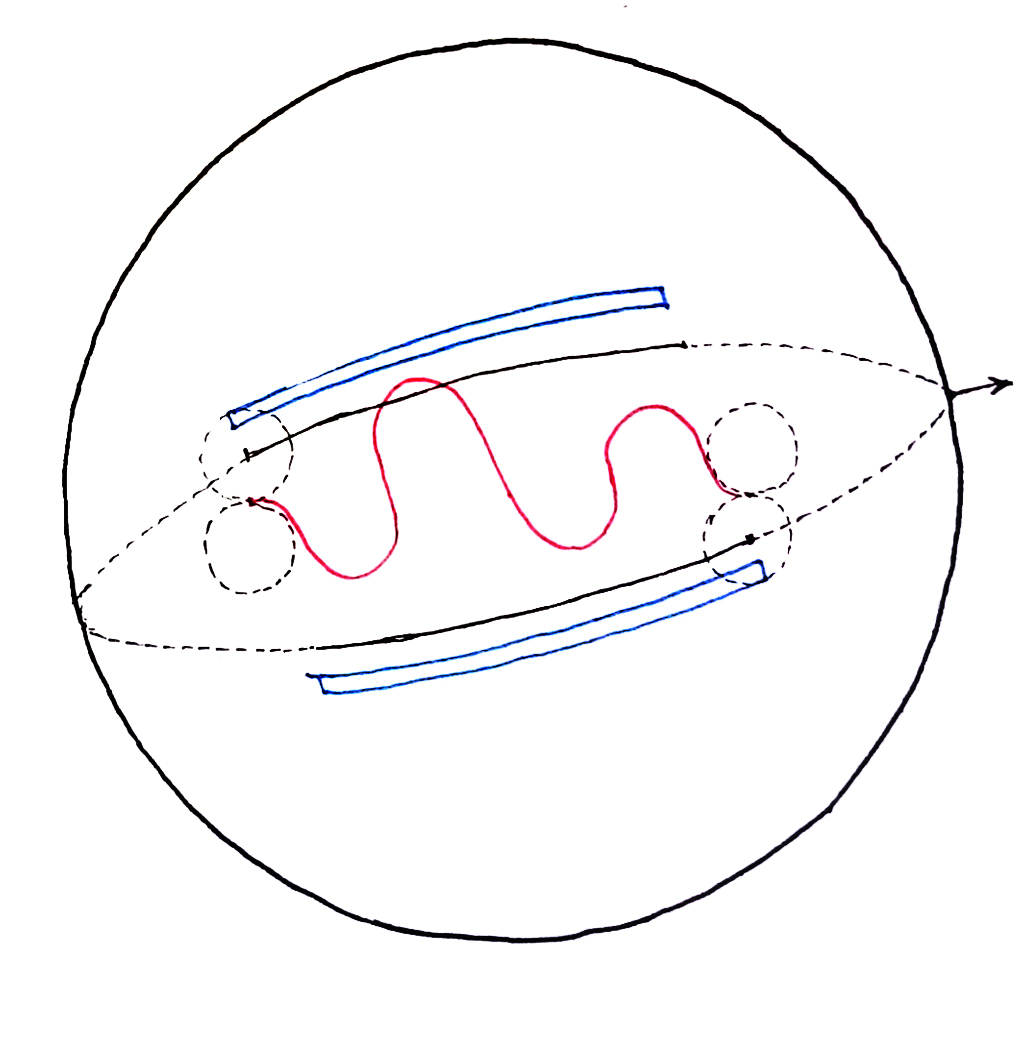}};
  \node[anchor=south west,inner sep=0] at (8.5,0) {\includegraphics[width=8cm]{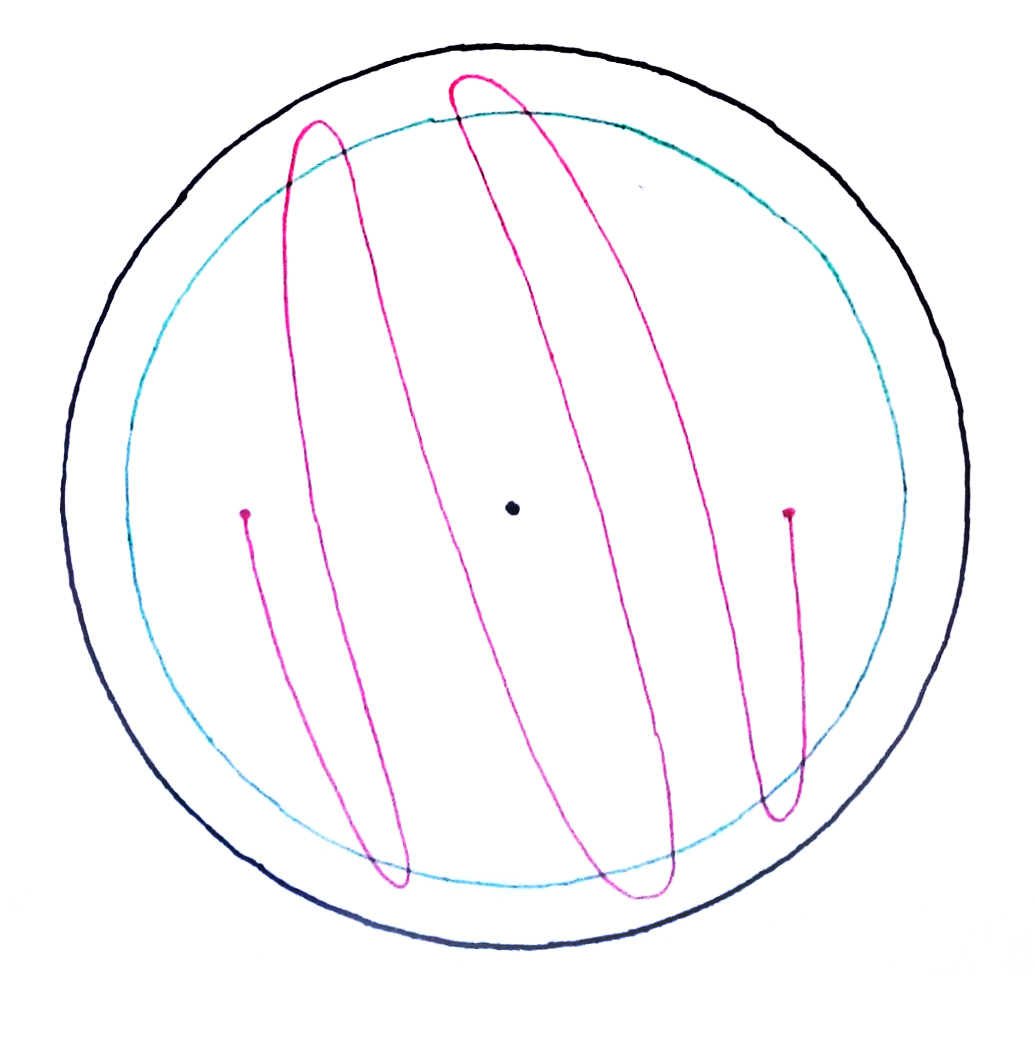}};
\node at (8,4.8) {$v_\gamma$};
\node at (12.8,4.3) {$v_\gamma$};
\node at (4,6) {$\Xi^+_0$};
\node at (4,2.6) {$\Xi^-_0$};
\node at (12.8,7.35) {$\Xi^+_1$};
\node at (12.8,0.95) {$\Xi^-_1$};
\end{tikzpicture}
\caption{Another illustration of an example of $\gamma$ (the curve in red on the left-hand side) and its tangent vector $\vt_\gamma$ in $\mathbb{S}^2$ (the curve in red on the right-hand side). In this example we have $x_0=x_1=0$, $x_2>0$, $x_3=0$, $x_4>0$, $x_5=0$, $x_6>0$, $x_7=0$, $x_8>0$, $x_9=0$ and $x_{10}>0$, $x_k=0$ for all $k\geq 11$. The $\epsilon$-index of $\gamma$ is $4$.}
\label{fig:extract2}
\end{figure}
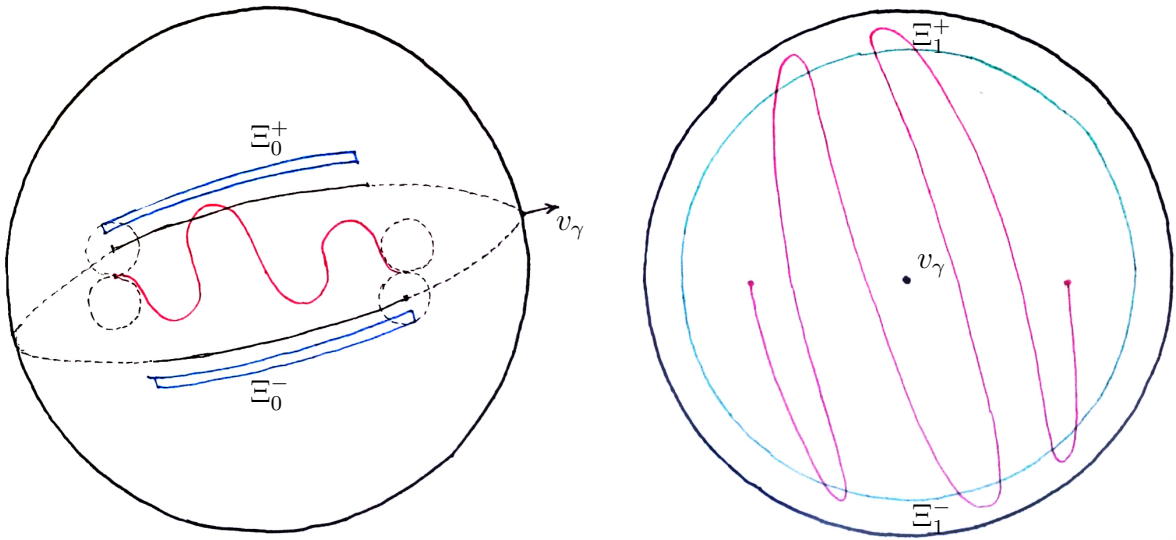
\noindent A note about the second case in Equation (\ref{eqindex}), the condition ``$x_0$ and $x_1$ are both equal to zero'' together with Condition 2 of Step 4, implies that $x_2$ or $x_3$ is not zero. So these three cases in the definition of $\epsilon$-index do include all possible scenarios for the sequence $(x_k)_{k\in\mathbb{N}}$. We note that index of a curve is non-decreasing relative to $\epsilon$ in proposition below:
\begin{proposition} If there are two real values $\epsilon$ and $\bar{\epsilon}$ such that $0<\epsilon\leq\bar{\epsilon}<\epsilon_0$, then $\cindex_\epsilon(\gamma) \leq \cindex_{\bar{\epsilon}}(\gamma)$.
\end{proposition}
\begin{proof} Given a curve $\gamma\in\mathcal{C}_0$. The inequality $\epsilon\leq\bar{\epsilon}$ and Formula (\ref{defxi}) imply that $\Xi_0(\epsilon,\gamma)\subset\Xi_0(\bar{\epsilon},\gamma)$ and $\Xi_1(\epsilon,\gamma)\subset\Xi_1(\bar{\epsilon},\gamma)$. This subsequently implies that $J_1^+(\epsilon,\gamma)\subset J_1^+(\bar{\epsilon},\gamma)$, $J_1^-(\epsilon,\gamma)\subset J_1^-(\bar{\epsilon},\gamma)$, $J_2^+(\epsilon,\gamma)\subset J_2^+(\bar{\epsilon},\gamma)$ and $J_2^-(\epsilon,\gamma)\subset J_2^-(\bar{\epsilon},\gamma)$. Now we check the rules for subdivision in Step 4 and Step 5, it is clear that implies $\cindex_\epsilon(\gamma)\leq\cindex_{\bar{\epsilon}}(\gamma)$
\end{proof}

In the visual aspect, a curve $\gamma$ having an $\epsilon$-index indicates that $\gamma$ resembles a critical curve of $\cindex_\epsilon(\gamma)$. The exact meaning and reasons of this similarity will be clarified in the next subsection. Now we shall prove the following essential proposition about $\epsilon$-index:

\begin{proposition}\label{propindex} There exists an $\epsilon_1>0$ such that for all $\epsilon\in (0,\epsilon_1)$, the function $\cindex_\epsilon:\mathcal{C}_0\to \mathbb{N}$ satisfies the following condition:
$$ \cindex_\epsilon(\gamma) \leq n_\mQ \quad \text{for all $\gamma\in\mathcal{C}_0$}.$$
\end{proposition}

To prove Proposition \ref{propindex}, we need the following result.
\begin{lemma} \label{convergsubseq} Let $(\gamma_i)_{i\in\mathbb{N}}$ be a sequence of $C^1$ curves in $\mathbb{S}^2$ satisfying the following properties.
\begin{enumerate}
\item There exists a limited region $\mathcal{R}\subset\mathbb{S}^2$ such that $\gamma_i\subset\mathcal{R}$ for all $i\in\mathbb{N}$.
\item The $\kappa_{\gamma}^-$ and the $\kappa_{\gamma}^+$ lie inside an interval $[-\kappa_0,+\kappa_0]$, with $\kappa_0\in\mathbb{R}^+$.
\item There exists a positive number $L_0$ such that $\length(\gamma_i)\leq L_0$ for all $i\in\mathbb{N}$.
\end{enumerate}
Then $(\gamma_i)_{i\in\mathbb{N}}$ admits a convergent subsequence, and the limit of this subsequence satisfies all three conditions above.
\end{lemma}

\begin{lemma} \label{semicontinuity} Given $\gamma\in \mathcal{I}$, $\kappa^+_\gamma$ and $\kappa^-_\gamma$ are upper-semicontinuous and lower-semicontinuous respectively.\footnote{A real function $f:J\to\mathbb{R}$ is said to be upper-semicontinous if $\limsup_{s\to s_0} f(s) \leq f(s_0)$ for all $s_0\in J$; $f$ is said to be lower-semicontinuous if $-f$ is upper-semicontinous.}
\end{lemma}

\begin{proof}[Proof of Lemma \ref{semicontinuity}] Let us prove that $\kappa^+_\gamma$ is upper-semicontinuous, which is equivalent to prove that $r^+ = \arccot \kappa^+_\gamma$ is lower-semicontinous. Given an $s_0\in J$ and an $r<r^+(s_0)$, there exists a $\delta>0$ and $a_0\in\mathbb{S}^2$ such that $d(a_0,\gamma(s_0)) = r $ and
\begin{equation}\label{eq6}
d(a_0,\gamma(s))\geq r, \quad \forall s\in (s_0-\delta,s_0+\delta).
\end{equation}
Since $\vt_\gamma$ is continuous, (\ref{eq6}) implies that there exists a $\delta_1$ such that for all $s\in (s_0-\delta_1,s_0+\delta_1)$, we can define the center $a_s$ of the left tangent circle of radius $r$ at $\gamma(s)$, with $d(a_s,\gamma(s))=r $ and
$$ d(a_s,\gamma(s))\geq r ,\quad \forall s\in (s_0-\delta_1,s_0+\delta_1).$$
That means $\kappa^+_\gamma(s)<\cot r$, for all $s$ in a neighborhood of $s_0$. Since $r< r^+$ is arbitrary, then
$$\limsup_{s\to s_0}\kappa^+_\gamma(s) \leq \kappa^+_\gamma(s_0) .$$  
For the case of $\kappa^-_\gamma$ the procedure is analogous.
\end{proof}

\begin{proof}[Proof of Lemma \ref{convergsubseq}] First, we unify the domains of all curves to $[0,1]$ by writing $\gamma_i: [0,1]\to \mathbb{S}^2$. Take a dense sequence $(s_j)_{j\in\mathbb{N}}$ in $[0,1]$. Condition (1) and the diagonal argument allow us to pick a subsequence $(\tilde{\gamma}_i)_{i\in\mathbb{N}}$ which converges for all $s_j$, with $j\in\mathbb{N}$. Condition (3) implies that $(\tilde{\gamma}_i)_{i\in\mathbb{N}}$ has a limit $\gamma$ such that $\gamma\in\mathcal{R}$ and $\length(\gamma)\leq L_0$. 
Using Condition (2), we take a subsequence for $(\tilde{\gamma}_i)_{i\in\mathbb{N}}$ such that the tangent vector, $\kappa^+_{\gamma_i}$ and $\kappa^-_{\gamma_i}$ converge for all $s_j$. 

We need to verify if $\gamma$ also satisfies Condition (2). 
We take $\rho=\cot\big(\kappa^+_\gamma(s)\big)$ and $a$ the center of the left tangent circle of radius $\rho$ at $\gamma(s)$. 
Take $r^+_{\gamma_i}(s_j)$ as the radius of the left tangent circle of $\gamma_i$ at $\gamma_i(s_j)$ and $r^+_{\gamma_i}(s)$ as the radius of the left tangent circle at $\gamma_i(s)$. By Lemma \ref{semicontinuity} and $\rho_0\geq r^+_{\gamma_i}$ for all $i,j\in\mathbb{N}$ we obtain:
$$ \rho_0 \geq \lim_{j\to\infty} r^+_{\gamma_i}(s_j) = r^+_{\gamma_i}(s), \quad \forall j \in\mathbb{N}. $$
\noindent Taking $i\to\infty$, we obtain:
$$\rho_0\geq r^+_\gamma(s).$$
\noindent The equation above means $\kappa_0\geq\kappa^+_\gamma(s)$. The proof for another inequality $-\kappa_0\leq\kappa^-_\gamma(s)$ is analogous, and so we omit it.
\end{proof}

Now we are ready to prove Proposition \ref{propindex}. 

\begin{proof}Suppose, by contradiction, that no such $\epsilon_1$ exists. Then there exists a decreasing sequence $(\epsilon_i)_{i\in\mathbb{N}}$ converging to $0$ and a sequence $(\gamma_i)_{i\in\mathbb{N}}$ such that:
$$ \cindex_{\epsilon_i}(\gamma_i)\geq n_\mQ+1 . $$
\noindent Since all these curves $\gamma_i$ lie in a hemisphere, and have an upper bound for their length and $\mathcal{C}_0$ is closed, by Lemma \ref{convergsubseq}, the sequence $(\gamma_i)_{i\in\mathbb{N}}$ has a convergent subsequence with limit $\bar{\gamma}\in\mathcal{C}_0$. To simplify notations, we now assume that the sequence $(\gamma_i)_{i\in\mathbb{N}}$ converges to $\bar{\gamma}$. So we have:
$$\cindex_{\epsilon_i}(\bar{\gamma})\geq n_\mQ+1, \quad\text{for all $i\in\mathbb{N}$.}$$
We shall use the curve $\bar{\gamma}$ to construct a critical curve of the same index, which is greater than $n_\mQ$, contradicting the definition of $n_\mQ$ on page \pageref{defnq}. 

We first consider the case when $x_0$ or $x_1$ is not zero. Since $\bar{\gamma}\in\mathcal{C}_0$, we take the vector $v_{\bar{\gamma}}$ such that $\vt_{\bar{\gamma}}(J)\in \bar{B}_{\frac{\pi}{2}}(v_{\bar{\gamma}})$. For each $x_k\neq 0$, denote by $\bar{J}_k\subset J$ the interval associated to $x_k$ as described in Step 4 above. Take a $t_k\in J_k$, we consider the following circles described below:
\begin{enumerate}
\item If $k \equiv 0 \pmod 4 $, we draw circles $\zeta_k^+$ and $\zeta_k^-$ of radius $\rho_0$ (measured on $\mathbb{S}^2$) tangent to $\bar{\gamma}$ at $\bar{\gamma}(t_k)$ from left and right respectively.
\item If $k \equiv 1 \pmod 4 $, we draw the circle $\zeta_k$ of radius $\rho_0$ tangent to $\bar{\gamma}$ at $\bar{\gamma}(t_k)$ from right.
\item If $k \equiv 2 \pmod 4 $, we draw the circle $\zeta_k^+$ and $\zeta_k^-$ of radius $\rho_0$ tangent to $\bar{\gamma}$ at $\bar{\gamma}(t_k)$ from left and right respectively.
\item If $k \equiv 3 \pmod 4 $, we draw the circle $\zeta_k$ of radius $\rho_0$ tangent to $\bar{\gamma}$ at $\bar{\gamma}(t_k)$ from left.
\end{enumerate}
\noindent We separate the next part into \emph{two} cases. The first case is for $n_\mQ$ even. In addition to the circles that we have defined previously, we also consider circles $\zeta_{-1}$ and $\zeta_{2(n_\mQ+1)+1}$ of radius $\rho_0$ tangent to $\bar{\gamma}$, respectively, at $\bar{\gamma}(0)$ from left and $\bar{\gamma}(L)$ from left. Using polar coordinates with $-v_{\bar{\gamma}}$ as axis, so that the $(-v_\gamma)$-parallel coordinate of the curve $\bar{\gamma}$ is non-decreasing with respect to the parameter of the curve. For each $k$ such that $\zeta_k$ or $\zeta_k^\pm$ is defined, denote by $(\theta_k,\varphi_k)\in [0,\pi]\times (-\pi,\pi)$ center of $\zeta_k$ and $(\theta_k^\pm,\varphi_k^\pm)\in [0,\pi]\times (-\pi,\pi)$ center of $\zeta_k^\pm$. For each $j\in\mathbb{Z}$, when a comparison is possible\footnote{Here we use convention that if the number $a$ is \emph{undefined} and the number $b$ is defined then $\min \{a,b\} = \max\{a,b\}=b$.}, we have:
\begin{align} \label{eqcompare}
\min \{\varphi_{4j}^-,\varphi_{4j+1},\varphi_{4j+2}^-\} & \geq \max\{\varphi_{4j+2}^+,\varphi_{4j+3},\varphi_{4j+4}^+\} \\ \label{eqcompare15}
\max \{\varphi_{4j+2}^+,\varphi_{4j+3},\varphi_{4j+4}^+\} & \leq \min\{\varphi_{4j+4}^-,\varphi_{4j+5},\varphi_{4j+6}^-\}.
\end{align}
\noindent Since the curve $\bar{\gamma}$ is contained in a hemisphere, recall that the distance from $\bar{\gamma}$ to center of $\zeta_k$ is greater or equal to $\rho_0$, so equations above and the manner that $\zeta_k$ are constructed implies, for all $j\in\mathbb{Z}$:
\begin{align} \label{eqcompare2}
\min \{\theta_{4j+2}^+,\theta_{4j+3},\theta_{4j+4}^+\} - \max\{\theta_{4j}^-,\theta_{4j+1},\theta_{4j+2}^-\} & \geq 2\rho_0 \\ \label{eqcompare25}
\min \{\theta_{4j+4}^-,\theta_{4j+5},\theta_{4j+6}^-\} - \max\{\theta_{4j+2}^+,\theta_{4j+3},\theta_{4j+4}^+\} & \geq 2\rho_0. 
\end{align}
\noindent Since we cannot have three consecutive zero values for $x_k$ for $k\in \left\{0,1,2,\ldots, 2\cdot\cindex(\bar{\gamma})\right\}$, using Equation (\ref{eqcompare2}) we have:
$$ L_1=d(p_1,q_1)= \theta_{2(n_\mQ+1)+1}- \ldots -\theta_{-1} \geq (2\rho_0)\cdot(\cindex(\bar{\gamma})+1)\geq (2\rho_0)\cdot (n_\mQ+2).  $$
Now we recall the definition of $\bar{L}_1$ on Equation (\ref{eqLD}). This implies: 
$$\bar{L}_1\geq 2\left\lfloor\frac{2\rho_0(n_\mQ+2)}{4\rho_0}+1\right\rfloor-1=n_\mQ+1 > n_\mQ = \bar{D}_1 = \bar{D}_2.$$
This contradicts with the definition of $n_\mQ$.

Now we analyze the second case: $n_\mQ$ as an odd number. We consider circles $\zeta_{1}$ and $\zeta_{2(n_\mQ+1)+1}$ of radius $\rho_0$ tangent to $\bar{\gamma}$, respectively, at $\bar{\gamma}(0)$ from left and $\bar{\gamma}(L)$ from \emph{right}. Again we use polar coordinates with axis $-v_{\bar{\gamma}}$. Equations (\ref{eqcompare}) and (\ref{eqcompare15}) still hold. Thus this implies Equations (\ref{eqcompare2}) and (\ref{eqcompare25}). Again using the fact that three consecutive zeroes cannot happen for $x_k$ for $k\in\{0,1,2,\ldots,2\cdots\cindex(\bar{\gamma})\}$, we have:
$$D_1=d(p_1,q_2)\geq (2\rho_0)\cdot(n_q+2) .$$
Recalling the definition of $\bar{D}_1$ in Equation (\ref{eqLD}). This implies:
$$\bar{D}_1\geq 2\left\lfloor\frac{2\rho_0\left(2\left\lfloor\frac{n_\mQ}{2}\right\rfloor+3\right)}{4\rho_0}-\frac{1}{2}\right\rfloor=n_\mQ+1 > n_\mQ = \bar{L}_1 = \bar{L}_2 .$$ 
Contradicting the definition of $n_\mQ$.

For the case $x_0$ and $x_1$ equals to zero, the procedure entirely is analogous; We need to show $\bar{L}_2=n_\mQ+1$ for $n_\mQ$ even, and $\bar{D}_2=n_\mQ+1$ for $n_\mQ$ odd. We follows exactly the same steps by drawing tangent circles $\zeta_k$ at $\bar{\gamma}(t_k)$. Additionally we consider circles $\zeta_1$ and $\zeta_{2(n_\mQ+1)+1}$ of radius $\rho_0$ tangent to $\bar{\gamma}$ at $\bar{\gamma}(0)$ and $\bar{\gamma}(L)$. The remaining argument is identical to previous case, so we omit it here.
\end{proof}

\begin{remark}\upshape
As an additional information, one may have noted that by doing the proof of Proposition \ref{propindex} more carefully, it is possible to construct a homotopy from $\bar{\gamma}$ to the critical curve that we have constructed. In fact, if $\gamma\in\mathcal{C}_0$ is a critical curve, $\cindex_0(\gamma)$ is the index of critical curve as defined on page \pageref{defnq}. However it is unnecessary for the result we are going to prove. Now with Proposition \ref{propindex} in hand, we are ready to define the application $\bar{G}$ in the next subsection.
\end{remark}

\subsection{Defining the map $G_\epsilon$}

We follow Step 5 of the construction of $G$ on page \pageref{step4}. First from Proposition \ref{propindex}, we take an $\epsilon$ such that $\cindex_\epsilon(\gamma)\leq n_\mQ$ for all $\gamma\in\mathcal{C}_0$. Now we recall that the sequence $(x_k)_{k\in\mathbb{N}}$ is given by the formulas in (\ref{eqextract}), and defined for this $\epsilon$. We call the sequence $(z_k)_{k\in\mathbb{N}}$ a \emph{good subsequence} of $(x_k)_{k\in\mathbb{N}}$ if it satisfy the following three conditions:
\begin{enumerate}\label{goodsub}
\item There exists an increasing function $\bar{k} : \mathbb{N}\to\mathbb{N}$ such that $ z_i = x_{\bar{k}(i)}$, for all $i\in\mathbb{N}$.
\item The function $\bar{k}$ also satisfy: $ \bar{k}(i)\mod 4 = i\mod 4$.
\item For sequence $(z_k)_{k\in\mathbb{N}}$ if, for some $k\in\mathbb{N}$, $k>0$ and $z_k = z_{k+1} = z_{k+2} = 0$ are satisfied, then, for all integer $l$ such that $l>k$, it holds that $z_l=0$.
\end{enumerate}

\noindent We define the \emph{length} and the \emph{sign} of a good subsequence as follows:
\begin{enumerate}
\item If $ z_k = 0$ for all $k\in\mathbb{N}$, we set $\length \big((z_k)_{k\in\mathbb{N}}\big) = 0$ and $\sign \big((z_k)_{k\in\mathbb{N}}\big) = 0$.
\item If the sequence is not zero and $z_0\neq 0$ or $z_1\neq 0$, we set $\sign \big((z_k)_{k\in\mathbb{N}}\big) = +1$ and:
 $$\length \big((z_k)_{k\in\mathbb{N}}\big) = \left\lceil\frac{\max\{k\in\mathbb{N};z_k\neq 0\}}{2}\right\rceil . $$
\item If the sequence is not zero and $z_0=z_1=0$, we set $\sign \big((z_k)_{k\in\mathbb{N}}\big) = -1$ and:
$$\length \big((z_k)_{k\in\mathbb{N}}\big) = \left\lceil\frac{\max\{k\in\mathbb{N};z_k\neq 0\}}{2}\right\rceil - 1 . $$
\end{enumerate}
\noindent From this definition, the sequence $ (x_k)_{k\in\mathbb{N}} $ is a good subsequence of itself (because of Properties (2) and (3) on page \pageref{step4}). We use the notation $\mathcal{G}_k$ to represent the set of all good subsequences of $(x_k)_{k\in\mathbb{N}}$ that have length $k$. Now we define the sequence $(y_j)_{j\in\mathbb{N}}$ given by the formula below:
\begin{equation} \label{Gcoord}
y_j(\gamma) = \sum_{(z_k)\in\mathcal{G}_{j}} \left(\sign\big((z_k)_{k\in\mathbb{N}}\big) \prod_{z_k\neq 0}z_k \right), \quad j\in\mathbb{N}.
\end{equation}
The application $G_\epsilon : \mathcal{C}_0\to\mathbb{R}^{n_\mQ}$ is defined as:
$$ G_\epsilon(\gamma) = \big(y_1(\gamma), y_2(\gamma),\ldots, y_{n_\mQ}(\gamma)\big), \text{ for all $\gamma\in\mathcal{C}_0$}. $$

\begin{lemma}\label{lemmabound} The application $G_\epsilon$ is continuous and $G_\epsilon(\gamma)\neq 0$ for all $\gamma\in\partial\mathcal{C}_0$.
\end{lemma}

\begin{proof} We start verifying the continuity. For this, it is sufficient to check that each coordinate $y_j(\gamma)$ defined by Formula (\ref{Gcoord}) is a continuous function. We will broaden the definition of the concept of \emph{good subsequence} and use some new notations. For each $j\in\mathbb{N}$, we consider the following set:
$$\mathcal{G}_j = \left\{ (k_0,k_1,\ldots,k_l); k_i\in\mathbb{N} \text{, $(k_i)$ is strictly increasing sequence and satisfies the properties below} \right\} .$$
\begin{enumerate}
\item $k_i \mod 4 \neq k_{i+1} \mod 4 $ for all $i \in \{1,2,\ldots,l\}$
\item We write $k_i= 4q_i+r_i$, with $q_i\in\mathbb{N}$ and $r_i\in\{0,1,2,3\}$. We define $\bar{\sigma}:\{0,1,2,3\}\to \{0,1\}$, with $\bar{\sigma}(0) = \bar{\sigma}(2) = 0$, $\bar{\sigma}(1) = \bar{\sigma}(3) = 1$ and $\sigma:\{0,1,2,3\}\times\{0,1,2,3\}\to\{0,1,2,3\}$ by the table below: 
\begin{center}
\begin{tabular}{ | c | c | c | c | c | }
\hline
$\sigma(a,b)$ & $0$ & $1$ & $2$ & $3$ \\
\hline
$0$ & $0$ & $1$ & $2$ & $3$ \\
\hline
$1$ & $3$ & $0$ & $1$ & $2$ \\
\hline
$2$ & $2$ & $3$ & $0$ & $1$ \\
\hline
$3$ & $1$ & $2$ & $3$ & $0$ \\
\hline
\end{tabular}
\end{center}
\noindent Each line of the table represents a value for $a$ and each column is a value for $b$. We define the \emph{length} of a finite sequence $(k_i)$ by:
$$ \length ((k_i)) = \left\lceil\frac{\bar{\sigma}(r_0)+\sum_{i=0}^{l-1} \sigma(r_i,r_{i+1})}{2}\right\rceil .$$
The second condition is that $\length((k_i))=j$. We also define $\sign (k_i) = +1 $ if $r_i= 0$ or $1$, otherwise we define $\sign (k_i) = -1$
\end{enumerate}
\noindent Also for any sequence $(x_i)_{i\in\mathcal{N}}$ we may define its \emph{good subsequence} using exactly the same conditions as on page \pageref{goodsub}. 

Consider two curves $\alpha$ and $\beta$ in $\mathcal{C}_0$ such that $d(\alpha,\beta)<\delta$. This implies that sets $\Xi_0^+(\alpha)$, $\Xi_0^-(\alpha)$, $\Xi_1^+(\alpha)$, $\Xi_1^-(\alpha)$ and $\Xi_0^+(\beta)$, $\Xi_0^-(\beta)$, $\Xi_1^+(\beta)$, $\Xi_1^-(\beta)$ are close to each other, respectively, in sense that their exclusion is small. To be precise about the last statement we can rewrite these sets into:
\begin{align*}
\Xi_0^+(\alpha) = \bigsqcup_{i\in\mathbb{N}}A_{4i+1}, \quad \Xi_1^+(\alpha) = \bigsqcup_{i\in\mathbb{N}}A_{4i}, \quad \Xi_0^-(\alpha) = \bigsqcup_{i\in\mathbb{N}}A_{4i+3} \quad \text{and} \quad \Xi_1^-(\alpha) = \bigsqcup_{i\in\mathbb{N}}A_{4i+2}.\\
\Xi_0^+(\beta) = \bigsqcup_{i\in\mathbb{N}}B_{4i+1}, \quad \Xi_1^+(\beta) = \bigsqcup_{i\in\mathbb{N}}B_{4i}, \quad \Xi_0^-(\beta) = \bigsqcup_{i\in\mathbb{N}}B_{4i+3} \quad \text{and} \quad \Xi_1^-(\beta) = \bigsqcup_{i\in\mathbb{N}}B_{4i+2}.
\end{align*}
So that the area of $A_i\bigtriangleup B_i$ is small for each $i\in\mathbb{N} $. Now we define sequences $(\bar{x}_{\alpha,i})_{i\in\mathbb{N}}$ and $(\bar{x}_{\beta,i})_{i\in\mathbb{N}}$, with $\bar{x}_{\alpha,i} = \area (A_i) $ and $\bar{x}_{\beta,i} = \area (B_i)$. Note that these sequences are the augmented version of the original sequences $(x_i(\alpha))_{i\in\mathbb{N}}$ and $(x_i(\beta))_{i\in\mathbb{N}}$, in the sense that we have the following 2 equations, for all $j\in\mathbb{N}$:
\begin{align*}
\sum_{(z_k)\in\mathcal{G}_{j}(\alpha)} \left(\sign\big((z_k)_{k\in\mathbb{N}}\big) \prod_{z_k\neq 0}z_k \right) = & \sum_{(z_k)\in\mathcal{G}_{\alpha,j}} \left(\sign\big((z_k)_{k\in\mathbb{N}}\big) \prod_{z_k\neq 0}z_k \right). \\
\sum_{(z_k)\in\mathcal{G}_{j}(\beta)} \left(\sign\big((z_k)_{k\in\mathbb{N}}\big) \prod_{z_k\neq 0}z_k \right) = & \sum_{(z_k)\in\mathcal{G}_{\beta,j}} \left(\sign\big((z_k)_{k\in\mathbb{N}}\big) \prod_{z_k\neq 0}z_k \right).  
\end{align*}

\noindent In the equations above, the sets $\mathcal{G}_j(\alpha)$, $\mathcal{G}_j(\beta)$ stand for the set of good subsequences of length $j$ of $(\bar{x}_{\alpha,i})$, $(\bar{x}_{\beta,i})$, respectively, the sets $\mathcal{G}_{\alpha,j}$, $\mathcal{G}_{\beta,j}$ stand for the set of good subsequences of length $j$ of $(x_i(\alpha))$, $(x_i(\beta))$, respectively. 

So these equations imply:
\begin{align*}
y_j(\alpha)-y_j(\beta) & = \sum_{(z_k)\in\mathcal{G}_{j}(\alpha)} \left(\sign\big((z_k)_{k\in\mathbb{N}}\big) \prod_{z_k\neq 0}z_k \right) - \sum_{(z_k)\in\mathcal{G}_{j}(\beta)} \left(\sign\big((z_k)_{k\in\mathbb{N}}\big) \prod_{z_k\neq 0}z_k \right) \\
& = \sum_{(k_i)\in\bar{\mathcal{G}}_{j}} \sign(k_0) \left(\lambda\big((k_i),(\bar{x}_{\alpha,i})\big)- \lambda\big((k_i),(\bar{x}_{\beta,i})\big) \right) \\
& \leq \sum_{(k_i)\in\bar{\mathcal{G}}_{j}} \left|\lambda\big((k_i),(\bar{x}_{\alpha,i})\big)- \lambda\big((k_i),(\bar{x}_{\beta,i})\big) \right| \\
& \leq \mathcal{O}(\delta),
\end{align*}
\noindent where $\mathcal{O}:(0,\delta_1)\to \mathbb{R}$ is a function such that $\lim_{\delta\to 0}\mathcal{O}(\delta) = 0$.

For the second part of assertion, note that if $\gamma\in\partial\mathcal{C}_0$ then
$$ \gamma\cap\Xi_0 \neq \emptyset \quad \text{or} \quad \left\{\begin{array}{l} \vt_\gamma\cap\Xi_1^+\neq\emptyset \\
\vt_\gamma\cap\Xi_1^-\neq\emptyset
\end{array}\right.. $$
\noindent So the sequence $(x_j(\gamma))_{j\in\mathbb{N}}$ constructed is so that the $\cindex_\epsilon(\gamma)\geq 1$ and by Proposition \ref{propindex} we have $\cindex_\epsilon(\gamma)\leq n_\mQ$. We denote $i_\gamma = \cindex_\epsilon(\gamma)$. By a direct computation we obtain $ y_{i_\gamma} \neq 0 $. This implies $G_\epsilon(\gamma)=(y_1,\ldots,y_{i_\gamma}, 0, \ldots, 0)\neq 0$.
\end{proof}

As an immediate consequence of Lemma \ref{lemmabound}, we have:
\begin{corollary}\label{endg} There exists a $R_0>0$ such that $|G_\epsilon(\gamma)|>R_0$ for all $\gamma\in\partial\mathcal{C}_0$.
\end{corollary}

So take $R_0$ from the lemma above, we define $\bar{G}:\mathcal{C}_0\to\bar{B}_1(0)\subset\mathbb{R}^{n_\mQ}$ by setting:
$$\bar{G}(\gamma)=\left\{\begin{array}{ll} \dfrac{1}{R_0} G_\epsilon(\gamma) \quad & \text{if $ |G_\epsilon(\gamma)|\leq R_0 $.} \\[1.2em] 
 \dfrac{1}{|G_\epsilon(\gamma)|} G_\epsilon(\gamma) \quad & \text{if $ |G_\epsilon(\gamma)|> R_0 $.} 
\end{array}\right.$$

\noindent Consider the surjective map $p:\bar{B}_1(0)\to\mathbb{S}^{n_\mQ}$, defined as:
 $$ p\left(a_1,a_2,\ldots,a_{n_\mQ}\right) = \left\{
 \begin{array}{lrr} \dfrac{1}{(n_\mQ)^\frac{1}{2}}\left(\sin (\pi a_1),\sin(\pi a_2),\ldots,\sin(\pi a_{n_\mQ}),\left[\sum_{i=1}^{n_\mQ}\cos^2(\pi a_i)\right]^{\frac{1}{2}}\right) \mkern-188mu &   & \\  & \text{if $ \sum_{i=1}^{n_\mQ}\cos^2 (\pi a_i) \geq 0 $. }& \\[1.2em]
 \dfrac{1}{(n_\mQ)^\frac{1}{2}}\left(\sin (\pi a_1),\sin(\pi a_2),\ldots,\sin(\pi a_{n_\mQ}),-\left[-\sum_{i=1}^{n_\mQ}\cos^2 (\pi a_i)\right]^{\frac{1}{2}}\right) \mkern-188mu&   & \\  & \text{if $ \sum_{i=1}^{n_\mQ}\cos^2(\pi a_i) < 0 $. }&
 \end{array}\right.
 $$

\noindent We put $G:\Lspace\to\mathbb{S}^{n_\mQ}$ as:
\begin{equation*}
G(\gamma)=\left\{\begin{array}{ll} \left(p\circ\bar{G}\right)(\gamma) \quad & \text{if $ \gamma\in\mathcal{C}_0 $.} \\ 
 G(\gamma) = (0,0,\ldots,0,-1) \quad & \text{if $\gamma\in\Lspace\smallsetminus\mathcal{C}_0 $.} 
\end{array}\right.
\end{equation*}
\noindent This is the definition for $G$, which concludes the second part of the proof of the main theorem.

\newpage
\section{Non-triviality of maps $F$ and $G$}\label{sec:gcircf}

This is the last part of the proof of the main theorem. We shall verify that the composition $G\circ F:\mathbb{S}^{n_\mQ}\to\mathbb{S}^{n_\mQ}$ has degree $1$, with the applications $F$ and $G$ defined in Sections \ref{sec:mapF} and \ref{sec:mapG}. Thus $[F]\in\ho_{n_\mQ}\big(\Lspace\big)$ is a non-zero element. Since we saw in the end of the Section \ref{sec:mapF} that $[\boldsymbol{i}\circ F]$ is trivial in $\ho_k\big(\mathcal{I}(\mQ)\big)$, this implies that $\Lspace$ is homotopically equivalent to $\Omega\mathbb{S}^3\vee\mathbb{S}^{n_\mQ}\vee E $ (where $E$ is a space yet to be discovered).

\subsection{The degree of $G\circ F:\mathbb{S}^{n_\mQ}\to\mathbb{S}^{n_\mQ}$}

First we use some of notations as in the last two Sections. We recall the definition of the application $F$ in Section \ref{sec:mapF}. Take an $a\in\mathbb{R}^{n_\mQ}$, $a=(a_1,a_2,\ldots,a_{n_\mQ})$. $\gamma = \bar{F}(a)$ is defined as a concatenation of $(n_\mQ+1)$ curves $\alpha_0,\ldots,\alpha_{n_\mQ}$. We also recall that in Section \ref{sec:mapG}, for $\gamma\in\mathcal{C}_0$, in order to obtain $G(\gamma)$, we constructed sequences $x=\big(x_i(\gamma)\big)_{i\in\mathbb{N}}$ and $y=\bar{G}\big(\bar{F}(a)\big)=\big(y_i(\gamma)\big)_{n\in\mathbb{N}}$. 

Note that if $a\in\mathbb{R}^{n_\mQ}$ is such that $\gamma\in\Lspace\smallsetminus\mathcal{C}_0$, then $G(\gamma)=(0,0,\ldots,0,-1)$. Thus we shall focus on $a\in\mathbb{R}^{n_\mQ}$ such that $\gamma\in\mathcal{C}_0$. From now on, we will assume that $\gamma\in\mathcal{C}_0$. We extract from $F(a)$ a ``non-reduced'' sequence $w=(w_0,w_1,\ldots,w_{8n_\mQ-1})$ such that its reduced version is $z$. The sequence $w$ is defined as follows.

\vspace{\cvspace}
\noindent \textbf{Step 1:} We consider the sets $\Xi_0^+$, $\Xi_0^-$, $\Xi_1^+$ and $\Xi_1^-$ defined for $\gamma$. For each $k\in\{0,\ldots,n_\mQ\}$, we consider the following subsets of $\mathbb{R}$, $\tj_{k,0}$, $\tj_{k,1}$, $\tj_{k,2}$, $\tj_{k,3}$, $\tj_{k,4}$ and $\tj_{k,5}$, defined by the following properties:
\begin{enumerate}
\item $\tj_{k,0}<\tj_{k,1}<\tj_{k,2}<\tj_{k,3}<\tj_{k,4}<\tj_{k,5}$.
\item $\alpha_k\big(\tj_{k,0}\big)=\Xi_1^+\cap\big(\img (\vt_{\alpha_k})\big)$ or $\emptyset$, $\alpha_k\big(\tj_{k,1}\big)=\Xi_0^+\cap\big(\img (\alpha_k)\big)$ of $\emptyset$, $\alpha_k\big(\tj_{k,2}\big)=\Xi_1^-\cap\big(\img (\vt_{\alpha_k})\big)$, $\alpha_k\big(\tj_{k,3}\big)=\Xi_0^-\cap\big(\img (\alpha_k)\big)$, $\alpha_k\big(\tj_{k,4}\big)=\Xi_1^+\cap\big(\img (\vt_{\alpha_k})\big)$ or $\emptyset$ and $\alpha_k\big(\tj_{k,5}\big)=\Xi_0^+\cap\big(\img (\alpha_k)\big)$ or $\emptyset$. 
\item If $\tj_{k,2}$ and $\tj_{k,3}$ are both empty sets then $\tj_{k,4}$ and $\tj_{k,5}$ are both empty.
\end{enumerate}
The existence of the sets with 3 properties above is due to the following proposition:
\begin{proposition} For each $k\in\{0,\ldots,n_\mQ\}$, the following 2 cases cannot happen for $\alpha_k$:
\begin{enumerate}
\item There exist numbers $s_1<s_2<s_3$ such that the next 3 properties are satisfied: 
\begin{enumerate}
\item $\alpha_k(s_1)\in\Xi^+_0$ or $\vt_{\alpha_k}(s_1)\in\Xi^+_1$.
\item $\alpha_k(s_2)\in\Xi^-_0$ or $\vt_{\alpha_k}(s_2)\in\Xi^-_1$.
\item $\alpha_k(s_3)\in\Xi^+_0$ or $\vt_{\alpha_k}(s_3)\in\Xi^+_1$.
\end{enumerate}
\item There exist numbers $s_1<s_2<s_3$ such that the next 3 properties are satisfied: 
\begin{enumerate}
\item $\alpha_k(s_1)\in\Xi^-_0$ or $\vt_{\alpha_k}(s_1)\in\Xi^-_1$.
\item $\alpha_k(s_2)\in\Xi^+_0$ or $\vt_{\alpha_k}(s_2)\in\Xi^+_1$.
\item $\alpha_k(s_3)\in\Xi^-_0$ or $\vt_{\alpha_k}(s_3)\in\Xi^-_1$.
\end{enumerate}
\end{enumerate}
\end{proposition}
This proposition may be directly verified by a careful examination on the definition of $\alpha_k$'s from the construction of $F$ on page \pageref{segment}.

\vspace{\cvspace}
\noindent \textbf{Step 2:} For each $k\in\{0,\ldots,n_\mQ\}$ and $i\in\{1,5\}$, we define $\mathcal{A}^+_{\alpha_k,i}$ as empty if $\tj_{k,i}$ is empty. Otherwise we define it as the the region on $\mathbb{S}^2$ delimited by $\alpha_k\big(\tj_{k,i}\big)$, $\partial\Xi_0^+$ and the only two circles centered at $v_\gamma$ passing through each of the endpoints of $\alpha_k|_{\tj_{k,i}}$. We also define, for each $k\in\{0,\ldots,n_\mQ\}$, the set $\mathcal{A}^+_{\alpha_k,3}$ as empty if $\tj_{k,3}$ is empty. Otherwise we define it as the the region on $\mathbb{S}^2$ delimited by $\alpha_k\big(\tj_{k,3}\big)$, $\partial\Xi_0^-$ and the only two circles centered at $v_\gamma$ passing through each of the two endpoints of $\alpha_k|_{\tj_{k,3}}$.

\vspace{\cvspace}
\noindent \textbf{Step 3:} We set $w_k=\length\big(\vt_{\alpha_k}|_{\tj_{k,0}}\big)$, $w_{k+1}=\area\big(\mathcal{A}_{\alpha_k,1}^+\big)$, $w_{k+2}=\length\big(\vt_{\alpha_k}|_{\tj_{k,2}}\big)$, $w_{k+3}=\area\big(\mathcal{A}_{\alpha_k,3}^-\big)$, $w_{k+4}=\length\big(\vt_{\alpha_k}|_{\tj_{k,4}}\big)$, $w_{k+5}=\area\big(\mathcal{A}_{\alpha_k,5}^+\big)$, $w_{k+6}=w_{k+7}=0$. This defines a sequence $w=(w_0,w_1,\ldots,w_{8n_\mQ-1})$. Let $\mathcal{G}_j(w)$ and $\mathcal{G}_j(x)$ be the sets of \emph{good subsequences} of $w$ and $x$ respectively (see definition on page \pageref{goodsub}), it is easy to check that:
\begin{equation}\label{eqxeqw}
y_j=\sum_{(z_k)\in\mathcal{G}_{j}(x)} \left(\sign\big((z_k)_{k\in\mathbb{N}}\big) \prod_{z_k\neq 0}z_k \right) =\sum_{(z_k)\in\mathcal{G}_{j}(w)} \left(\sign\big((z_k)_{k\in\mathbb{N}}\big) \prod_{z_k\neq 0}z_k \right), \quad j\in\{1,\ldots,n_\mQ\} .
\end{equation}

\noindent So to understand the behavior of $y_j$'s, we will study the sequence $w$ in the next step.


\vspace{\cvspace}
\noindent \textbf{Step 4:} For each $k\in\{0,1,\ldots,n_\mQ\}$ we say that the curve $\alpha_k$ is of type:
\begin{enumerate}
\item $+-$ if $w_{8k+0}> 0$ or $w_{8k+1}>0$, $w_{8k+2}>0$ or $w_{8k+3}>0$ and $w_{8k+i}=0$ for all $i\in\{4,5,6,7\}$.
\item $-+$ if $w_{8k+2}> 0$ or $w_{8k+3}>0$, $w_{8k+4}>0$ or $w_{8k+5}>0$ and $w_{8k+i}=0$ for all $i\in\{0,1,6,7\}$.
\item $+$ if $w_{8k+0}> 0$ or $w_{8k+1}>0$ and $w_{8k+i}=0$ for all $i\in\{2,3,4,5,6,7\}$.
\item $-$ if $w_{8k+2}> 0$ or $w_{8k+3}>0$ and $w_{8k+i}=0$ for all $i\in\{0,1,4,5,6,7\}$.
\item $0$ if $w_{8k+i}=0$ for all $i\in\{0,1,\ldots,7\}.$
\end{enumerate}
 Now we observe the behavior of $w$ based on the values of $a$. We have the following relations:
\begin{enumerate}
\item If $a_1\geq 0$ then $\alpha_0$ is of type $-+$, $+$ or $0$. Moreover, there exists a constant $C$ such that $a_1>C$ implies $\alpha_0$ is of type $-+$ or $+$.
\item If $a_1\leq 0$ then $\alpha_0$ is of type $+-$, $-$ or $0$. Moreover, there exists a constant $C$ such that $a_1<C$ implies $\alpha_0$ is of type $+-$ or $-$.
\item If, for $k\in\{1,\ldots,n_\mQ\}$, $a_{k-1}\geq 0$ and $a_{k}\geq 0$ then $\alpha_k$ is of type $+$ or $0$. Moreover, there exists a constant $C>0$ such that $\max (a_{k-1},a_k)>C$ implies $\alpha_k$ is of type $+$.
\item If, for $k\in\{1,\ldots,n_\mQ\}$, $a_{k-1}\geq 0$ and $a_{k}\leq 0$ then $\alpha_k$ is of type $+-$, $+$, $-$ or $0$. Moreover, there exists constants $C_1$ and $C_2$ such that, if $a_{k-1}>C_1$ then $\alpha_k$ is of type $+-$ or $+$ and if $a_k<C_2$ then $\alpha_k$ is of type $+-$ or $-$.
\item If, for $k\in\{1,\ldots,n_\mQ\}$, $a_{k-1}\leq 0$ and $a_{k}\geq 0$ then $\alpha_k$ is of type $-+$, $-$, $+$ or $0$. Moreover, there exists constants $C_1$ and $C_2$ such that, if $a_{k-1}<C_1$ then $\alpha_k$ is of type $-+$ or $-$ and if $a_k>C_2$ then $\alpha_k$ is of type $-+$ or $+$.
\item If, for $k\in\{1,\ldots,n_\mQ\}$, $a_{k-1}\leq 0$ and $a_{k}\leq 0$ then $\alpha_k$ is of type $-$ or $0$. Moreover, there exists a constant $C<0$ such that $\min (a_{k-1},a_k)<C$ implies $\alpha_k$ is of type $-$.
\item If $a_{n_\mQ}\geq 0$ then $\alpha_{n_\mQ}$ is of type $+-$, $+$ or $0$. Moreover, there exists a constant $C$ such that $a_{n_\mQ}>C$ implies $\alpha_{n_\mQ}$ is of type $+-$ or $+$.
\item If $a_{n_\mQ}\leq 0$ then $\alpha_{n_\mQ}$ is of type $-+$, $-$ or $0$. Moreover, there exists a constant $C$ such that $a_{n_\mQ}<C$ implies $\alpha_{n_\mQ}$ is of type $-+$ or $-$. 
\end{enumerate}

\vspace{\cvspace}
\noindent \textbf{Step 5:} Given two finite sequences $b=(b_0,b_1,\ldots,b_7)$ and $c=(c_0,c_1,\ldots,c_7)$. If both $b$ and $c$ are of type $+-$, we write $b\succeq c$ if $b_i\geq c_i$ for all $i\in\{0,1,2,3\}$. Moreover, if $b_j>c_j$ for some $j\in\{0,1,2,3\}$, we write $b\succ c$. 

If both $b$ and $c$ are of type $-+$, we also write $b\preceq c$ if $b_i\geq c_i$ for all $i\in\{2,3,4,5\}$. Moreover, if $b_j>c_j$ for some $j\in\{2,3,4,5\}$, we write $b\prec c$.
Given two numbers $a,\ba\in\mathbb{R}^{n_\mQ}$, $a=\big(a_1,a_2,\ldots,a_{n_\mQ}\big)$ and $\ba=\big(\ba_1,\ba_2,\ldots,\ba_{n_\mQ}\big)$ such that $F(a)$ and $F(\ba)$ are both of maximal index. If for some $i\in\{0,1,2,\ldots,n_\mQ\}$, $a_j=\ba_j$ for all $j\neq i$ and $a_i>\ba_i>0$ then:
\begin{align*}
\big(w_{8(i-1)}(a),w_{8(i-1)+1}(a),\dots,w_{8(i-1)+7}(a)\big) & \prec & & \big(w_{8(i-1)}(\ba),w_{8(i-1)+1}(\ba),\dots,w_{8(i-1)+7}(\ba)\big), \\
\big(w_{8i}(a),w_{8i+1}(a),\dots,w_{8i+7}(a)\big) & \succ & & \big(w_{8i}(\ba),w_{8i+1}(\ba),\dots,w_{8i+7}(\ba)\big), \\
\big(w_{8j}(a),w_{8j+1}(a),\dots,w_{8j+7}(a)\big) & \preceq & & \big(w_{8j}(\ba),w_{8j+1}(\ba),\dots,w_{8j+7}(\ba)\big) \\ 
& & & \quad \forall j\neq i, \text{$j$ with the same parity as $(i-1)$}, \\
\big(w_{8j}(a),w_{8j+1}(a),\dots,w_{8j+7}(a)\big) & \succeq & & \big(w_{8j}(\ba),w_{8j+1}(\ba),\dots,w_{8j+7}(\ba)\big) \\
& & & \quad \forall j\neq i, \text{$j$ with the same parity as $i$}.
\end{align*}
 On the other hand, if for some $i\in\{0,1,2,\ldots,n_\mQ\}$, $a_j=\ba_j$ for all $j\neq i$ and $a_i<\ba_i<0$ then:
\begin{align*}
\big(w_{8(i-1)}(a),w_{8(i-1)+1}(a),\dots,w_{8(i-1)+7}(a)\big) & \succ & & \big(w_{8(i-1)}(\ba),w_{8(i-1)+1}(\ba),\dots,w_{8(i-1)+7}(\ba)\big), \\
\big(w_{8i}(a),w_{8i+1}(a),\dots,w_{8i+7}(a)\big) & \prec & & \big(w_{8i}(\ba),w_{8i+1}(\ba),\dots,w_{8i+7}(\ba)\big), \\
\big(w_{8j}(a),w_{8j+1}(a),\dots,w_{8j+7}(a)\big) & \succeq & & \big(w_{8j}(\ba),w_{8j+1}(\ba),\dots,w_{8j+7}(\ba)\big) \\ 
& & & \quad \forall j\neq i, \text{$j$ with the same parity as $(i-1)$}, \\
\big(w_{8j}(a),w_{8j+1}(a),\dots,w_{8j+7}(a)\big) & \preceq & & \big(w_{8j}(\ba),w_{8j+1}(\ba),\dots,w_{8j+7}(\ba)\big) \\
& & & \quad \forall j\neq i, \text{$j$ with the same parity as $i$}.
\end{align*}

\vspace{\cvspace}
\noindent \textbf{Step 6:} By Sard's Theorem for $\bar{G}\circ F$, we take a regular value $y=(y_1,\ldots,y_{n_\mQ})$ close to the axis generated by the vector $(0,\ldots,0,1)$ such that $0<|y|<1$ (that implies $\gamma\in\mathcal{C}_0$, see page \pageref{endg}) and $y_{n_\mQ}\neq 0$. From Equation (\ref{Gcoord}):
$$ y_{n_\mQ} = \sum_{(z_k)\in\mathcal{G}_{n_\mQ}} \left(\sign\big((z_k)_{k\in\mathbb{N}}\big) \prod_{z_k\neq 0}z_k \right)\neq 0.$$

\noindent This implies that the curve is of type $+-+-+-\cdots$ or $-+-+-+\cdots$ of maximal index, that is $n_\mQ$. We shall prove that there is only one $a\in\mathbb{R}^{n_\mQ}$ such that $\bar{G}\circ F(a)=y$. We recall Equation (\ref{eqxeqw}), for each $j\in\{1,2,\ldots,n_\mQ\}$: 
\begin{align*}
y_j & = \sum_{(z_k)\in\mathcal{G}_{j}(w)} \left(\sign\big((z_k)_{k\in\mathbb{N}}\big) \prod_{z_k\neq 0}z_k \right) \\
& = \sum_{(k_i)\in\tilde{\mathcal{G}}_{j}} \left(\sign\big((k_i)\big) \prod_{i}w_{k_i}\big(a\big) \right) \\
& \simeq \sum_{(l_i)\in\mathcal{S}_{j}} \left((-1)^{l_1} \prod_{i}f_{l_i}\big(a_{l_i}\big) \right) \\
& \simeq \sum_{(l_i)\in\mathcal{S}_{j}} \left((-1)^{l_1} \prod_{i}a_{l_i} \right)
\end{align*}

The third and fourth lines above are a homotopic equivalence, where $\tilde{\mathcal{G}}_j$ is the set of all good subsequences of length $j$ of the sequence $(0,1,2,\ldots,8n_\mQ-1)$ and $S_j$ are strictly increasing subsequences of $(1,2,\ldots,n_\mQ)$ that have length $j$. $f_{k_i}$ is a non-decreasing function and there exists a $R>0$ such that $f_{k_i}(t)>R$ for all $t$ sufficiently large and $f_{k_i}(t)<-R$ for all $t$ sufficiently small. The third line follows from the properties in Step 5. Since there is only one $\bar{G}(F(a))=y$, we deduce that $\bar{G}\circ F$ has degree $1$, and thus $G\circ F$ has degree $1$.

\ifnum\aditcont=1 

\newpage
\section{Appendix: on the hypothesis of the main theorem and related topics} \label{sec:def}
We first present a criterion to determine for which $\mQ\in\SO$, the space $\Lspace$ is homotopically equivalent to $\mathcal{I}(\mI,\mQ)$. In the second subsection we give an explicit method to calculate the length of a CSC curve in a space of the type $\bar{\mathcal{L}}_\rho(\mI,\mQ)$. In the last subsection, we present a version of proof to show that the length-minimizing curve in $\cLspace$ is composed only by arc of circles of radius $\rho_0$ and geodesics. This is a property to Theorem \ref{perez} proven by F. Monroy-Pérez in \cite{monroy}.

\subsection{Cases in which $\boldsymbol{i}:\Lspace\hookrightarrow\mathcal{I}(\mI,\mQ)$ is a homotopical equivalence }\label{trivialresult}

Here we give a sufficient condition for the natural inclusion $\boldsymbol{i}:\Lspace\hookrightarrow\mathcal{I}(\mI,\mQ)$ to be a homotopical equivalence. More precisely, we show that for some $\mQ\in\SO$ there is an obvious way to add loops simultaneously and continuously to all curves in $\Lspace$ which is a sufficient condition for the equivalence. 

We reuse the notation for the rotation matrix defined in Equation (\ref{eqrotation}). For each $\theta\in(\rho_0,\pi-\rho_0)$, $\vartheta\in (\rho_0,\pi-\rho_0)$ and $\rho\in[0,2\pi)$ consider $\tilde{\mQ}\in\SO$ given by:
\begin{equation}\label{eqQ}
\tilde{\mQ}(\theta,\vartheta,\rho)=\left(\begin{array}{ccc} | & | & | \\
\big[(\rot_\rho(v)\big](p) & \big[(\rot_{\rho+\frac{\pi}{2}}(v)\big](q) & \big[(\rot_\rho(v)\big](p) \times \big[(\rot_{\rho+\frac{\pi}{2}}(v)\big](q) \\
| & | & |
\end{array} \right),
\end{equation}

\noindent where $v=(-\cos\theta,0,-\sin\theta)$, $p=(\cos(\theta+\vartheta),0,\sin(\theta+\vartheta))$ and $q=(-\sin\theta , 0, \cos\theta)$. Consider the subset of $\SO$ consisting of all matrices above, that is:

$$ \mathfrak{C} = \big\{\tilde{\mQ}(\theta,\vartheta,\rho)\in\SO  ;  \theta\in(\rho_0,\pi-\rho_0), \vartheta\in (\rho_0,\pi-\rho_0),\rho\in[0,2\pi)\big\}.$$

\begin{figure}[htb]
\centering
\begin{tikzpicture}
 \node[anchor=south west,inner sep=0] at (-4,0) {\includegraphics[width=8cm]{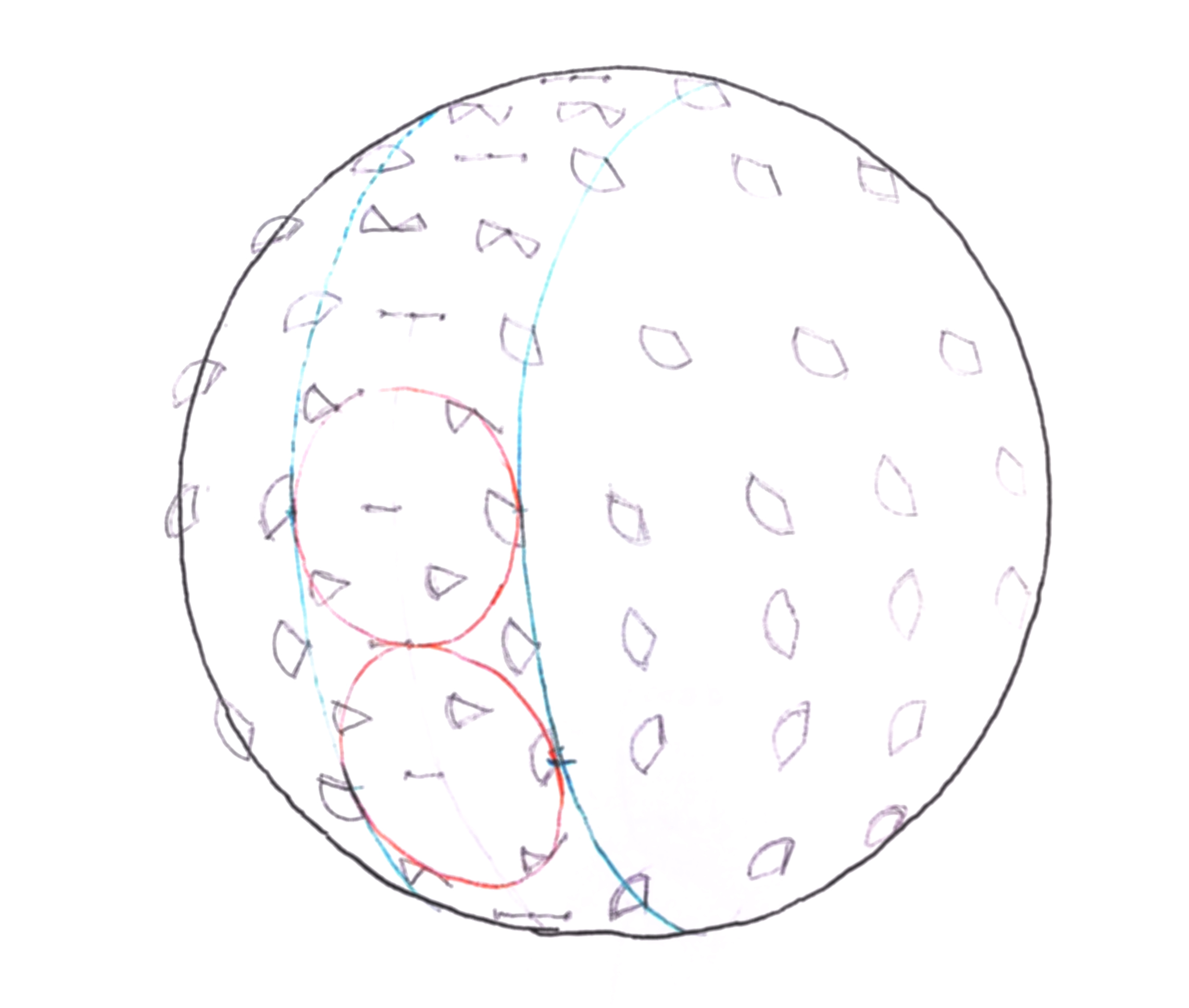}};
  \node[anchor=south west,inner sep=0] at (4,0) {\includegraphics[width=8cm]{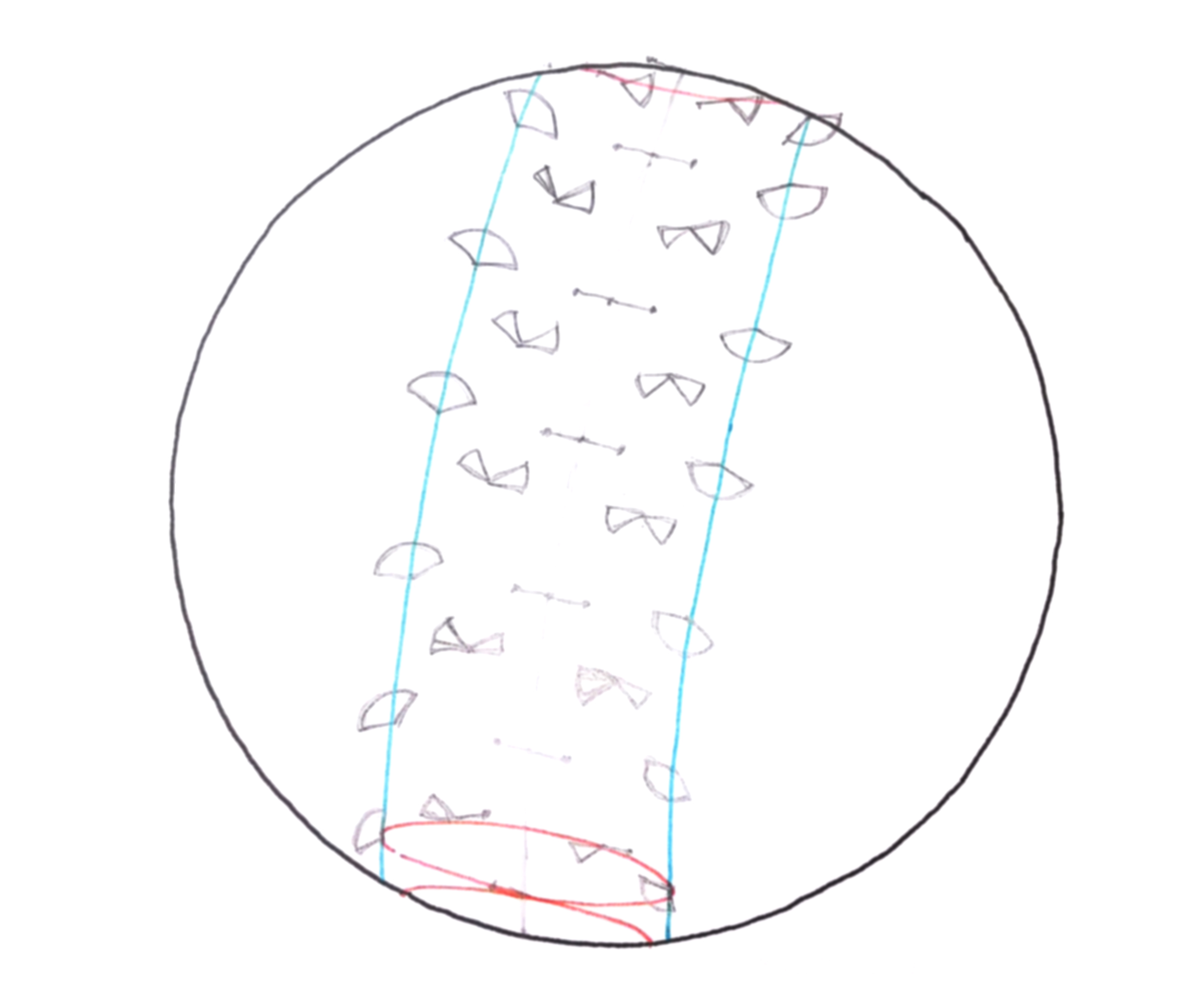}};
\end{tikzpicture}
\caption{This is an illustration of matrices in $\mathfrak{C}$ viewed as tangent vectors with the base point on the sphere. In the images above, the red circles are left and right tangent circles at $\mI$ of radius $\rho$, the fan shaped pieces on the surface of the sphere represents all possible tangent vectors that are in $\mathfrak{C}$.}
\label{fig:leques}
\end{figure}

We view $\SO$ as the unit tangent bundle on the sphere. For each $\tilde{\mQ}(\theta,\vartheta,\rho)\in\mathfrak{C}$ there is an obvious axis $v = (-\cos\theta,0,-\sin\theta)$ in which the unitary vector  field $V$, defined in $\mathbb{S}^2-\{v,-v\}$, tangent to the anticlockwise rotation around this axis satisfy the following three properties:
\begin{enumerate}
\item $V(e_1)= -e_2$.
\item $V(\tilde{\mQ}e_1) = \tilde{\mQ}e_2$.
\item $\rho < d(v,e_1),d(v,\tilde{\mQ}e_1) < \pi-\rho$.
\end{enumerate}

This property guarantees that we can attach the same arcs of circles on the endpoints of all curves in $\mathcal{L}_{\rho}(\tilde{\mQ})$ simultaneously. More details are in the demonstration of the proposition below.

\begin{proposition} If $\mQ\in\mathfrak{C}$ then the natural inclusion $\boldsymbol{i}:\Lspace\hookrightarrow\mathcal{I}(\mI,\mQ)$ is a homotopic equivalence.
\end{proposition}

\begin{proof}
Given a $\mQ\in\mathfrak{C}$, there exists $\theta\in(\rho_0,\pi-\rho_0)$, $\vartheta\in (\rho_0,\pi-\rho_0)$ and $\rho\in[0,2\pi)$ such that Equation (\ref{eqQ}) holds. For each continuous map $f:K\to\Lspace$ and given a $p\in K$, denote $f(p)=\gamma:[0,1]\to\mathbb{S}^2$ and $\gamma\in\Lspace$. Consider the family of curves given by $\gamma_\tau$, where $\tau\in [0,+\infty)$ defined in the equation below.
$$ \gamma_\tau (t) = \left\{ \begin{array}{ll} \zeta_{-v,\theta,\circlearrowleft}(t\tau) \quad & t\in[0,1] \\
\big[\rot_\tau(-v)\big](\gamma(t-1)) \quad & t\in[1,2] \\
\zeta_{-v,\theta+\vartheta,\circlearrowright}((t-2)\tau) \quad & t\in[2,3]
\end{array}
\right. $$
The first and the last curve in the concatenation above are arc of circles, for definitions, refer Equations (\ref{circle1}) and (\ref{circle2}). These circles are reparametrized so that the endpoint of $\zeta_{-v,\theta,\circlearrowleft}$ is $\big[\rot_\tau(-v)\big](\gamma(0))$ and the start point of $\zeta_{-v,\theta+\vartheta,\circlearrowright}$ is $\big[\rot_\tau(-v)\big](\gamma(1))$.

Thus, for any integer $n$, we have a homotopy between $f$ and $f^{[t_0\# 2n]}$ in $\Lspace$. This homotopy is defined by $H_n:K\times [0,2n\pi+1]\to \Lspace$, 
$$H(p,\tau)=\left\{\begin{array}{ll} \gamma_\tau \quad & \tau\in [0,2n\pi] \\ 
\tilde{\gamma}_{n,\tau} \quad & \tau\in [2n\pi,2n\pi+1]
\end{array}\right. $$
where $\tilde{\gamma}_{n,\tau}$ is sliding the loops $\zeta_{-v,\theta,\circlearrowleft}$ and $\zeta_{-v,\theta+\vartheta,\circlearrowright}$ to the position $t_0$ then transform them into loops of great circles.

By Proposition \ref{prop64}, if $f:K\to\Lspace$ is homotopic to a constant in $\mathcal{I}(\mI,\mQ)$, then there exists an $n\geq 1$ such that $f^{[t_0\# 2n]}$ is homotopic to a constant in $\Lspace$. Since $f^{[t_0\# 2n]}$ and $f$ are homotopics in $\Lspace$. This proves that $f$ is homotopic to a constant in $\Lspace$.

Conversely, it is trivial that any $f:K\to\Lspace$ is homotopic to a constant in $\Lspace$ imply that $f$ is homotopic to a constant in $\mathcal{I}(\mI,\mQ)$. So the inclusion map $\boldsymbol{i}:\Lspace\hookrightarrow \mathcal{I}(\mI,\mQ)$ is a homotopic equivalence.
\end{proof}

\subsection{Oriented circles and some basic properties}

In the hypothesis of the main theorem we considered a CSC curve. The purpose of this subsection is to define concepts and basic properties to compute the length of a CSC curve.
\begin{definition}[Oriented circle] An \emph{oriented circle} on $\mathbb{S}^2$ is a curve given by 
$$\mathcal{C}(\eta) = M(\sin\xi,\cos\xi\cos\eta,\cos\xi\sin\eta) \text{, with $\eta\in [0,2\pi]$ and $M\in\SO$.}$$ 
Here $\xi\in(0,\pi)$ is the radius of the circle (measured on sphere).
\end{definition}

Provided an oriented circle of radius $\rho_0$, we look at the vector field generated by tangent vectors of geodesic segment of length $\pi$ starting tangentially at the circle. This vector field is defined in entire $\mathbb{S}^2$ except the two open discs of radius $\rho_0$. We denote $\mathcal{C}_1(v)$ the \emph{counter-clockwise} oriented circle centered at the point $v\in\mathbb{S}^2$ with radius $\rho_0$, in the same manner, we use $\mathcal{C}_2(v)$ to denote the \emph{clockwise} oriented circle centered at $v$ with radius $\rho_0$. The following proposition is straightforward.

\begin{proposition}[properties of oriented circles] Given two oriented circles of the same radius $0<r<\frac{\pi}{2}$ and the opposite orientation on sphere $\mathbb{S}^2$, $\mathcal{C}_1(p)$ and $\mathcal{C}_2(q)$, then:
\begin{itemize}
\item If $d\left(p,q\right)<2r$ then there is no geodesic tangent to both $\mathcal{C}_1(p)$ and $\mathcal{C}_2(q)$ with the same orientation as in both circles.
\item If $d\left(p,q\right)\geq 2r$ and $q\neq -p$ then there are two geodesics tangent to both $\mathcal{C}_1(p)$ and $\mathcal{C}_2(q)$ with the same orientation as in both circles.
\item If $q=-p$ then every geodesic tangent to $\mathcal{C}_1(p)$ is also tangent to $\mathcal{C}_2(q)$. In this case, $\mathcal{C}_1(p)=-\mathcal{C}_2(q)$.
\end{itemize}
Given two oriented circles of the same radius $0<r<\frac{\pi}{2}$ and the same orientation on $\mathbb{S}^2$, $\mathcal{C}_1(p)$ and $\mathcal{C}_1(q)$, then:
\begin{itemize}
\item If $d\left(p,-q\right)<2r$ then there is no geodesic tangent to both $\mathcal{C}_1(p)$ and $\mathcal{C}_1(q)$ with the same orientation as in both circles.
\item If $d\left(p,-q\right)\geq 2r$ and $q\neq p$ then there are two geodesics tangent to both $\mathcal{C}_1(p)$ and $\mathcal{C}_1(q)$ with the same orientation as in both circles.
\item If $q=p$ then every geodesic tangent to $\mathcal{C}_1(p)$ is also tangent to $\mathcal{C}_1(q)$. In this case, $\mathcal{C}_1(p)=\mathcal{C}_1(q)$.
\end{itemize}
\end{proposition}

We are interested in studying such vector field defined for the following circles $\mathcal{C}_1(\p1)$, $\mathcal{C}_2(\p2)$, $\mathcal{C}_1(\q1)$ and $\mathcal{C}_2(\q2)$. Note that the first two circles are tangent to each other at the point $e_1$ and have direction $e_2$, and the last two circles are tangents at $Qe_1$ with tangent direction $Qe_2$. For simplicity in the next theorem and its proof we denote:
\begin{align*}
\mathcal{C}_1 & \coloneqq \mathcal{C}_1(\p1) & \mathcal{C}_2 &\coloneqq \mathcal{C}_2(\p2) \\
\mathcal{C}_3 & \coloneqq \mathcal{C}_1(\q1) & \mathcal{C}_4 &\coloneqq \mathcal{C}_2(\q2)
\end{align*}

\subsection{Computing the length of candidates for $\gamma_0$}
We re-enunciate the adapted version of Theorem \ref{perez} by F. Monroy-Perez in \cite{monroy}.

\begin{theorem} Let $\rho\in \left(0,\frac{\pi}{2}\right]$ and $\kappa=\cot\rho$. Every length-minimizing curve in $\bar{\mathcal{L}}_\rho(\mI,\mQ)$ is a concatenation of at most three pieces of arcs with constant curvature equal to $+\kappa$, $-\kappa$ and $0$. Moreover,
\begin{enumerate}
\item If the length-minimizing curve contains a geodesic arc, then it is of the form CSC.
\item If the length-minimizing curve is of the form CCC. Let $\alpha$, $\lambda$ and $\beta$ be angles of the first, the second and the third arc respectively. Then 
\begin{enumerate}
\item $\min\{\alpha,\beta\}<\pi\sin\rho$.
\item $ \lambda > \pi $.
\item $\max\{\alpha,\beta\}<\lambda $.
\end{enumerate}
\end{enumerate}
\end{theorem}

To determine whether a Dubins' curve is unique, we need to compare the length of each candidate. Here we assume that $\langle q_1, e_2 \rangle > 0$, $\langle q_2, e_2 \rangle > 0$ and $n_\mQ\geq 1$ (this implies $\mathcal{C}_1\cap\mathcal{C}_4=\mathcal{C}_2\cap\mathcal{C}_3=\emptyset$). By Corollary \ref{cor16}, there are essentially $8$ candidates for the shortest curve which are of type CSC (see Figure \ref{fig:candidates}). For each of $4$ cases below, there are two different choices for the geodesic segment:
\begin{enumerate}
\item Concatenation of an arc of circle $\mathcal{C}_1$, a geodesic segment and an arc of circle  $\mathcal{C}_3$.
\item Concatenation of an arc of circle $\mathcal{C}_1$, a geodesic segment and an arc of circle  $\mathcal{C}_4$.
\item Concatenation of an arc of circle $\mathcal{C}_2$, a geodesic segment and an arc of circle  $\mathcal{C}_3$.
\item Concatenation of an arc of circle $\mathcal{C}_2$, a geodesic segment and an arc of circle  $\mathcal{C}_4$.
\end{enumerate}

To calculate the length of these candidates we need some elementary formulas from spherical trigonometry. Denote a CSC curve by $\gamma$, we denote the angle of the first arc by $\alpha$, the second arc by $\theta$ and the third arc by $\beta$. Then:
\begin{equation}\label{eqlength1}
\length(\gamma) = \theta + (\alpha+\beta)\sin\rho_0 .
\end{equation}
We shall explicit the relation between the numbers $\theta$, $\alpha$ and $\beta$, and the final frame $\mQ$ which appears in the definition of $\cLspace$. We start with Case 1 (Case 4 is analogous).

\begin{figure}[htb]
\centering
\begin{tikzpicture}
\begin{scope}[xshift = -6cm,decoration={
    markings,
    mark=at position 0.55 with {\arrow{>}}}]
\draw[thick,postaction={decorate}] (0,0.5) ++(190:0.5) arc (190:270:0.5) -- (2,0) arc (270:320:0.5);
\draw[dashed] (0,.5) node {$\mathcal{C}_1$} circle (.5) ;
\draw[dashed] (2,.5) node {$\mathcal{C}_3$} circle (.5) ;
\end{scope}
\begin{scope}[xshift = -2cm,decoration={
    markings,
    mark=at position 0.55 with {\arrow{>}}}]
\draw[thick,postaction={decorate}] (0,0.5) ++(190:0.5) arc (190:270:0.5) -- (2,0) arc (90:20:0.5);
\draw[dashed] (0,.5) node {$\mathcal{C}_1$} circle (.5) ;
\draw[dashed] (2,-.5) node {$\mathcal{C}_4$} circle (.5) ;
\end{scope}
\begin{scope}[xshift = +2cm,decoration={
    markings,
    mark=at position 0.55 with {\arrow{>}}}]
\draw[thick,postaction={decorate}] (0,-.5) ++(170:0.5) arc (170:90:0.5) -- (2,0) arc (270:320:0.5);
\draw[dashed] (0,-.5) node {$\mathcal{C}_2$} circle (.5) ;
\draw[dashed] (2,.5) node {$\mathcal{C}_3$} circle (.5) ;
\end{scope}
\begin{scope}[xshift = +6cm,decoration={
    markings,
    mark=at position 0.55 with {\arrow{>}}}]
\draw[thick,postaction={decorate}] (0,-.5) ++(170:0.5) arc (170:90:0.5) -- (2,0) arc (90:20:0.5);
\draw[dashed] (0,-.5) node {$\mathcal{C}_2$} circle (.5) ;
\draw[dashed] (2,-.5) node {$\mathcal{C}_4$} circle (.5) ;
\end{scope}
\end{tikzpicture}
\caption{These are example of curves for Cases 1-4 respectively, from left to right.}
\label{fig:candidates}
\end{figure}
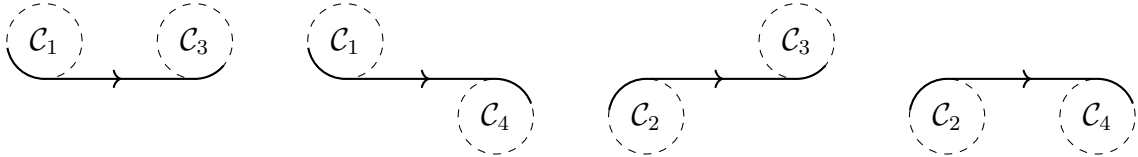

Set $a$ as the endpoint of the first arc of $\gamma$, $b$ as the start point of the last arc of $\gamma$. Draw two great circles. First one starts from $a$ and passes through $p_1$ by the shortest arc. Second one starts from $b$ and passes through $q_1$ by the shortest arc. These two great circles meet by first the time at a point which we shall call it $c$ (see Figure \ref{fig:theta}).

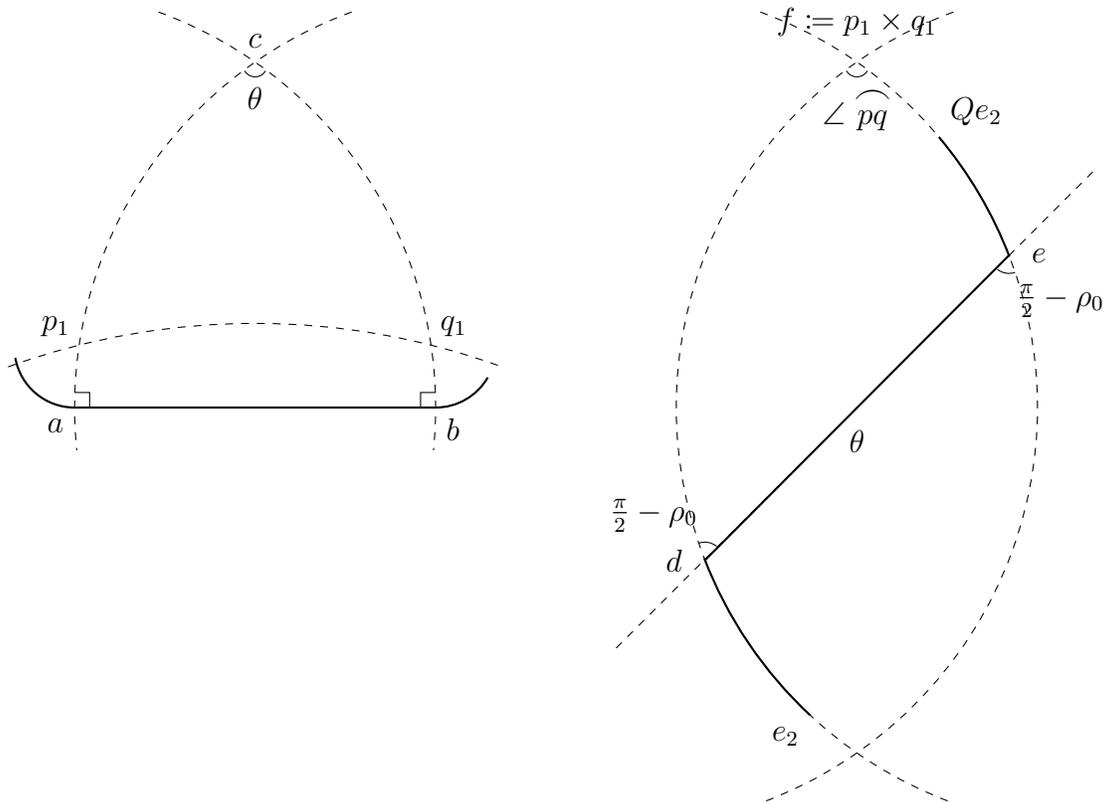
\begin{figure}[htb]
\centering
\begin{tikzpicture}
\begin{scope}[xshift = -4cm, scale=.8]
\clip(-2,-.7) rectangle (8,7);
\draw[thick] (0,1) ++(190:1) arc (190:270:1) -- (6,0) arc (270:330:1);
\draw[dashed] (7,0) ++(110:7) arc (110:187:7);
\draw[dashed] (-1,0) ++(70:7) arc (70:-7:7);
\draw[dashed] (3,-10.6) ++(110:12) arc (110:70:12);
\draw (0,.25) -- (.25,.25) -- (.25,0);
\draw (6,.25) -- (5.75,.25) -- (5.75,0);
\node[anchor=north east] (a) at (0,0) {$a$};
\node[anchor=north west] (b) at (6,0) {$b$};
\node[anchor=south] (c) at (3,5.8) {$c$};
\node[anchor=south east] (p1) at (.1,1.0) {$p_1$};
\node[anchor=south west] (q1) at (5.9,1.0) {$q_1$};
\draw (3,5.7) ++(210:.2) arc (210:330:.2) node[pos=.5,anchor=north]{$\theta$};
\end{scope}
\begin{scope}[xshift = +4cm, scale=.8]
\clip(-2,-7) rectangle (8,7);
\draw[dashed] (7,0) ++(110:7) arc (110:250:7);
\draw[thick] (7,0) ++(227:7) arc (227:201:7) node[pos=0,anchor=north east]{$e_2$} node[pos=.99,label=left:$d$] (a) {};
\draw[dashed] (-1,0) ++(70:7) arc (70:-70:7);
\draw[thick] (-1,0) ++(40:7) arc (40:21:7) node[pos=0,anchor=south west]{$Qe_2$} node[pos=.99,label=right:$e$] (b) {};
\draw[dashed] (-1,-4) -- (7,4);
\node[anchor=south] (pq) at (3,6.0) {$f\coloneqq p_1\times q_1$};
\draw[thick] (a.center) -- (b.center) node[pos=.5,label=below:$\theta$]{};
\draw (3,5.7) ++(210:.2) arc (210:330:.2) node[pos=.5,anchor=north]{$\angle \arc{pq}$};
\draw (a) ++(50:.3) arc (50:110:.3) node[pos=.5,anchor=south east]{$\frac{\pi}{2}-\rho_0$};
\draw (b) ++(230:.3) arc (230:290:.3) node[pos=.5,anchor=north west]{$\frac{\pi}{2}-\rho_0$};
\end{scope}
\end{tikzpicture}
\caption{These are illustrations for Case 1. In the illustration, the circle containing arc $\arc{ac}$ and the circle containing arc $\arc{bc}$ represent geodesics on sphere. On the right illustration, it shows tangent vectors of $\mathbb{S}^2$ translated so that its base point is at the origin. The thick curve is trajectory of $\gamma'$. All curves in the image are segments of great circles on sphere.}
\label{fig:theta}
\end{figure}

\begin{figure}[htb]
\centering
\begin{tikzpicture}
 \node[anchor=south west,inner sep=0] at (0,0) {\includegraphics[width=8cm]{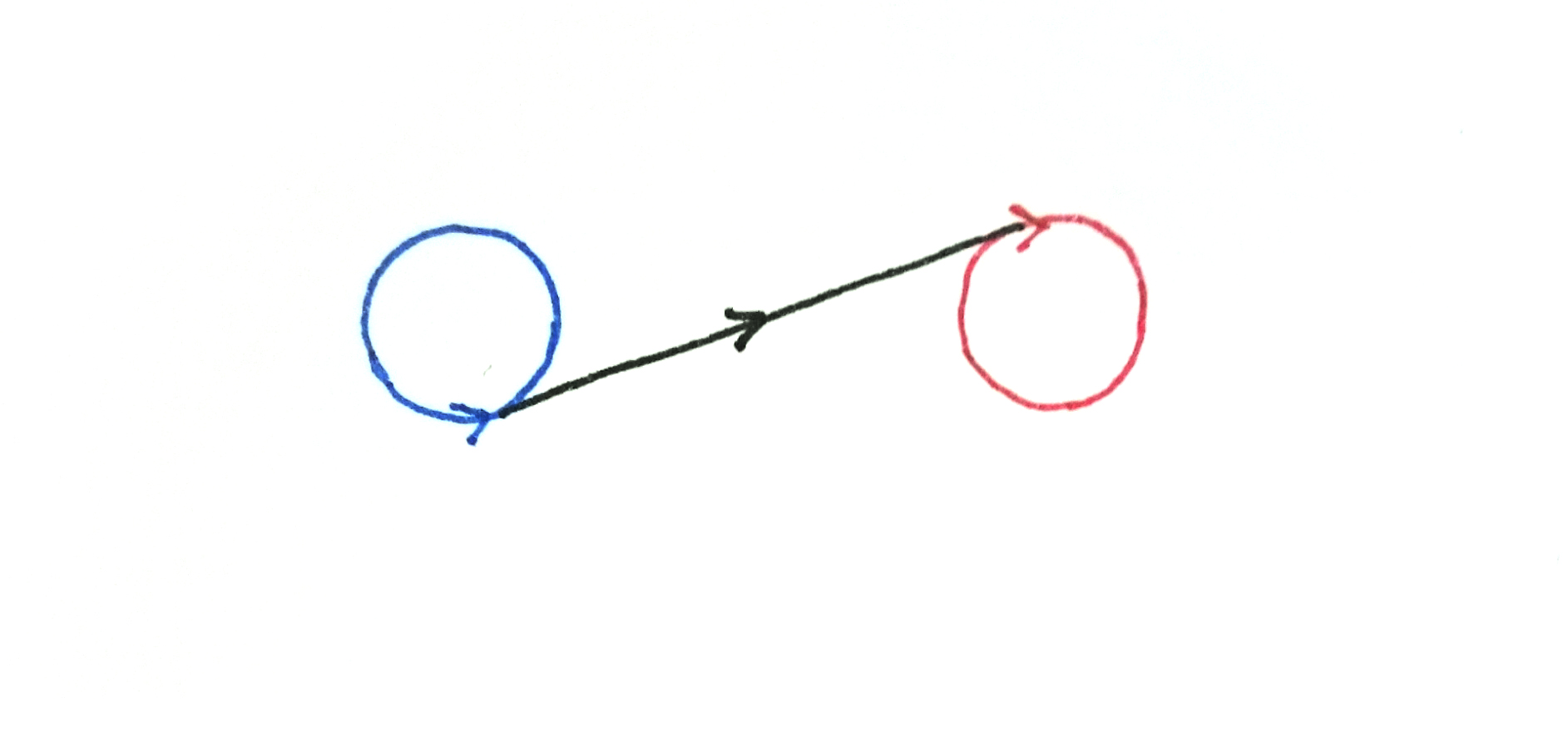}};
  \node[anchor=south west,inner sep=0] at (0,-2.5) {\includegraphics[width=8cm]{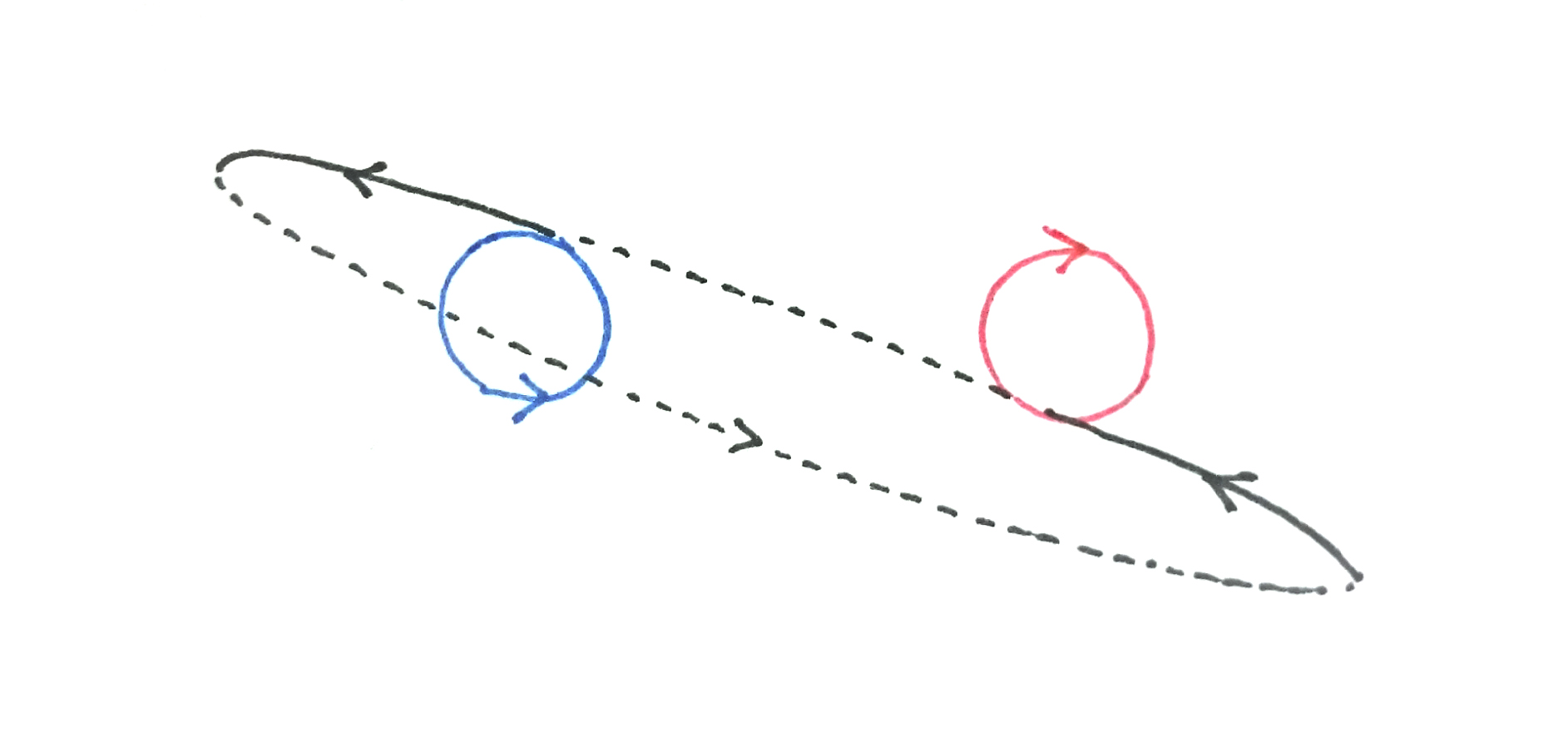}};
   \node[anchor=south west,inner sep=0] at (8.5,0) {\includegraphics[width=8cm]{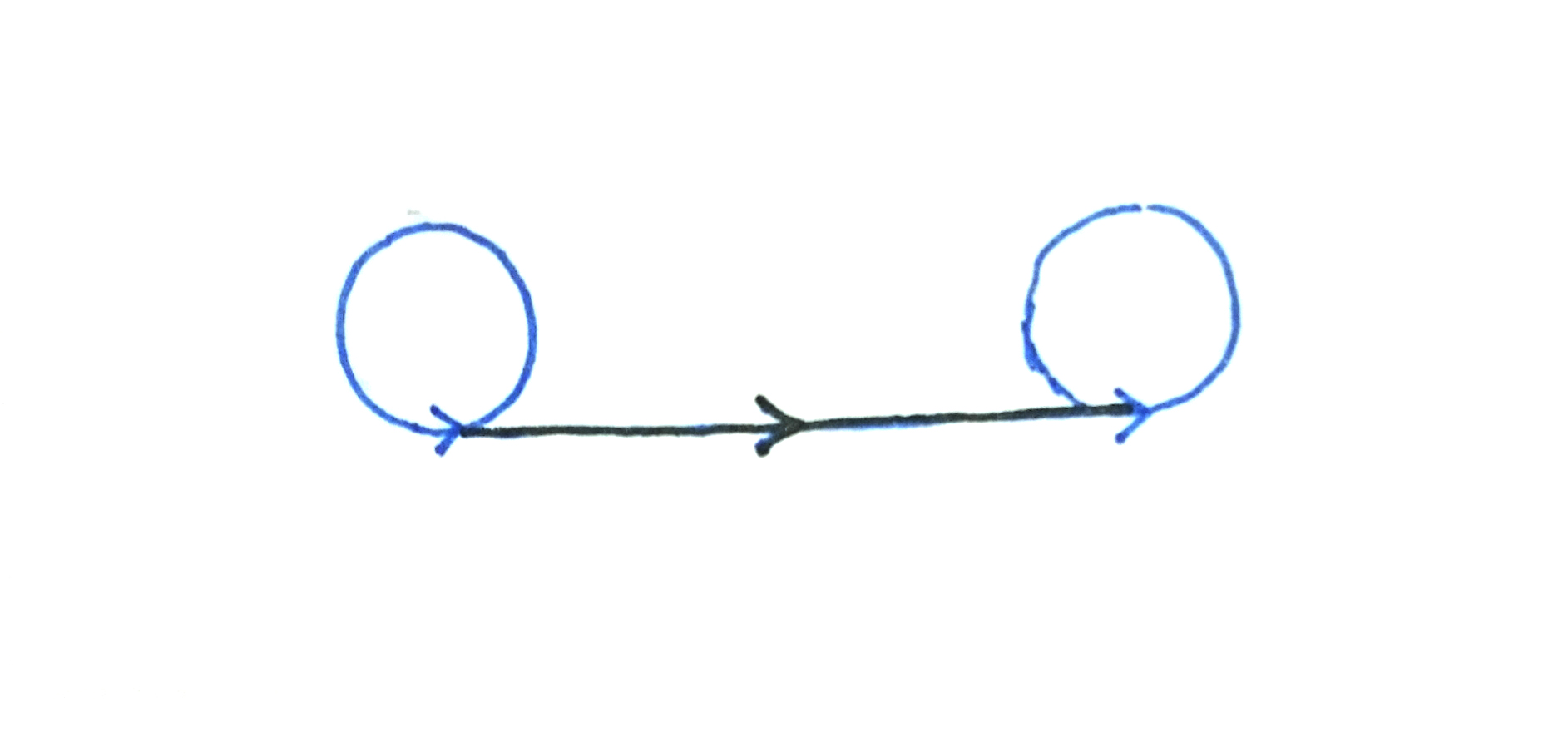}};
  \node[anchor=south west,inner sep=0] at (8.5,-2.5) {\includegraphics[width=8cm]{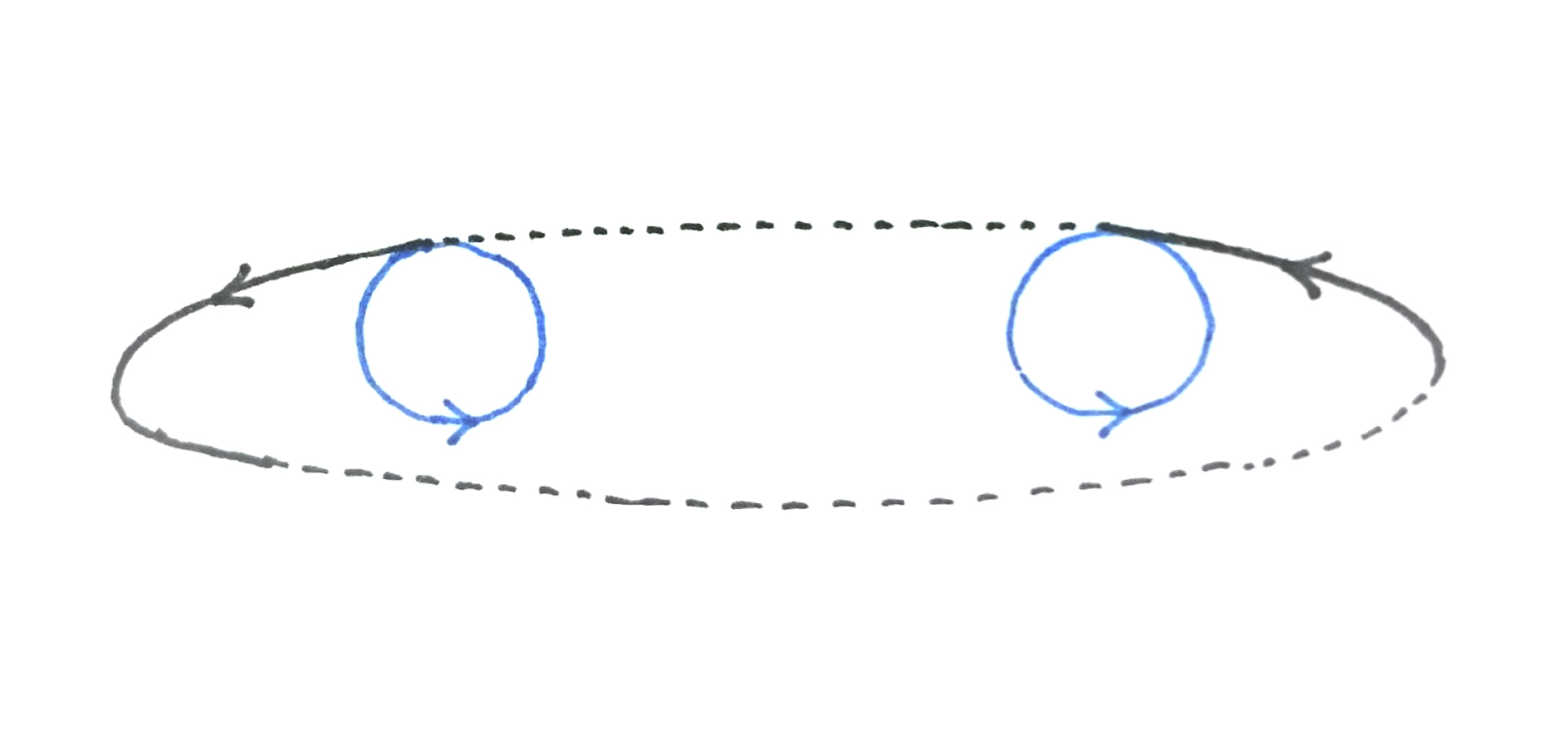}};
   \node[anchor=south west,inner sep=0] at (0,-10.5) {\includegraphics[width=8cm]{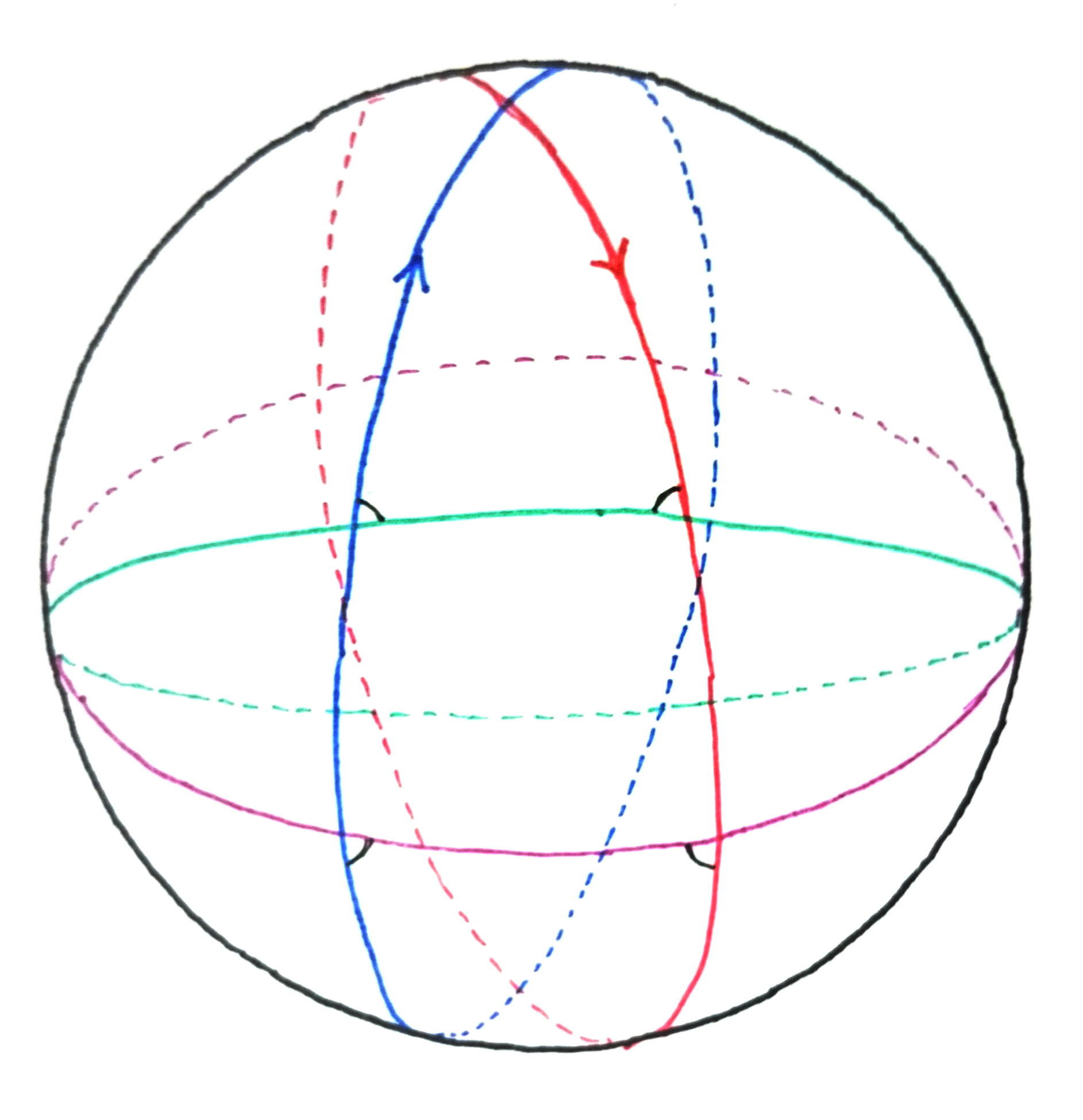}};
  \node[anchor=south west,inner sep=0] at (8.5,-10.5) {\includegraphics[width=8cm]{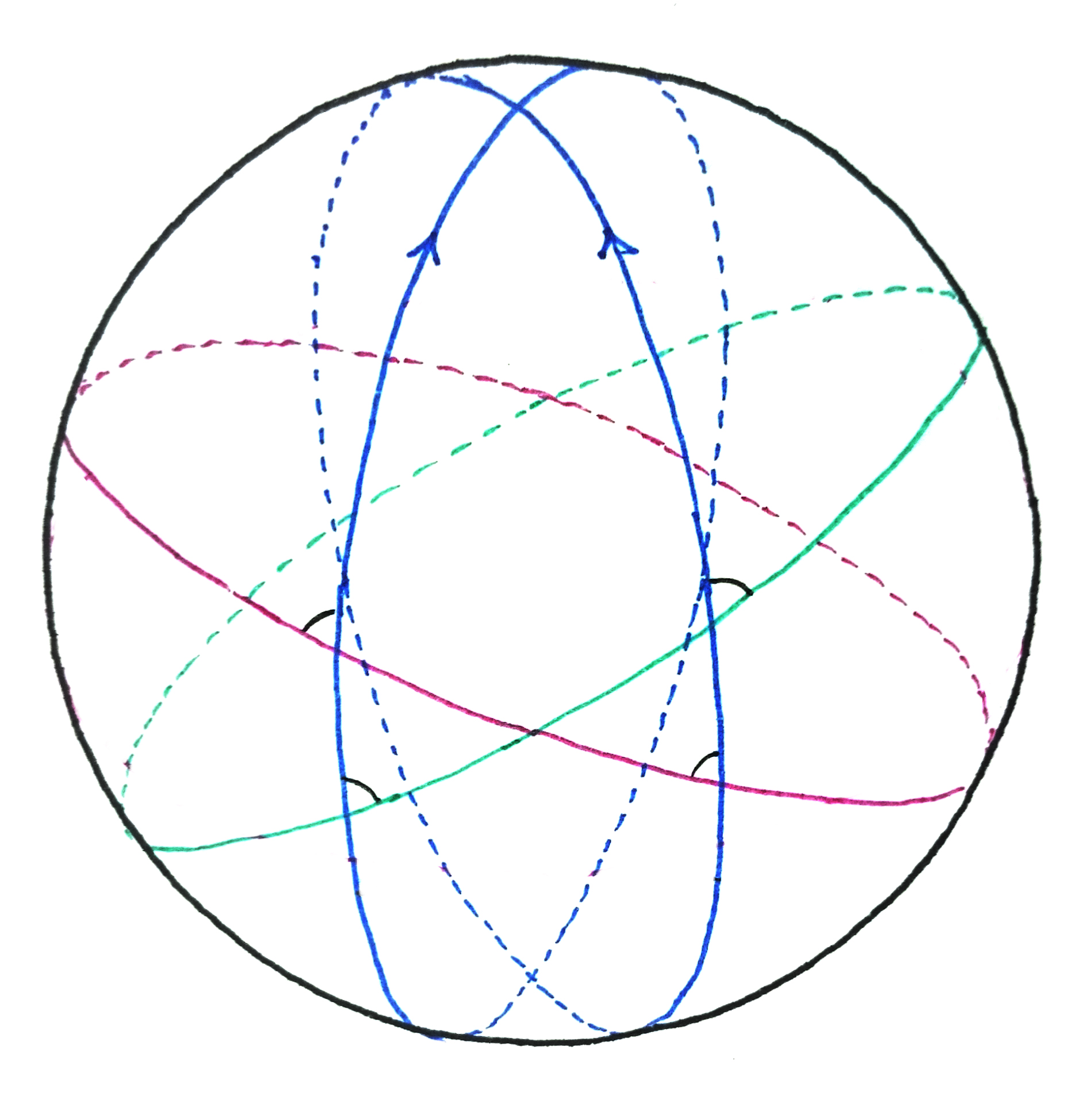}};
  \node[anchor=south west,inner sep=0] at (0,3) {Configuration 1};
  \node[anchor=south west,inner sep=0] at (0,0.5) {Configuration 2};
  \node[anchor=south west,inner sep=0] at (8.5,3) {Configuration 3};
  \node[anchor=south west,inner sep=0] at (8.5,0.5) {Configuration 4};
\end{tikzpicture}
\caption{This figure illustrates ``railroads'' of trajectories of tangent vector for the all four configurations above. All curves in two images above are great circles on sphere and marked angles measures exactly $\frac{\pi}{2}-\rho_0$.}
\label{fig:curves}
\end{figure}

Note that on the triangle $\bigtriangleup abc$ we have the angle $\angle cab = \frac{\pi}{2}$ and $\overline{bc}=\frac{\pi}{2}$, by Sine rule, we get $\angle bca = \overline{ab}=\theta$. Also, observe that $\overline{cp_1} = \overline{ca}-\overline{p_1a} = \frac{\pi}{2} - \rho_0$, $\overline{cq_1} = \overline{cb}-\overline{q_1b} = \frac{\pi}{2} - \rho_0$. Now applying Cosine rule on triangle $\bigtriangleup p_1q_1c$, and considering previous relations, we deduce the following equation for $\theta$: 
\begin{equation}\label{eqlength21}
\cos\theta = \frac{\langle p_1,q_1 \rangle}{\cos^2\rho_0}-\tan^2\rho_0 .
\end{equation}

Next, we need to write $(\alpha + \beta)$ in terms of known parameters. For this we look at the $\gamma'$ translated into $\mathbb{S}^2$ as shown on the right-hand side in Figure \ref{fig:theta}. First we observe since $\gamma$ is concatenation of three arcs of circles, its derivative $\gamma'$ may be split into three geodesic segments. With the first segment lies in the great circle perpendicular to $p_1$ and the last segment lies in the great circle perpendicular to $q_1$. Obviously, these two great circles intersect each other at points $\pm p_1\times q_1$ and the angle between them is $\angle \arc{p_1q_2}$. We denote $f\coloneqq p_1\times q_1$. On the other hand, the middle segment $\overline{de}$ has length $\theta$. Considering the variation on radius of curvature of $\gamma$, we deduce that $\angle fde = \frac{\pi}{2}-\rho_0$ and $\angle def = \frac{\pi}{2}+\rho_0$. Denote by $L$ the length of the shorter arc from $f$ to $e$. By applying Sine rule on triangle $\bigtriangleup fde$, we obtain:
\begin{equation}\label{eqlength31}
\sin L = \frac{\cos\rho_0\cdot\sin\theta}{\sin(\angle \arc{p_1q_1})} .
\end{equation}

We also have relations:
\begin{equation}\label{eqlength41}
\alpha = L - \angle\arc{e_2(-f)} \quad \text{and} \quad \beta = L - \angle\arc{fQe_2} \quad \text{where $f=p_1\times q_1$.}
\end{equation}

Since the values of $\angle\arc{e_2(-f)}$ and $\angle\arc{fQe_2}$ may be calculated directly in terms of $\mQ$, so Equations (\ref{eqlength1}), (\ref{eqlength21}), (\ref{eqlength31}) and (\ref{eqlength41}) are complete formulas that give the length of $\gamma$ for Case 1. Although they are technically calculable, these formulas involve taking several times inverse of trigonometric functions. It is not clear for the author whether these formulas may be simplified into shorter expressions.

Cases 2 and 3 are also analogous. Here we present the demonstration for Case 2. The procedure is similar to Case 1. Set $a$ as the endpoint of the first arc of $\gamma$, $b$ as the start point of the last arc of $\gamma$. Draw two great circles. First one starts from $a$, and passes through $p_1$ by the shortest arc. Second one, in contrast with the previous case, starts from $q_1$, and passes through $b$ by the shortest arc. These two great circles meet by the first time at a point which we call $c$ (see Figure \ref{fig:theta2}).

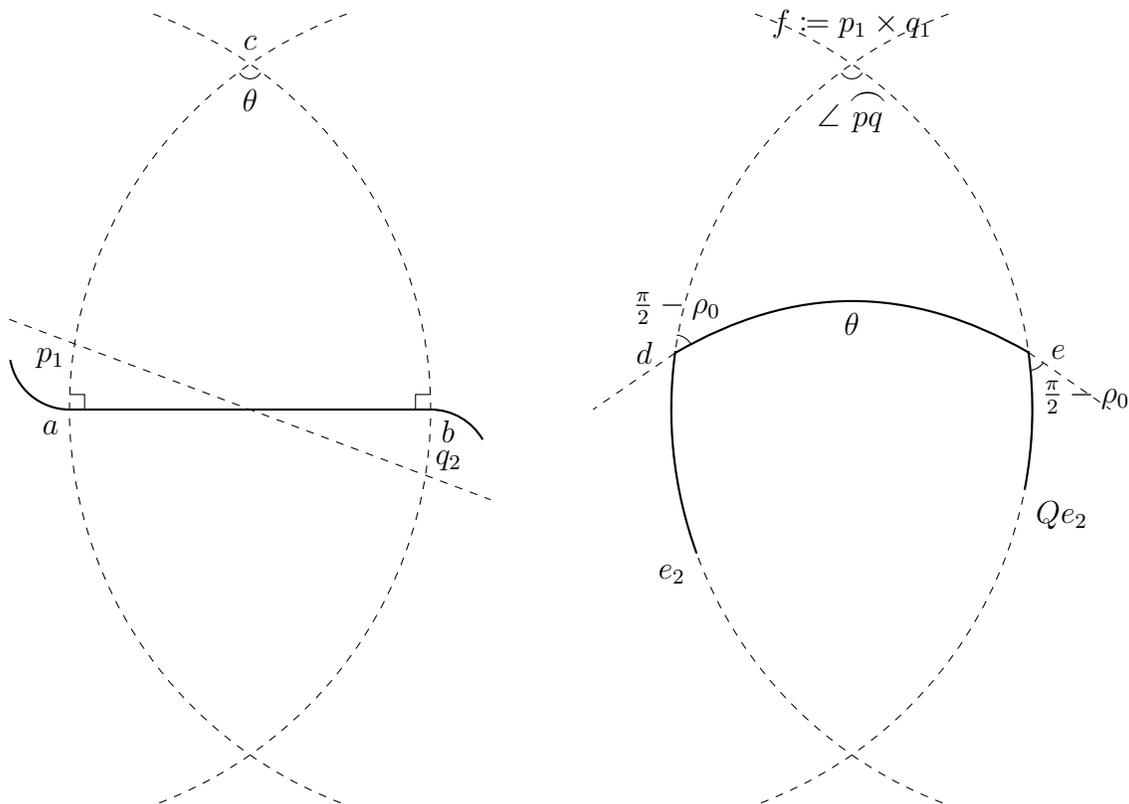
\begin{figure}[htb]
\centering
\begin{tikzpicture}
\begin{scope}[xshift = -4cm, scale=.8]
\clip(-1.3,-7) rectangle (8,7);
\draw[thick] (0,1) ++(190:1) arc (190:270:1) -- (6,0) arc (90:30:1);
\draw[dashed] (7,0) ++(110:7) arc (110:250:7);
\draw[dashed] (-1,0) ++(70:7) arc (70:-70:7);
\draw[dashed] (-1,1.5) -- (7,-1.5);
\draw (0,.25) -- (.25,.25) -- (.25,0);
\draw (6,.25) -- (5.75,.25) -- (5.75,0);
\node[anchor=north east] (a) at (0,0) {$a$};
\node[anchor=north west] (b) at (6,0) {$b$};
\node[anchor=south] (c) at (3,5.8) {$c$};
\node[anchor=north east] (p1) at (.1,1.2) {$p_1$};
\node[anchor=south west] (q2) at (5.9,-1.2) {$q_2$};
\draw (3,5.7) ++(210:.2) arc (210:330:.2) node[pos=.5,anchor=north]{$\theta$};
\end{scope}
\begin{scope}[xshift = +4cm, scale=.8]
\clip(-2,-7) rectangle (8,7);
\draw[dashed] (7,0) ++(110:7) arc (110:250:7);
\draw[thick] (7,0) ++(200:7) arc (200:172:7) node[pos=0,anchor=north east]{$e_2$} node[pos=.99,label=left:$d$] (d) {};
\draw[dashed] (-1,0) ++(70:7) arc (70:-70:7);
\draw[thick] (-1,0) ++(-11:7) arc (-11:8:7) node[pos=0,anchor=north west]{$Qe_2$} node[pos=.99,label=right:$e$] (e) {};
\draw[thick](d.center) to[bend left] node[pos=.5,anchor=north]{$\theta$} (e.center) ;
\draw [dashed] (d.center) -- (-1.3,0);
\draw [dashed] (e.center) -- (7.3,0);
\node[anchor=south] (pq) at (3,6.0) {$f\coloneqq p_1\times q_1$};
\draw[thick] (a.center) -- (b.center) node[pos=.5,label=below:$\theta$]{};
\draw (3,5.7) ++(210:.2) arc (210:330:.2) node[pos=.5,anchor=north]{$\angle \arc{pq}$};
\draw (d) ++(30:.3) arc (30:85:.3) node[pos=1,anchor=south]{$\frac{\pi}{2}-\rho_0$};
\draw (e) ++(280:.3) arc (280:325:.3) node[pos=0,anchor=north west]{$\frac{\pi}{2}-\rho_0$};
\end{scope}
\end{tikzpicture}
\caption{These are illustrations for Case 2. In the illustration, the circle containing arc $\arc{ac}$ and the circle containing arc $\arc{bc}$ represent geodesics on sphere. On the right illustration, it shows tangent vectors of $\mathbb{S}^2$ translated so that its base point is at the origin. The thick curve is trajectory of $\gamma'$. All curves in the image are segments of great circles on sphere.}
\label{fig:theta2}
\end{figure}

Note that on the triangle $\bigtriangleup abc$, $\angle cab = \frac{\pi}{2}$ and $\overline{bc}=\frac{\pi}{2}$, and by applying the sine rule, we obtain $\angle bca = \overline{ab}=\theta$. Also, observe that $\overline{cp_1} = \overline{ca}-\overline{p_1a} = \frac{\pi}{2} - \rho_0$, $\overline{cq_2} = \overline{cb}+\overline{q_2b} = \frac{\pi}{2} + \rho_0$. Now applying the cosine rule on the triangle $\bigtriangleup p_1q_2c$, and considering the previous relations, we deduce the following equation for $\theta$: 
\begin{equation}\label{eqlength22}
 \cos\theta = \frac{\langle p_1,q_2 \rangle}{\cos^2\rho_0}-\tan^2\rho_0 .
\end{equation}

Now we proceed to write $(\alpha + \beta)$ in terms of the known parameters. For this we look at the $\gamma'$ translated into $\mathbb{S}^2$ as shown in Figure \ref{fig:theta2}. Again the derivative $\gamma'$ may be split into three geodesic segments: the first segment lies in the great circle perpendicular to $p_1$, the last segment lies in the great circle perpendicular to $q_2$ and the middle segment has length $\theta$. Two great circles intersect each other at points $\pm p_1\times q_2$ and the angle between them is $\angle \arc{p_1q_2}$. We denote $f\coloneqq p_1\times q_1$. Considering the variation on radius of curvature of $\gamma$, we deduce that $\angle fde = \angle def = \frac{\pi}{2}-\rho_0$. Let $L$ the length of the short arc from $f$ to $e$. By applying the sine rule on the triangle $\bigtriangleup fde$, we obtain:
\begin{equation}\label{eqlength32}
 \sin L = \frac{\cos\rho_0\cdot\sin\theta}{\sin(\angle \arc{p_1q_2})} . 
\end{equation}

We also have relations:
\begin{equation}\label{eqlength42}
 \alpha = \angle\arc{fe_2} - L \quad \text{and} \quad \beta = \angle\arc{fQe_2} - L \quad \text{where $f=p_1\times q_1$.}.
\end{equation}

Since the values of $\angle\arc{fe_2}$ and $\angle\arc{fQe_2}$ may be calculated directly in terms of $\mQ$, these are formulas to obtain the length of $\gamma$ for Case 2.

Given a $\mQ\in\SO$ such that $\mathcal{C}_1\cap\mathcal{C}_4$ and $\mathcal{C}_1\cap\mathcal{C}_4$ consist of, at most, one point (i.e. circles are tangent). Formulas (\ref{eqlength1}), (\ref{eqlength21}), (\ref{eqlength31}), (\ref{eqlength41}), (\ref{eqlength22}), (\ref{eqlength32}) and (\ref{eqlength42}) permit us to obtain exact length of all possible CSC curves thus to determine which curve is length-minimizing. It is unclear for the author if those formulas may be simplified.

 \else
 \fi


\subsection{Curve shortening}\label{Shortening}

We will use the following parallel-meridian coordinates on sphere. Let $v\in\mathbb{S}^2$, each vector $u\in\mathbb{S}^2$ may be written as $(\theta(u),\varphi(u))\in [0,\pi]\times [-\pi,+\pi)$ with $\theta(u)=d(u,v)$. These values are unique if, and only if, $u\neq \{-v,v\}$. We often refer the coordinate $\theta(u)$ as $v$\emph{-parallel} coordinate of the vector $u$.

Let us $v\in\mathbb{S}^2$, we define a vector field in $\mathbb{S}^2\smallsetminus\{v,-v\}$, given by $W_v(w)\coloneqq v \times P(w)$ where $\times$ is the usual cross product of vectors in $\mathbb{R}^3$ and $\mP$ is the normalized projection of $w$ onto plane perpendicular to $v$.

This subsection is for the characterization of length-minimizing curves. We use an idea similar and inspired by Birkhoff curve shortening (see \cite{birkh}). Similar ideas and related studies by others may be found in \cite{monroy}, \cite{ayala}, \cite{ayala2}, \cite{ayala3}, \cite{ayala4} and \cite{ayala5}. Given a curve $\gamma\in\cLspace$, let $L_0$ be the length of $\gamma$, we construct a new curve that is shorter than or has the same length as the original curve, by the following process: we separate the curve in small sections, each section, except the first and last (that have length $\leq l$), have the fixed length $l$, with $l\leq\pi\sin\rho_0$. Then we replace each section for another segment that minimizes the length with the same starting and ending Frenet frames as before. The resulting curve will be shorter or has the same length as $\gamma$. Moreover, both curves are homotopical to each other.

It is enough to prove that each section is homotopic:
\begin{lemma}\label{sectionhomotopy} Let $\mP,\mQ\in\SO$ and let $\alpha\in\bar{\mathcal{L}}_{-\kappa_0}^{+\kappa_0}(\mP,\mQ)$ be a curve of length $l$, with $l\leq\pi\sin\rho_0$, and suppose also that $\rho_0 \leq \frac{\pi}{4}$. Let $\alpha_0$ be the shortest curve in $\bar{\mathcal{L}}_{-\kappa_0}^{+\kappa_0}(\mP,\mQ)$. Then $\alpha_0$ and $\alpha$ are homotopic within $\bar{\mathcal{L}}_{-\kappa_0}^{+\kappa_0}(\mP,\mQ)$.
\end{lemma}

\begin{proof} Applying the transformation $\mP^T$ to the curve $\alpha$, we may assume without loss of generality that $\mP=\mI$. Since $\length (\alpha) \leq \pi\sin\rho_0$, $\alpha$ lies inside a ball of radius $\pi\sin\rho_0$ centered at $\alpha(0)=(1,0,0)$. First, we show that there exists a vector $v \in \mathbb{S}^2 \cap (1,0,0)^{\bot} $ such that $ \langle\alpha'(t), v \rangle \geq 0 $ for all $t\in [0,1]$. 

Suppose there is no such $v$. Then there exists a $u\in\mathbb{S}^2\cap (1,0,0)^\bot$ such that we can find $t_1,t_2\in [0,1]$, $t_1<t_2$ satisfying following 3 properties:
\begin{enumerate}
\item $\langle\alpha'(t_i),u \rangle = 0$, for $i=1,2$.
\item $\langle \alpha'(t_1)\times(1,0,0),u\rangle$, $\langle \alpha'(t_2)\times(1,0,0),u\rangle$ have opposite signs.
\item $\langle \alpha'(t),u \rangle >0$ for all $t\in [t_1,t_2]$.
\end{enumerate}
Now we check that the length of segment $\alpha([t_1,t_2])$ is greater or equal to $\pi\sin\rho_0$.
We suppose without loss of generality that $\langle \alpha(t_1),u \rangle < \langle \alpha(t_2),u \rangle$ and $\langle\alpha(t_1)\times (1,0,0),u\rangle > 0$. We consider the circles $C_1$ and $C_2$ of radius $\rho_0$ tangent to $\alpha$ at points $\alpha(t_1)$ and $\alpha(t_2)$ respectively, and that these circles lie on the right side of the curve. The curve $\alpha$ cannot cross neither of two circles $C_1$ and $C_2$. Properties 1-3 imply that the length of $\alpha([t_1,t_2])$ must be greater or equal to half turn of either $C_1$ or $C_2$. This implies that its length must be greater than or equal to $\pi\sin\rho_0$. Thus $t_1=0$, $t_2=1$. So $\alpha$ is an arc of a circle of radius $\rho_0$ which clearly has the direction $(0,1,0)$ satisfying the desired $\langle \alpha'(t),(0,1,0)\rangle\geq 0$.

Since there exists a $v\in\mathbb{S}^2\cap (1,0,0)^\bot$ for all $t\in [0,1]$, we parametrize both $\alpha$ and $\alpha_0$ by polar coordinates with $v$ as axis:
\begin{align*}
\alpha(t) & = \big(\theta(t),\varphi_\alpha(t)\big)\\
\alpha_0(t) & = \big(\theta(t),\varphi_{\alpha_0}(t)\big)
\end{align*}
We define the homotopy from $\alpha_0$ to $\alpha$ as $\alpha_s(t)=\big(\theta(t),s\varphi_\alpha(t)+(1-s)\varphi_{\alpha_0}(t)\big)$. It is easy to check that $\alpha_s\in\mathcal{L}_{\rho_0}(\mI,\mQ)$ for all $s\in [0,1]$.
\end{proof}

Now given a curve $\gamma$, we will construct a sequence of curves $(\eta_n)_{n\in\mathbb{N}}$ by the following the method below. First we consider a sequence of numbers $(l_k)_{k\in\mathbb{N}}$ that is dense in the interval $[0,1]$ and whose set of accumulation points is the entire $[0,1]$:

\begin{align*}
& m_{1,1} = \frac{1}{2} & & m_{1,2} = 1 & && && & && \\
& m_{2,1} = \frac{1}{4} & & m_{2,2} = \frac{2}{4} && m_{2,3} = \frac{3}{4} && m_{2,4} = 1 && && \\
& m_{3,1} = \frac{1}{8} & & m_{3,2} = \frac{2}{8} && m_{3,3} = \frac{3}{8} && m_{3,4} = \frac{4}{8} && \hdots && m_{3,8} = 1 & \\
 & \quad\quad\vdots && \quad\quad\vdots && \quad\quad\vdots && \quad\quad\vdots &&  && \quad\quad\vdots & \\
& m_{k,1} = \frac{1}{2^k} & &  m_{k,2} = \frac{2}{2^k} && m_{k,3} = \frac{3}{2^k} && m_{k,4} = \frac{4}{2^k} && \hdots && m_{k,2^k} = 1 & \\
 & \quad\quad\vdots & & \quad\quad\vdots && \quad\quad\vdots && \quad\quad\vdots &&  && \quad\quad\vdots &
\end{align*}

Define the sequence $(l_k)_{k\in\mathbb{N}}$ by setting $l_1 = m_{1,1}$, $l_2 = m_{1,2}$, $l_3 = m_{2,1}$, $l_4 = m_{2,2}$, $l_5 = m_{2,3}$, $l_6 = m_{2,4}$, $l_7 = m_{3,1}$, $l_8=m_{3,2}$, so on. Set $\eta_0\coloneqq\gamma$, for each $n\geq 0$ we define each curve $\eta_{n+1}$ as the curve $\eta_n$ separated into sections of length $l_{n+1}\pi\sin\rho_0$, $\pi\sin\rho_0$, $\pi\sin\rho_0$, \dots, $\pi\sin\rho_0$, $m_n \leq \pi\sin\rho_0$, respectively, where $m_n$ is the remaining length at the end of the curve $\eta_n$. Then we replace each section by a segment that minimizes the length. By Lemma \ref{sectionhomotopy}, since each small segment is homotopical to its replacement within $\bar{\mathcal{L}}_{-\kappa_0}^{+\kappa_0}$, we conclude that $\eta_{n+1}$ and $\eta_n$ are homotopical in $\bar{\mathcal{L}}_{-\kappa_0}^{+\kappa_0}(\mI,\mQ)$. 

Since $\cLspace$ is closed, each $\eta_n$ has limited curvature and the length of $\eta_n$ is non-increasing. Thus the sequence $(\eta_n)_{n\in\mathbb{N}}$ has a convergent subsequence, denote by $\tilde{\gamma}\in\cLspace$ the limit of this subsequence. The following proposition is an immediate consequence of the construction.

\begin{proposition} Let $\tilde{\gamma}$ be a limit obtained by the shortening process above. Then $\tilde{\gamma}$ consists of concatenation of the following segments:
\begin{itemize}
\item The first segment has curvature $\pm \kappa_0$.
\item The segments in the middle are either: geodesics or arcs of circle with curvature $\pm\kappa_0$ length $\geq\pi\sin\rho_0$.
\item The last segment has curvature $\pm \kappa_0$.
\end{itemize}
\end{proposition}

\subsection{Remarks on the definition and the topology of $\Lspace$}\label{appdeftop}

As seen in the beginning of the Section \ref{sec2}, recall that the definition of $\Lspace$ differs from the definition in \cite{salzuh}. In this subsection we show that the space $\Lspace$ is a subset of the former space that is weakly homotopically equivalent to it. We will give, in this subsection, another equivalent definition of the space $\Lspace$. For sake of completeness, we introduce the definition given in \cite{salzuh} here. 

\begin{definition} A function $f:[a,b]\to\real$ is said to be of class $H^1$ if it is an indefinite integral of some $g\in L^2[a,b]$. We extend this definition to maps $F:[a,b]\to\real^n$ by saying that $F$ is of class $H^1$ if and only if each of its component functions is of class $H^1$.
\end{definition}

Since $L^2[a,b]\subset L^1[a,b]$, an $H^1$ function is absolutely continuous and differentiable almost everywhere. Given a  $\mat{P}\in\SO$ and a $L^2[0,1]$ map $\Lambda:[0,1]\to \mathfrak{so}_3(\mathbb{R})$ of the form:
\begin{equation*}
\Lambda(t)=\left(\begin{array}{ccc} 0 & -v(t) & 0 \\
v(t) & 0 & -w(t) \\
0 & w(t) & 0 
\end{array}\right),
\end{equation*}

\noindent let $\Phi: [0,1]\to \SO$ be the unique solution to the initial value problem:
\begin{equation}\label{ivp}
\Phi^\prime(t) = \Phi(t)\Lambda(t),\quad \Phi(0)= \mat{P} .
\end{equation}
Define $\gamma:[0,1]\to\mathbb{S}^2$ to be the smooth curve given by $\gamma(t)=\Phi(t)e_1$. Then $\gamma$ is regular if and only if $v(t)\neq 0$ for all $t\in [0,1]$, and it satisfies $\Phi_\gamma = \Phi$ if and only if $v(t)>0$ for all $t$. Let $\mathbf{E}=L^2[0,1]\times L^2[0,1]$ and let $h:(0,+\infty)\to\real $ be the smooth diffeomorphism
$$ h(t) = t - t^{-1} .$$
For each pair $\kappa_1<\kappa_2\in\real$, let $h_{\kappa_1,\kappa_2}:(\kappa_1,\kappa_2)\to\real$ be the smooth diffeomorphism
$$h_{\kappa_1,\kappa_2}(t) = (\kappa_1-t)^{-1} + (\kappa_2 - t)^{-1} $$
and, similarly, set
\begin{align*}
h_{-\infty,+\infty}:\real\to\real & \quad h_{-\infty,+\infty}(t)=t \\
h_{-\infty,\kappa_2}:(-\infty,\kappa_2)\to\real & \quad h_{-\infty,\kappa_2}(t)=t+(\kappa_2-t)^{-1} \\
h_{\kappa_1,+\infty}:(\kappa_1,+\infty)\to\real & \quad h_{\kappa_1,+\infty}(t)=t+(\kappa_1-t)^{-1} .
\end{align*}

\begin{definition} Let $\kappa_1,\kappa_2$ satisfy $-\infty\leq \kappa_1<\kappa_2\leq +\infty$. A curve $\gamma:[0,1]\to\mathbb{S}^2$ will be called \emph{$(\kappa_1,\kappa_2)$-admissible} if there exist $\mat{P}\in\SO$ and a pair $(\hat{v},\hat{w})\in\mathbf{E}$ such that $\gamma(t)=\Phi(t)e_1$ for all $t\in [0,1]$, where $\Phi$ is the unique solution to the initial value problem \eqref{ivp}, with $v,w$ given by
\begin{equation}
v(t)=h^{-1}(\hat{v}(t)),\quad w(t)=h^{-1}_{\kappa_1,\kappa_2}(\hat{w}(t)).
\end{equation}
When it is not important to keep track of the bounds $\kappa_1,\kappa_2$, we shall simply say that $\gamma$ is \emph{admissible}.
\end{definition}

\begin{definition} Let $-\infty\leq\kappa_1<\kappa_2\leq +\infty$ and $\mat{P},\mat{Q}\in\SO$. Define $\boldsymbol{L}_{\kappa_1}^{\kappa_2}(\mat{P},\mat{Q})$ to be the set of all $(\kappa_1,\kappa_2)$-admissible curves $\gamma$ such that
$$\Phi_\gamma(0) = \mat{P} \quad \text{and} \quad \Phi_\gamma(1)=\mat{Q},$$
where $\Phi_\gamma$ is the frame of $\gamma$. This set is identified with $\mathbf{E}$ via correspondence $\gamma\leftrightarrow (\hat{v},\hat{w})$, and this defines a (trivial) Hilbert manifold structure on $\boldsymbol{L}_{\kappa_1}^{\kappa_2}(\mat{P},\mat{Q})$.
\end{definition}

We consider $\mathbf{E}_{\infty}=L^\infty[0,1]\times L^\infty[0,1]\subset \mathbf{E}$. 

\begin{definition} Let $\kappa_1,\kappa_2$ satisfy $-\infty\leq \kappa_1<\kappa_2\leq +\infty$. A curve $\gamma:[0,1]\to\mathbb{S}^2$ will be called \emph{$(\kappa_1,\kappa_2)$-strongly admissible} if thre exist $\mat{P}\in\SO$ and a pair $(\hat{v},\hat{w})\in\mathbf{E}_{\infty}$ such that $\gamma(t)=\Phi(t)e_1$ for all $t\in [0,1]$, where $\Phi$ is the unique solution to the initial value problem \eqref{ivp}, with $v,w$ given by
\begin{equation}
v(t)=h^{-1}(\hat{v}(t)),\quad w(t)=h^{-1}_{\kappa_1,\kappa_2}(\hat{w}(t)).
\end{equation}
When it is not important to keep track of the bounds $\kappa_1,\kappa_2$, we shall simply say that $\gamma$ is \emph{strongly admissible}.
\end{definition}

\begin{definition} Let $-\infty\leq\kappa_1<\kappa_2\leq +\infty$ and $\mat{P},\mat{Q}\in\SO$. Define $\mathfrak{L}_{\kappa_1}^{\kappa_2}(\mat{P},\mat{Q})$ to be the set of all $(\kappa_1,\kappa_2)$-strongly admissible curves $\gamma$ such that
$$\Phi_\gamma(0) = \mat{P} \quad \text{and} \quad \Phi_\gamma(1)=\mat{Q},$$
where $\Phi_\gamma$ is the frame of $\gamma$. This set is identified with $\mathbf{E}_\infty$ via correspondence $\gamma\leftrightarrow (\hat{v},\hat{w})$, and this defines a (trivial) Banach manifold structure on $\mathfrak{L}_{\kappa_1}^{\kappa_2}(\mat{P},\mat{Q})$.
\end{definition}

It follows from the definitions that $\mathfrak{L}_{\kappa_1}^{\kappa_2}(\mat{P},\mat{Q})\subsetneq\boldsymbol{L}_{\kappa_1}^{\kappa_2}(\mat{P},\mat{Q})$. We shall prove that $\mathfrak{L}_{\kappa_1}^{\kappa_2}(\mat{P},\mat{Q})=\mathcal{L}_{\kappa_1}^{\kappa_2}(\mat{P},\mat{Q})$.

\begin{proposition} Let $-\infty<\kappa_1<\kappa_2< +\infty$ and $\mat{P},\mat{Q}\in\SO$. Then $\mathfrak{L}_{\kappa_1}^{\kappa_2}(\mat{P},\mat{Q})=\mathcal{L}_{\kappa_1}^{\kappa_2}(\mat{P},\mat{Q})$.
\end{proposition}

\begin{proof}[Proof Part 1]
We begin by showing that $\mathfrak{L}_{\kappa_1}^{\kappa_2}(\mat{P},\mat{Q})\subseteq\mathcal{L}_{\kappa_1}^{\kappa_2}(\mat{P},\mat{Q})$. Given a curve $\gamma\in\mathfrak{L}_{\kappa_1}^{\kappa_2}(\mat{P},\mat{Q})$, $\gamma:I\to\mathbb{S}^2$ and $t_0\in I$, we show that $\kappa_1 < \kappa_\gamma^-(t_0)\leq\kappa_\gamma^+(t_0) < \kappa_2$. After a reparametrization, we suppose without loss of generality that, $t_0=0$ and $v(t)=\|\gamma'(t)\|\equiv 1$. Applying the Taylor's formula with integral remainder:
\begin{equation}\label{taylor}
\gamma(s) = \gamma(0) + \vet{t}_\gamma(0)\cdot s + R_1(s), \quad R_1(s)=\int_0^s \dot{\vet{t}}_\gamma(t)(s-t)\, dt.
\end{equation}
Given a $\rho\in (0,\pi)$, the center of the left tangent circle of radius $\rho$ is given by 
$$\vet{v}_1 = \gamma(0)\cos\rho + \vet{n}_\gamma(0)\sin\rho .$$
We consider the function
$$ g_1(s)=\langle\gamma(s),\vet{v}_1\rangle .$$
We shall verify that $s=0$ is a local minimum of $g_1$. Firstly, from the definition $g_1'(0)=\langle\vet{t}_\gamma(0),\vet{v}_1\rangle =0$. Now we give estimation on the second order variation at $s=0$:
$$\frac{g'_1(s)-g'_1(0)}{s} = \frac{1}{s}\left\langle\vet{t}_\gamma(s)-\vet{t}_\gamma(0),\vet{v}_1 \right\rangle .$$
From Equation \eqref{taylor},
$$ \gamma'(s)= \vet{t}_\gamma(s) = \vet{t}_\gamma(0)+R_1'(s). $$
\noindent Here we compute directly from the definition of the term $R_1'(s)$:
\begin{align*}
V(h) & = \int_0^{s+h}\dot{\vet{t}}_\gamma(t)(s+h-t)\,dt - \int_0^s \dot{\vet{t}}_\gamma(t)(s-t)\,dt \\
 & = \int_0^{s}\dot{\vet{t}}_\gamma(t)h \, dt + \int_s^{s+h}\dot{\vet{t}}_\gamma(t)h \, dt + \int_s^{s+h} \dot{\vet{t}}_\gamma(t)(s-t)\,dt .
\end{align*}
Dividing the above term by $h$ and taking the limit $h\to 0$ we obtain:
$$ R_1'(s) = \lim_{h\to 0} \frac{V(h)}{h} = \int_0^s\dot{\vet{t}}_\gamma(t)\, dt. $$
Since:
\begin{equation}\label{eqt}
\int_0^s \dot{\vet{t}}_\gamma(t)\, dt = \int_0^s\big(-\gamma(t)+w(t)\vet{n}_\gamma(t)\big)\, dt .
\end{equation}
For simplicity, denote $\gamma=\gamma(0)$, $\vet{t}_\gamma = \vet{t}_\gamma(0) $, $\vet{n}_\gamma=\vet{n}_\gamma(0)$ and $w=w(t)$. We will also use a form of big $O$ notation, this means $O(t)$ will be used to denote a map such that
$$ \limsup_{t\to 0}\frac{\|O(t)\|}{t} < \infty .$$
By definition
\begin{align*}
\vet{n}_\gamma(t) & = \gamma(t)\times\vet{t}_\gamma(t) \\
& = \big(\gamma+O(t)\big)\times\big(\vet{t}_\gamma+O(t)\big)\\
& = \vet{n}_\gamma + O(t).
\end{align*}
Substituting the expressions into Equation \eqref{eqt} we obtain:
\begin{align*}
\vet{t}_\gamma(s)-\vet{t}_\gamma(0) = R_1'(s) & = \int_0^s \big((-\gamma+O(t))+(w\vet{n}_\gamma+wO(t))\big)\, dt \\
& = \int_0^s \big(-\gamma+w\vet{n}_\gamma+O(t)\big)\,dt
\end{align*}
Thus
\begin{align}
\frac{g_1'(s)-g_1'(0)}{s} & = \frac{1}{s}\left\langle\vet{t}_\gamma(s)-\vet{t}_\gamma(0),\vet{v}_1\right\rangle \\
& = \frac{1}{s}\left\langle \int_0^s \big(-\gamma+w\vet{n}_\gamma+O(t)\big)\,dt , \gamma\cos\rho+\vet{n}_\gamma\sin\rho \right\rangle \\ \label{eqlast}
& = \frac{1}{s}\left(\int_0^s (w\sin\rho-\cos\rho)\,dt + \left\langle \int_0^s O(t)\, dt,\vet{v}_1 \right\rangle \right). 
\end{align}
Since $w(t)=h_{\kappa_1\kappa_2}^{-1}(\hat{w}(t))$ for some $\hat{w}(t)\in L^\infty(I)$. There exists a $\bar{\kappa}_2<\kappa_2$ such that $w(t)<\bar{\kappa}_2$ almost everywhere. On the other hand
$$\lim_{s\to 0}\frac{1}{s}\int_0^s O(t)dt = 0 .$$
Taking $\rho=\arccot\left(\frac{\kappa_2+\bar{\kappa}_2}{2}\right)$ in Equation \eqref{eqlast}, the first integral is a negative number. We conclude that:
$$\limsup_{s\to 0}\left(\frac{g_1'(s)-g_1'(0)}{s}\right) < 0  \quad \text{for the $\rho$ chosen above.}$$
Hence $s=0$ is a local minimum of $g_1$, this implies that it is a local maximum of $d(\gamma(s) ,\vet{v}_1)$. Because $d(\gamma(0),\vet{v}_1)=\rho$, the circle with the center at $\vet{v}_1$ radius $\rho$ is a left tangent circle of the curve $\gamma$ at $s=0$. Thus $\kappa_\gamma^+(0)\leq \frac{\kappa_2+\bar{\kappa}_2}{2} < \kappa_2 $. The procedure to show that $\kappa_1<\kappa_\gamma^-(0)$ is analogous. This finishes the proof of the first inclusion.
\end{proof}

\begin{proof}[Proof Part 2] Now we prove that $\mathcal{L}_{\kappa_1}^{\kappa_2}(\mat{P},\mat{Q})\subseteq\mathfrak{L}_{\kappa_1}^{\kappa_2}(\mat{P},\mat{Q})$. Given a $C^1$ curve $\gamma\in\mathcal{L}_{\kappa_1}^{\kappa_2}(\mat{P},\mat{Q})$, we assume without loss of generality that it is parametrized by arc-length $\|\gamma'(s)\|\equiv 1$ and we shall prove that $\vet{t}_\gamma$ is a Lipschitz continuous function. 

Let $s\in I$, by definition, $\kappa_\gamma^+(s)<\kappa_2$ and the upper-semicontinuity of $\kappa_\gamma^+$ (see Lemma \ref{semicontinuity}) implies that there exists a $\bar{\kappa}_2<\kappa_2$, $\rho = \arccot\bar{\kappa}_2$, $\vet{v}_1=\gamma(s)\cos\rho+\vet{n}_\gamma(s)\sin\rho$ such that the circle with center $\vet{v}_1$ and radius $\rho$ is tangent to $\gamma$ from the left at $s$. So $s$ is a local maximum point of the function $g_1(s)=\langle\gamma(s),\vet{v}_1\rangle$, from the part 1 of the demonstration we saw that $g_1'(s)=0$ and $g_1$ is a $C^1$ function and
$$ \limsup_{h\to 0}\frac{g_1'(s+h)-g_1'(s)}{h} \leq 0.$$
By the same calculations as in part 1 of the demonstration, we deduce from the above inequality:
\begin{equation}\label{eqlip1}
\limsup_{h\to 0}\left\langle \frac{\vet{t}_\gamma(s+h)-\vet{t}_\gamma(s)}{h},\vet{n}_\gamma(s) \right\rangle \leq \bar{\kappa}_2 < \kappa_2. 
\end{equation}
Analogously, from $\kappa_1<\kappa_\gamma^-(s)$ and the lower semi-continuity of $\kappa_\gamma^-$ we deduce that:
\begin{equation}\label{eqlip2}
\kappa_1 < \bar{\kappa}_1 \leq \liminf_{h\to 0}\left\langle \frac{\vet{t}_\gamma(s+h)-\vet{t}_\gamma(s)}{h},\vet{n}_\gamma(s) \right\rangle . 
\end{equation}
On the other hand, $\langle \vet{t}_\gamma(s),\vet{t}_\gamma(s) \rangle = \|\gamma'(s)\|^2 \equiv 1$ implies that:
\begin{equation}\label{eqlip3}
\lim_{h\to 0}\left\langle \frac{\vet{t}_\gamma(s+h)-\vet{t}_\gamma(s)}{h},\vet{t}_\gamma(s) \right\rangle = 0. 
\end{equation}
By differentiating the expression $\|\gamma(s)\|^2 \equiv 1 $ twice, we obtain
\begin{equation}\label{eqlip4}
\lim_{h\to 0}\left\langle\frac{\vet{t}_\gamma(s+h)-\vet{t}_\gamma(s)}{h},\gamma(s)\right\rangle = -1.
\end{equation}
Inequalities \eqref{eqlip1}, \eqref{eqlip2} and Equations \eqref{eqlip3}, \eqref{eqlip4} implies that $\vet{t}_\gamma(s)$ is a Lipschitz continuous function. Thus $\vet{t}_\gamma(s)$ is absolutely continuous, it has a derivative $\dot{\vet{t}}_\gamma(s)$ almost everywhere and
\begin{equation}
\vet{t}_\gamma(s+h)=\vet{t}_\gamma(s)+\int_s^{s+h}\dot{\vet{t}}_\gamma(t)\,dt .
\end{equation}
Again, it follows from \eqref{eqlip1}, \eqref{eqlip2}, \eqref{eqlip3} and \eqref{eqlip4} that there exists a function $w: I\to [\bar{\kappa}_1,\bar{\kappa}_2]\subset (\kappa_1,\kappa_2)$ such that
$$ \dot{\vet{t}}_\gamma(s)= -\gamma(s)+w(s)\vet{n}_\gamma(s) \quad \text{almost everwhere}$$
and $w(s)\in h_{\kappa_1,\kappa_2}(L^\infty(I)) $. This means that $\gamma$ is $(\kappa_1,\kappa_2)$-strongly admissible concluding that $\gamma\in \mathfrak{L}_{\kappa_1}^{\kappa_2}(\mat{P},\mat{Q})$.
\end{proof}

Since the main results in this article are about the spaces of form $\Lspace = \mathcal{L}_{-\kappa_0}^{+\kappa_0}(\mat{I},\mat{Q})$, with $\cot\rho_0 = \kappa_0<\infty$, the questions about the relationships between the spaces of curves for the cases $\kappa_1=-\infty$ and $\kappa_2=+\infty$ are open for future researches.

In spite of the spaces $\mathcal{L}_{\kappa_1}^{\kappa_2}(\mat{P},\mat{Q})$ and $\boldsymbol{L}_{\kappa_1}^{\kappa_2}(\mat{P},\mat{Q})$ being different, they are homotopically equivalent due to the following propositions. We recall the definition of $\mathcal{C}_{\kappa_1}^{\kappa_2}(\mat{P},\mat{Q})$ below

\begin{definition} Let $-\infty \leq \kappa_1 < \kappa_2 \leq +\infty$, consider the matrices $\mat{P},\mat{Q}\in\SO$ and $r\in\{2,3,\ldots,\infty\}$. Define $\mathcal{C}_{\kappa_1}^{\kappa_2}(\mat{P},\mat{Q})$ to be the set, furnished with the $C^r$ regular curves $\gamma:[0,1]\to\mathbb{S}^2$ such that:
\begin{enumerate}
\item $\Phi_\gamma(0)=\mat{P}$ and $\Phi_\gamma(1)=\mat{Q}$;
\item $\kappa_1<\kappa_\gamma(t)<\kappa_2$ for each $t\in [0,1]$.
\end{enumerate}
\end{definition}

\begin{lemma} Let $r\in\{2,3,\ldots,\infty\}$. Then the subset $\mathcal{C}_{\kappa_1}^{\kappa_2}(\mat{P},\mat{Q})$ is dense in $\mathcal{L}_{\kappa_1}^{\kappa_2}(\mat{P},\mat{Q})$.
\end{lemma}

\begin{proof}
It follows from the Lemma (1.8) in \cite{salzuh}.
\end{proof}

\begin{theorem}\label{thmequiv}
Let $-\infty<\kappa_1<\kappa_2<\infty$, consider the matrices $\mat{P},\mat{Q}\in\SO$ and $r\in\{2,3,\ldots,\infty\}$. Then the set inclusion $\boldsymbol{i}:\mathcal{C}_{\kappa_1}^{\kappa_2}(\mat{P},\mat{Q})\hookrightarrow\mathcal{L}_{\kappa_1}^{\kappa_2}(\mat{P},\mat{Q}) $ is a homotopy equivalence. Therefore, the sets $\mathcal{C}_{\kappa_1}^{\kappa_2}(\mat{P},\mat{Q})$ and $\mathcal{L}_{\kappa_1}^{\kappa_2}(\mat{P},\mat{Q})$ are homeomorphic.
\end{theorem}

\begin{proof} The proof is identical to the proof of Lemma (1.14) in \cite{salzuh}.
\end{proof}

\printindex


\newpage

\begin{thebibliography}{9}

\bibitem{alves}
E. Alves.
\textit{Topology of the space of locally convex curves on the 3-sphere}.
arXiv:1608.04635.

\bibitem{alves2}
E. Alves and N. C. Saldanha.
\textit{Results on the homotopy type of the spaces of locally convex curves on $S^3$}.
arXiv:1703.02581.

\bibitem{anisov}
S. Anisov.
\textit{Convex curves in $\mathbb{RP}^2$}.
Proc. Steklov Inst. Math 221 (1998), no. 2, 3-39.

\bibitem{arnold} 
V. I. Arnol'd. 
\textit{The geometry of spherical curves and the algebra of quaternions}. 
Russian Math. Surveys, 50:1 (1995), 1-68.

\bibitem{ayala3}
J. Ayala.
\textit{Length minimising bounded curvature paths in homotopy classes}.
Topology and its Applications. v.193:140-151, 2015.

\bibitem{ayala4}
J. Ayala.
\textit{On the topology of the spaces of curvature constrained plane curves}.
arXiv:1404.4378.

\bibitem{ayala5}
J. Ayala, D. Kirszenblat and H. Rubinstein.
\textit{A geometric approach to shortest bounded curvature paths}.
arXiv:1403.4899.

\bibitem{ayala2}
J. Ayala and H. Rubinstein.
\textit{Non-uniqueness of the homotopy class of bounded curvature paths}.
arXiv:1403.4911.

\bibitem{ayala}
J. Ayala and H. Rubinstein.
\textit{The classification of homotopy classes of bounded curvature paths}.
Israel J. Math. 213, 1(2016), 79-107.

\bibitem{birkh}
G. D. Birkhoff.
\textit{Dynamical Systems}.
AMS, 1927. p.135ff.


\bibitem{botttu}
R. Bott and L. W. Tu. 
\textit{Differential forms in algebraic topology}.
Springer Verlag, 1982.

\bibitem{dubins1}
L. E. Dubins.
\textit{On plane curves with curvature}.
Pacific J. Math, Vol.11, n.2 (1961), 471-481.

\bibitem{dubins2}
L. E. Dubins.
\textit{On curves of minimal length with a constraint on average curvature, and with prescribed initial and terminal positions and tangents}.
Amer. J. of Math, 79 (1957), 497-516.

\bibitem{elimis}
Y. Eliashberg and N. Mishachev,
\textit{Introduction to the h-principle}.
Graduate Studies in Mathematics, 48. American Mathematical Society, Providence, RI, 2002. xviii+206 pp.

\bibitem{goulart}
J. V. Goulart,
\textit{Towards a combinatoria apporoach to the topology spaces of nondegenerate spherical curves}.
PhD Thesis.

\bibitem{grom}
M. Gromov, 
\textit{Partial differential relations}. 
Bull. Amer. Math. Soc. (N.S.) 18 (1988), no. 2, 214--220.

\bibitem{shakhe}
B. A. Khesin and B. Z. Shapiro,
\textit{Homotopy classification of nondegenerate quasiperiodic curves on the 2-sphere}.
Publicatons de L'Institut Mathématique, Nouvelle série, tome 66 (80), 1999, 127-156.

\bibitem{little} 
J. A. Little,
\textit{Nondegenerate homotopies of curves on the unit 2-sphere}.
J. Differential Geometry, 4 (1970), 339-348.

\bibitem{milnor}
J. Milnor, based on Lecture notes by M. Spivak and R. Wells.
\textit{Morse theory}.
Princeton University Press (1963).

\bibitem{monroy}
F. Monroy-Pérez,
\textit{Non-Euclidean Dubins' Problem}.
Journal of Dynamical and Control Systems, Vol. 4, No. 2, 1998, 249-272.

\bibitem{mossad}
J. Mostovoy and R. Sadykov,
\textit{The space of non-degenerate closed curves in a Riemannian manifold}.
arXiv:1209.4109. 2012.

\bibitem{sald1}
N. C. Saldanha,
\textit{The cohomology of spaces of locally convex curves in the sphere -- I}.
arXiv:0905.2111.

\bibitem{sald2}
N. C. Saldanha,
\textit{The cohomology of spaces of locally convex curves in the sphere -- II}.
arXiv:0905.2116.
 
\bibitem{sald} 
N. C. Saldanha,
\textit{The homotopy type of spaces of locally convex curves in the sphere}.
Geom. Topol. 19 (2015), 1155-1203.

\bibitem{salsha}
N. C. Saldanha and B. Z. Shapiro,
\textit{Spaces of locally convex curves in $S^n$ and combinatorics of the group $B_{n+1}^+$}.
Journal of Singularities 4 (2012), 1-22.

\bibitem{shasha}
B. Z. Shapiro and M. Z. Shapiro,
\textit{On the number of connected components in the space of closed nondegenerate curves on $S^n$}.
Bulletin of the AMS 25 (1991), no. 1, 75-79.

\bibitem{saltom}
N. C. Saldanha and C. Tomei,
\textit{The topology of critical sets of some ordinary differential operators}.
Progress in Nonlinear Differential Equations and Their Applications, 66, 491-504 (2005).

\bibitem{salzuh}
N. C. Saldanha and P. Zühlke,
\textit{On the components of spaces of curves on the 2-sphere with geodesic curvature in a prescribed interval}.
Int. J. Math, vol. 24 (2013), no. 14, 1-78.

\bibitem{salzuh2}
N. C. Saldanha and P. Zühlke,
\textit{Homotopy type of spaces of curves with constrained curvature on flat surfaces}.
arXiv:1410.8590.

\bibitem{salzuh3}
N. C. Saldanha and P. Zühlke,
\textit{Components of spaces of curves with constrained curvature on flat surfaces}.
Pacific J. Math. 281 (2016), No. 1, 185-242.

\bibitem{salzuh4}
N. C. Saldanha and P. Zühlke,
\textit{Spaces of curves with constrained curvature on hyperbolic surfaces}.
arXiv 1611.09109.

\bibitem{smale}
S. Smale,
\textit{Regular curves on Riemannian manifolds}.
Trans. Amer. Math. Soc. 87 (1956), no. 2, 492-512.

\bibitem{shap}
M. Z. Shapiro,
\textit{Topology of the space of nondegenerate curves}.
Math. USSR 57 (1993), 106-126.

\bibitem{solo}
B. Solomon,
\textit{Tantrices of Spherical Curves}.
The American Mathematical Monthly, vol. 103 (1996), 30-39.

\bibitem{whit}
H. Whitney,
\textit{On regular closed curves in the plane}.
Compositio Mathematica, tome 4 (1937), no. 1, 276-284.


\end{thebibliography}
\end{document}